\documentclass{article}

\usepackage{arxiv}

\usepackage[utf8]{inputenc} 
\usepackage[T1]{fontenc}    
\usepackage{hyperref}       
\usepackage{url}            
\usepackage{booktabs}       
\usepackage{amsfonts}       
\usepackage{nicefrac}       
\usepackage{microtype}      
\usepackage{lipsum}
\usepackage{graphicx}
\graphicspath{ {./images/} }

\usepackage{amsthm}
\usepackage{fancyhdr}
\usepackage[numbers]{natbib}
\usepackage{amsmath}
\usepackage{mathtools}
\usepackage{amssymb}
\usepackage{slashed}
\usepackage{mathrsfs}

\newtheorem{theorem}{Theorem}[section]
\newtheorem{theorem*}{Theorem}

\newtheorem{lemma}[theorem]{Lemma}

\newtheorem{definition}[theorem]{Definition}

\newtheorem{notation*}[theorem*]{Notations}

\newtheorem{example*}[theorem*]{Example}
\newtheorem{remark}[theorem]{Remark}
\newtheorem{remark*}[theorem*]{Remark}

\newcommand*{\dif}{\mathop{}\!\mathrm{d}}

\title{Normalized Solutions to Schr\"{o}dinger Equations with Critical Exponent and Mixed Nonlocal Nonlinearities}

\author{
 Yanheng Ding \\
  Academy of Mathematics and System Science, Chinese Academy of Sciences, Bejing 100190, P.R. China\\
  University of Chinese Academy of Sciences, Bejing 100049, P.R. China\\
  \texttt{dingyh@math.ac.cn} \\
   \And
 Hua-Yang Wang\footnote{Corresponding author.} \\
  Academy of Mathematics and System Science, Chinese Academy of Sciences, Bejing 100190, P.R. China\\
  University of Chinese Academy of Sciences, Bejing 100049, P.R. China \\
  \texttt{wanghuayang@amss.ac.cn} \\
}

\begin{document}
\maketitle
\begin{abstract}
 We study the existence and nonexistence of normalized solutions $(u_a, \lambda_a)\in H^{1}(\mathbb{R}^N)\times \mathbb{R}$ to the nonlinear Schr\"{o}dinger equation with mixed nonlocal nonlinearities:
	\begin{align*}
		\begin{cases}
			-\Delta u=\lambda u+ (I_{\alpha} * \vert u \vert^p)\vert u \vert^{p-2}u+\mu (I_{\alpha} * \vert u \vert^q)\vert u \vert^{q-2}u\quad \text{in }\mathbb{R}^N,\\
			\int_{\mathbb{R}^{N}} \vert u \vert^2 \dif x=a^2>0,
		\end{cases}
	\end{align*}
	where $N\geq 3$,  $\alpha\in (0,N)$, $\mu\in\mathbb{R}$, $\frac{N+\alpha}{N}< q< p \leq \frac{N+\alpha}{N-2}$,
	and $I_{\alpha}$ is the Riesz potential.
	This study can be viewed as a counterpart of the Brezis-Nirenberg problem in the context of normalized solutions to the nonlocal Schr\"{o}diger equation with a fixed $L^2$-norm $\|u\|_2=a>0$. 
	The leading term is $L^2$-supercritical, that is, $p\in (\frac{N+\alpha+2}{N},\frac{N+\alpha}{N-2}]$, where the Hardy-Littlewood-Sobolev critical exponent $p=\frac{N+\alpha}{N-2}$ appears.
	We first prove that there exist two normalized solutions if $q\in (\frac{N+\alpha}{N},\frac{N+\alpha+2}{N})$ with $\mu >0$ small, that is, one is at the negative energy level while the other one is at the positive energy level.
	For $q=\frac{N+\alpha+2}{N}$, we show that there is a normalized ground state for $0<\mu < \tilde{\mu} $ and there exist no ground states for $\mu >\tilde{\mu}$, where $\tilde{\mu}$ is a sharp positive constant.
	If $q\in (\frac{N+\alpha+2}{N},\frac{N+\alpha}{N-2})$, we deduce that there exists a normalized ground state for any $\mu>0$.
	We also obtain some existence and nonexistence results for the case $\mu<0$ and $q\in (\frac{N+\alpha}{N},\frac{N+\alpha+2}{N}]$.
	Besides, we analyze the asymptotic behavior of normalized ground states as $\mu\rightarrow 0^{+}$.
\end{abstract}



	\section{Introduction}
	In this paper, we are interested in the existence of standing waves with a given $L^2$-norm for the Schr\"{o}dinger equation with mixed nonlocal nonlinearities:
	\begin{equation}\label{eq:NLS}
		i\partial_{t}\psi+\Delta \psi + (I_{\alpha} * \vert \psi \vert^p)\vert \psi \vert^{p-2}\psi+\mu (I_{\alpha} * \vert \psi \vert^q)\vert \psi \vert^{q-2}\psi=0\quad \text{in }\mathbb{R}\times\mathbb{R}^{N},
	\end{equation}
	where $N\geq 3$, $\psi:\mathbb{R}\times \mathbb{R}^N\rightarrow \mathbb{C}$, $\mu\in\mathbb{R}$, $\frac{N+\alpha}{N}=:2_{\alpha}< q< p \leq 2^{*}_{\alpha}:=\frac{N+\alpha}{N-2}$ 
	 and $I_{\alpha}$ is the Riesz potential of order $\alpha\in (0,N)$ defined by
	\begin{equation*}
		I_{\alpha}=\frac{\mathcal{A}(N,\alpha)}{\vert x \vert ^ {N-\alpha}}\quad \text{with}\quad \mathcal{A}(N,\alpha)=\frac{\Gamma(\frac{N-\alpha}{2})}{\pi^{\frac{N}{2}}2^{\alpha}\Gamma(\frac{\alpha}{2})}\text{ for each } x\in\mathbb{R}^N\setminus\{0\}.
	\end{equation*}
	We first recall that the standing waves of \eqref{eq:NLS}, namely to solutions of the form $\psi(t,x)=e^{-i\lambda t}u(x)$, where $\lambda\in\mathbb{R}$ and $u:\mathbb{R}^N\rightarrow \mathbb{R}$.
	Then the function $u(x)$ satisfies the following equation:
	\begin{equation}\label{eq:NLS_S}
		-\Delta u=\lambda u+ (I_{\alpha} * \vert u \vert^p)\vert u \vert^{p-2}u+\mu (I_{\alpha} * \vert u \vert^q)\vert u \vert^{q-2}u\quad \text{in }\mathbb{R}^N,\ N\geq 3.
	\end{equation}
	
	Physically speaking, Schr\"{o}dinger equation appears naturally in many fields.
	For $N=3$, $\mu=0$, $\alpha=2$ and $p=2$, \eqref{eq:NLS_S} is called the Choquard-Pekar equation.
	It seems to originate from Fr\"{o}hlich and Pekar’s model of the polaron, where free electrons in an
	ionic lattice interact with photons associated with deformations of the lattice
	or with the polarization that it creates on the medium \cite{frohlich1937theory,Pekar+1954}.
	The Choquard-Pekar equation is also introduced by Choquard in the modeling of a one-component
	plasma \cite{MR471785}, which is a certain approximation to the Hartree-Fock theory.
	When dealing with the Schr\"{o}dinger equation of quantum physics together with nonrelativistic Newtonian gravity, the Choquard-Pekar equation is known as the Schr\"{o}dinger–Newton equation \cite{MR1386305}.
	In a few words, the nonlocal Schnr\"{o}diner equation plays an important role in the study of Physics.
	
	We shall focus more on the mathematical aspects of the Schr\"{o}dinger equation \eqref{eq:NLS_S}.
	On the one hand, when looking for solutions to \eqref{eq:NLS_S}, a possible way is to study the \emph{fixed frequency problem}.
	Precisely, we could fix $\lambda\in\mathbb{R}$ and search for solutions as critical points of the following functional defined on $H^{1}(\mathbb{R}^N)$:
	\begin{align*}
		\mathcal{I}_{\lambda,\mu}[u]:=\frac{1}{2}\int_{\mathbb{R}^N}\vert \nabla u \vert^2\dif x-\frac{\lambda}{2}\int_{\mathbb{R}^N}\vert  u \vert^2\dif x-\frac{\mu}{2q}\int_{\mathbb{R}^N}(I_{\alpha} * \vert u \vert^q)\vert u \vert^{q}\dif x-\frac{1}{2p}\int_{\mathbb{R}^N}(I_{\alpha} * \vert u \vert^p)\vert u \vert^{p} \dif x.
	\end{align*}
	For $N=3$, $\mu=0$, $\alpha=2$ and $p=2$, the existence of solutions to the Choquard-Pekar equation is proved via variational methods by Lieb \cite{MR471785}, Lions \cite{MR591299} and Menzala \cite{MR592556}.
	For the more general homogeneous case $N\geq 3$, $\mu=0$ and $p\in (2_{\alpha},2_{\alpha}^*)$, the existence of ground states can be founded in Ma and Zhao \cite{MR2592284}, Genev and Venkov \cite{MR2877355} as well as Moroz and Van Schaftingen \cite{MOROZ2013153}.
	Later, Moroz and Van Schaftingen \cite{MR3356947} prove the existence of a nontrivial solution with nonhomogeneous nonlinearities, including $\mu\ne 0$ and $2_{\alpha}<q<p<2_{\alpha}^{*}$.
	Gao and Yang investigate the existence and multiplicity of solutions to \eqref{eq:NLS_S} for $p=2_{\alpha}^{*}$ on a smooth bounded domain $\Omega \subset \mathbb{R}^N,$ $N\geq 3$ in \cite{MR3582271}.
	The fixed frequency problem has been widely investigated for decades. 
	It seems impossible to summarize it here since the related litterateurs are abundant.
	We advise readers to read a survey paper \cite{MR3625092} and its references.
	
	On the other hand, we can study the \emph{normalized solution problem}.
	Specifically, we can search for solutions to \eqref{eq:NLS_S} possessing a given $L^2$-norm, that is, finding $(\lambda_a,u_a)\in\mathbb{R}\times H^{1}(\mathbb{R}^N)$ solving $\eqref{eq:NLS_S}$ together with the normalized condition
	\begin{equation*}
		\|u_a\|_2^2=\int_{\mathbb{R}^N}\vert u_a(x) \vert^2  \dif x =a^2,\quad a>0.
	\end{equation*}
	Here $u_a$ is called the \emph{normalized solution}, $\lambda_a$ arises as an unknown \emph{Lagrange multiplier} which depends on the solution $u_a$.
	It is standard to show that the following energy functional is of class $\mathscr{C}^1$:
	\begin{align}
		E[u]=&\frac{1}{2}\int_{\mathbb{R}^N}\vert \nabla u \vert^2\dif x-\frac{1}{2p}\int_{\mathbb{R}^N}(I_{\alpha} * \vert u \vert^p)\vert u \vert^{p}\dif x -\frac{\mu}{2q}\int_{\mathbb{R}^N}(I_{\alpha} * \vert u \vert^q)\vert u \vert^{q}\dif x.
	\end{align}
	Indeed, normalized solutions to \eqref{eq:NLS_S} can be obtained as critical points of the energy functional $E$ under the constraint
	\begin{equation}\label{eq:S(a)}
		S(a)=\{u\in H^{1}(\mathbb{R}^N):\|u\|_2=a>0\}.
	\end{equation}
	For future's convenience, we recall the definition of a \emph{ground state} to \eqref{eq:NLS_S} on $S(a)$:
	\begin{definition}
		 We say that $u_a$ is a \textbf{ground state} of \eqref{eq:NLS_S} on $S(a)$ if $(\lambda_a,u_a)\in \mathbb{R}\times S(a)$ is a solution to \eqref{eq:NLS_S} and $u_a$ has the minimal energy among all the solutions which belong to $S(a)$, that is
		$$
		\left.\dif E \right|_{S(a)}[u_a]=0 \quad \text { and } \quad E(u_a)=\inf_{u\in S(a)} \{E[u]:\left.\dif E \right|_{S(a)}[u_a]=0\}.
		$$
	\end{definition}
	Before recalling some related works, we consider a general nonlinear Schr\"{o}dinger equation:
	\begin{align}\label{eq:NLS_g}
		-\Delta u =\lambda u + g(u),\quad \text{in }\mathbb{R}^N,\ N\geq 1,
	\end{align}
	where $g(u)$ could be a local or nonlocal nonlinearity and define 
	\begin{align}\label{functional:NLS_g}
		J[u]:=\frac{1}{2}\int_{\mathbb{R}^N}\vert \nabla u \vert^2 \dif x - \int_{\mathbb{R}^N} G(u(x))\dif x,\quad  \text{and}\quad m(a):=\inf_{u\in S(a)} J[u],
	\end{align}
	where $G(u)$ is the primitive function of $g(u)$.
	We find that $m(a)$ changes if we choose different conditions of $g(u)$, and in particular, it could be $-\infty$.
	The case that $m(a)$ is finite mainly starts at the work of Stuart \cite{MR587907} via a bifurcation approach with a nonhomogeneous local nonlinearity.
	Later, in the famous works about the Compactness by Concentration Principle of Lions \cite{MR778970,MR778974}, the case that $m(a)$ is finite has been studied systemically, where $g(u)$ may be a local or nonlocal nonlinearity.
	Under some assumptions of $g(u)$, Lions gets the existence of normalized solutions to \eqref{eq:NLS_g} via finding the global minimum of the functional $J[u]$ with the constraint \eqref{eq:S(a)}.
	However, before the work of Jeanjean \cite{MR1430506}, there is no suitable way to deal with the case $m(a)=-\infty$.
	It seems that this direction of research has received attention due to the papers \cite{MR3009665, MR3092340}.
	Recently, many people are committed to the normalized solution problem, see \cite{MR3539467,MR3639521,MR4297197,MR4232669,MR4188316,MR4021261,MR4150876,MR3369266} and their references.
	We also refer to \cite{MR3892400,Guo,MR4304693,MR4445673,MR4452342} for the equation with a linear potential and to \cite{MR3918087,MR4191345,MR3689156} for contributions when the equation considered on a bounded domain $\Omega$ of $\mathbb{R}^N$.
	
	In the above mentioned papers, whether $g(u)$ is local or nonlocal in \eqref{eq:NLS_g}, the involved nonlinearities are subcritical.
	We shall introduce some results about normalized solutions to Schr\"{o}dinger equations with critical exponent and mixed nonlinearities.
	
	\begin{itemize}
		\item Soave first studies the normalized solutions to the following Schr\"{o}dinger equation with Sobolev critical exponent and mixed local nonlinearities in \cite{SOAVE20206941,MR4096725}:
		\begin{align}\label{eq:NLS_l}
			-\Delta u=\lambda u+\mu\vert u \vert^{q-2}u+\vert u \vert^{p-2}u,\quad \text{in }\mathbb{R}^N,\ N\geq 1,
		\end{align}
		where $\mu \in \mathbb{R}$, $2<q\leq 2^{\sharp}:=2+\frac{4}{N}\leq p \leq 2^{*}$, $q\ne p$.
		When the leading term $\vert u \vert^{p-2}u$ is $L^2$-supercritical, it determines that the functional restricted to $S(a)$ is unbounded form below.
		He introduces a Pohozaev manifold and considers the energy functional with this natural constraint.
		The perturbation term $\mu\vert u \vert^{q-2}u$ influences the structure of Pohozaev manifold, where the range of exponent $q$ and the number $\mu$ play essential roles.
		Roughly speaking, Soave proves that if the perturbation term is small (i.e., $\vert \mu\vert>0$ is small), there exists at least one normalized radial ground state to \eqref{eq:NLS_l}.
		In particular, if $p\in (2^{\sharp},2^{*}]$, $q\in (2, 2^{*})$ and $\mu>0$ small, he showed that there are two radial positive solutions.
		Later, for $p=2^{*}$, $q\in (2^{\sharp}, 2^{*})$ and $\mu>0$, Jeanjean and Le \cite{2021Jeanjean} prove that the perturbation term can be large enough (i.e, we only need $\mu>0$) with $N\geq 4$.
		Wei and Wu \cite{MR4433054} solve this problem for $N=3$.
		Many authors pay attention to related problems, for example, see \cite{MR4350192, MR4426879, MR4290382, MR4135640} and their references.
		For authors' knowledge, it seems that the existence of nonradial normalized solutions and the multiplicity of normalized solutions to \eqref{eq:NLS_l} are open problems. 
		
		\item Recently, Yao, Chen, R\v{a}dulescu and Sun consider normalized solutions to the following Schr\"{o}dinger equation with one local and one nonlocal nonlinearities in 
		\cite{MR4443663}:
		\begin{align}\label{eq:NLS_SPS}
			-\Delta u+\lambda u= \gamma(I_{\alpha}*\vert u \vert^p)\vert u \vert^{p-2}u+\mu \vert u \vert^{q-2}u,\quad \text{in }\mathbb{R}^N,\ N\geq 3,
		\end{align}
		where $\gamma,\mu >0$, $p=2_{\alpha}$ and $2<q\leq2^{*}$.
		Bhattarai has studied a similar question for fractional Schr\"{o}dinger equation in \cite{MR3659360}.
		For $N=3$ and $p=2$, \eqref{eq:NLS_SPS} is called the Schr\"{o}dinger-Possion-Slater equation. 
		Jeanjean and Le consider the normalized solutions to this system in \cite{JEANJEAN2021277}, where $q\in (\frac{10}{3},6]$ involves the Sobolev critical case, see also \cite{MR3092340}.
		Generally speaking, the two nonlinearities are competing for different $\gamma$, $\mu$ and influence the structure of energy functional in the above works, where the $L^2$-critical exponents and Sobolev critical exponent play an important role.
		\item A natural goal is to investigate the normalized solutions to the Schr\"{o}dinger equation \eqref{eq:NLS_S} with mixed nonlocal  nonlinearities. 
		Cao, Jia and Luo \cite{MR4193068} study the normalized solution to the following equation:
		\begin{align*}
			-\Delta u=\lambda u+\mu\left(|x|^{-\alpha} *|u|^2\right) u+\left(|x|^{-\beta} *|u|^2\right) u,\quad \text{in }\mathbb{R}^N,\ N\geq 1,
		\end{align*}
		where $\alpha$ and $\beta$ have an influence on the existence of solutions, see also a similar result by Bhimani, Gou and Hajaiej \cite{Gou}.
		However, there exists few works about our equation \eqref{eq:NLS_S}, which means, for fixed $N\geq 3$ and $\alpha \in (0,N)$, studying the existence and multiplicity of solutions with a $L^2$-supercritical (may be HLS critical) leading term and different perturbations described by $p,q$ and $\mu$.
		Motivated by above works, we believe following problems are meaningful:
		\begin{enumerate}
			\item[\textbf{(Q1)}] \textbf{Can we find some sufficient conditions described by $\mu$ and $a$ such that there exists a normalized ground state to \eqref{eq:NLS_S}?}
			\item[\textbf{(Q2)}] \textbf{Can we find some sufficient conditions independent of $\mu$ and $a$ such that there exists a normalized ground state to \eqref{eq:NLS_S}? If not, what is the sufficient condition dependent on $\mu$ and $a$ such that there exists no normalized ground states to \eqref{eq:NLS_S}?}
			\item[\textbf{(Q3)}] \textbf{Can we show some multiplicity results about normalized solutions to  \eqref{eq:NLS_S}?}
			\item[\textbf{(Q4)}] \textbf{For $\mu>0$, if the equivalent condition of the existence of at least one normalized ground state to \eqref{eq:NLS_S} is  $\mu \in (0,\bar{\mu})$ with $\bar{\mu}\in (0,+\infty]$, can we describe the asymptotic behavior of the normalized ground state as $\mu \rightarrow 0^{+}$ and $\mu \rightarrow \bar{\mu}$?}
		\end{enumerate}
	\end{itemize}
	
	To describe our main results, we define
	$$
	A[u]:=\int_{\mathbb{R}^N}\vert \nabla u\vert^2\dif x,\quad B_s[u]:=\int_{\mathbb{R}^N}(I_{\alpha} * \vert u \vert^s)\vert u \vert^{s}\dif x,
	$$
	then we have
	\begin{equation*}
		E[u]=\frac{1}{2}A[u]-\frac{1}{2p}B_p[u]-\frac{\mu}{2q}B_q[u].
	\end{equation*}
	Setting $(t \star u)(x)=t^{\zeta}u(t^{\theta}x)$, $t>0$,  it holds that
	\begin{align*}
		\|t \star u \|_r^r= t^{\zeta r-\theta N}\|u\|_r^r,\quad A[t \star u]=t^{2\theta +2\zeta -\theta N}A[u],\quad
		B_r[t \star u]=t^{(2\zeta-\theta N)r+2r\eta_{r}\theta}B_r[u],
	\end{align*}
	where $\eta_{r}=\frac{Nr-N-\alpha}{2r}$. 
	The following inequality appears frequently in this paper:
	\begin{align*}
		0<2r\eta_{r}=Nr-N-\alpha
		\begin{cases}
			<2,&\text{if }r\in [2_{\alpha},2_{\alpha}^{\sharp}),\\
			=2,&\text{if }r=2_{\alpha}^{\sharp},\\
			>2,&\text{if }r\in (2_{\alpha}^{\sharp},2_{\alpha}^{*}].
		\end{cases}
	\end{align*}
	If $\|t \star u\|_2=\|u\|_2$, then it holds that $2\zeta-\theta N=0$. 
	In particular, if we choose $\theta=\frac{1}{2}$ with $2\zeta -\theta N=0$ and put $s\circ u(x)=s^{\frac{N}{4}}u(s^{\frac{1}{2}}x)$, then it follows that
	\begin{align*}
		\|s\circ u\|_2=\|u\|_2,\quad A[s\circ u]=sA[u],\quad B_r[s\circ u]=s^{r\eta_{r}}B_r[u].
	\end{align*}
	For $u\in H^1(\mathbb{R}^N)\setminus \{0\}$ fixed, we define a function $f_u:[0,+\infty)\rightarrow \mathbb{R}$ by
	\begin{equation}\label{eq:scaling_function}
		f_u(s):=2E[s\circ u]=sA[u]-\frac{\mu}{q}s^{q\eta_q}B_q[u]-\frac{1}{p}s^{p\eta_{p}}B_p[u].
	\end{equation}
	It is clear that
	\begin{equation}\label{eq:scaling_function_1st_derivative}
		f^{\prime}_u(s)=A[u]-\eta_{p}s^{p\eta_{p}-1}B_p[u]-\mu\eta_{q} s^{q\eta_{q}-1}B_q[u],
	\end{equation}
	and 
	\begin{equation}\label{eq:scaling_function_2nd_derivative}
		f^{\prime\prime}_u(s)=-\eta_{p}(p\eta_{p}-1)s^{p\eta_{p}-2}B_p[u]+\mu\eta_{q}(1-q\eta_{q}) s^{q\eta_{q}-2}B_q[u].
	\end{equation}
	By a similar proof in \cite[Proposition 3.1]{MOROZ2013153} (see also \cite[Theorem 2.1]{MR4106818}), it follows that if $u\in H^1(\mathbb{R}^N)$ is a solution to \eqref{eq:NLS_S}, then the following Pohozaev type identity holds:
	\begin{align}\label{eq:Pohozaev}
		\frac{N-2}{2}A[u]-\frac{N}{2}\lambda \|u\|_2^2-\frac{N+\alpha}{2p}B_p[u]-\frac{N+\alpha}{2q}\mu B_q[u]=0.
	\end{align} 
	Testing equation \eqref{eq:NLS_S} by $u$, we obtain that
	\begin{align}\label{eq:testing_1}
		A[u]-\lambda \|u\|_2^2-B_p[u]-\mu B_q[u]=0.
	\end{align}
	Combing \eqref{eq:Pohozaev} and $\eqref{eq:testing_1}$, it follows that $P[u]=0$, where
	\begin{align*}
		P[u]=A[u]-\eta_{p}B_p[u]-\mu \eta_{q}B_q[u].
	\end{align*}
	It is clear that $P[s\circ u]=f_u^{\prime}(s)$.
	Then we can define
	\begin{align*}
		\mathcal{P}(a):=&\{u\in S(a): f'_u(1)=0\}=\{u\in S(a): P[u]=0\},\\
		\mathcal{P}^{+}(a):=&\{u\in S(a): f'_u(1)=0,\ f^{\prime\prime}_u(1)>0\},\\
		\mathcal{P}^{0}(a):=&\{u\in S(a): f'_u(1)=0,\ f^{\prime\prime}_u(1)=0\},\\
		\mathcal{P}^{-}(a):=&\{u\in S(a): f'_u(1)=0,\ f^{\prime\prime}_u(1)<0\},
	\end{align*}
	and hence we can divide $\mathcal{P}(a)$ into the disjoint union  $\mathcal{P}(a)=\mathcal{P}^{+}(a)\cup\mathcal{P}^{0}(a)\cup\mathcal{P}^{-}(a)$.
	We define 
	\begin{align*}
		\gamma(a,\mu):=\inf_{u \in \mathcal{P}(a)}E[u] \quad\text{and}\quad \gamma^{\pm}(a,\mu):=\inf_{u \in \mathcal{P}^{\pm}(a)}E[u],
	\end{align*}
	where $\mathcal{P}^{\pm}(a)$ denotes either $\mathcal{P}^{+}(a)$ or $\mathcal{P}^{-}(a)$, and we use similar notations without explanation in the following part.
	Besides, we usually ignore the parameter $\mu$ of $\gamma(a,\mu)$.
	
	We shall give partial answers to \textbf{(Q1)-(Q4)} via variational methods.
	The first difficulty is to find the existence of a Palais-Smale sequence at suitable level.
	Motivated by above papers, since $E$ restricted to $S(a)$ is unbounded from below, we shall consider the functional $E$ with a natural constraint $\mathcal{P}(a)$.
	The different perturbations will have a huge influence on the structure of the manifold $\mathcal{P}(a)$, and we shall handle them with more detailed analysis separately.
	For example, under suitable assumptions, we will find that $\mathcal{P}^{+}(a)=\emptyset$ for $\mu >0$ and $q\in [2_{\alpha}^{\sharp},p)$ and $\mathcal{P}^{+}(a)\ne \emptyset$ for $\mu>0$ and $q\in (2_{\alpha},2_{\alpha}^{\sharp})$.
	However, since the functional $E$ restricted to $\mathcal{P}(a)$ is always coercive and bounded from below (see Lemma \ref{lm:coercive}), we can deal with such a case.
	Using a critical theorem in \cite{1993Duality} and some methods developed in \cite{MR3976588,MR4150876}, we can construct a bounded Palais-Smale sequence $\{u_n\}$ on $\mathcal{P}(a)$.
	Then we shall deal with the second obstacle, that is, to recover the compactness.
	If $p\in (2_{\alpha}^{\sharp},2_{\alpha}^{*})$, we can do this with the help of an extra condition $P[u_n]=0$.
	But the HLS critical case $p=2_{\alpha}^{*}$ is much more difficult than the HLS subcritical case.
	To deal with the the HLS critical case $p=2_{\alpha}^{*}$, motivated by \cite{MR4096725}, we prove an alternative lemma, see Lemma \ref{lm:PS_strong_convergence}.
	Then we use various methods to estimate the critical value for different perturbations.
	Especially, the dimension $N$ and the parameter $\alpha$ will have an impact on the estimate, which is a new difficulty compared with the local case, for example, see Lemma \ref{lm:energy_critical_sub}.
	It is valuable to point out that we find two methods to deal with the HLS critical leading term: the first one is purely variational and the another one is using the HLS subcritical approximation.
	We believe that the second method may be useful for other critical problems.
	
	Suppose that the following inequality holds:
	\begin{align}\label{assumptions:basic}
		&(\vert \mu \vert a^{2q(1-\eta_{q})})^{p\eta_{p}-1}(a^{2p(1-\eta_{p})})^{1-q\eta_{q}}\notag\\
		&<\Xi (N,\alpha,p,q):=
		\begin{cases}
			\left(\frac{1-q\eta_{q}}{\eta_{p}\mathcal{C}_{G}(p)(p\eta_{p}-q\eta_{q})}\right)^{1-q\eta_{q}}\left(\frac{p\eta_{p}-1}{\eta_{q}\mathcal{C}_{G}(q)(p\eta_{p}-q\eta_{q})}\right)^{p\eta_{p}-1}, &\text{if}\ \mu>0\ \text{and}\ 2_{\alpha}<q\leq 2_{\alpha}^{\sharp}<p<2_{\alpha}^{*},\\
			\left(\frac{1}{\eta_{q}\mathcal{C}_{G}(q)}\right)^{1-q\eta_{q}}
			\left(\frac{1-\eta_{p}}{(\eta_{p}-\eta_{q})\mathcal{C}_{G}(p)}\right)^{p\eta_{p}-1},&\text{if}\ \mu\leq 0\ \text{and}\ 2_{\alpha}<q\leq 2_{\alpha}^{\sharp}<p<2_{\alpha}^{*},\\
			\left(\frac{1-q\eta_{q}}{\mathcal{S}_H^{-p}(p-q\eta_{q})}\right)^{1-q\eta_{q}}\left(\frac{p-1}{\eta_{q}\mathcal{C}_{G}(q)(p-q\eta_{q})}\right)^{p-1}, &\text{if}\ \mu>0\ \text{and}\ 2_{\alpha}< q\leq 2_{\alpha}^{\sharp}<p=2_{\alpha}^{*},\\
			+\infty, &\text{if}\ \mu>0\ \text{and}\ 2_{\alpha}^{\sharp}<q<p\leq 2_{\alpha}^{*},
		\end{cases}
	\end{align}
	where the constants $\mathcal{C}_G(p)$, $\mathcal{S}_H$ and $\eta_{p}$ can be founded in \eqref{ieq:best_constant} and \eqref{ieq:GN_H} at section 2.
	It is time to introduce our main results in this paper.
	The first theorem gives a partial answer to \textbf{(Q1)} and \textbf{(Q3)} for the focusing $L^2$-subcritical perturbation. 
	\begin{theorem}\label{thm:mass_subcritical}
		Let $N\geq 3$, $\alpha \in (0,N)$, $2_{\alpha}<q<2_{\alpha}^{\sharp}<p\leq2_{\alpha}^{*}$ and $\mu,a>0$.
		Suppose that \eqref{assumptions:basic} holds.
		Then 
		\begin{enumerate}
			\item[(i)] there exists a pair $(u^{+}_a,\lambda_a^{+})\in H^{1}(\mathbb{R}^N)\times\mathbb{R}^{-}$  such that $u_a^+$ is a ground state to \eqref{eq:NLS_S} on $S(a)$ at level $\gamma^{+}(a)<0$; 
			\item[(ii)] there exists another pair $(u^{-}_a,\lambda_a^{-})\in H^{1}(\mathbb{R}^N)\times\mathbb{R}^{-}$  such that $u_a^-$ is a normalized solution to \eqref{eq:NLS_S} at level $\gamma^{-}(a)>0$ if 
			either $p\in (2_{\alpha}^{\sharp},2_{\alpha}^{*})$ holds or $p=2_{\alpha}^{*}$ with one of the following holds:
			\begin{itemize}
				\item $N\geq 6$, $0<\alpha<N-2$ and $\max\{\frac{2N-2+\alpha}{2N-4},\frac{N+\alpha-2}{N-2}\}<q<2_{\alpha}^{\sharp}$;
				\item $N=5$, $0<\alpha<3$ and $\max\{\frac{4\alpha}{3},\frac{7+2\alpha}{6}\}<q<2_{\alpha}^{\sharp}$;
				\item $3\leq N\leq 5$, $\max\{N-2, 2N-6\}\leq\alpha<N$ and $2\leq q<2_{\alpha}^{\sharp}$.
			\end{itemize}
		\end{enumerate}
	\end{theorem}
	
	\begin{remark} 
		Although we summarize the results for $p\in (2_{\alpha}^{\sharp},2_{\alpha}^{*})$ and $p=2_{\alpha}^{*}$ in one theorem, the proof of two cases are not the same. 
		Precisely, if $p=2_{\alpha}^{*}$, it is hard to recover the compactness, especially for the Palais-Smale sequence at $\gamma^{-}(a)$.
		The main difficulty is to estimate the range of critical value.
		In fact, similar question arises in the study of \eqref{eq:NLS_l}, and the method developed in \cite{MR4433054} is efficient, see also \cite{2021Jeanjean}.
		However, since our nonlinearities are nonlocal and sometimes $q<2$, we shall find some new estimates for our problem, see Lemma \ref{lm:energy_critical_sub}.
	\end{remark}
	Under the assumptions in Theorem \ref{thm:mass_subcritical}, we can describe the ground state $u^{+}(a)$ as a local minima for $E$ restricted to $S(a)$. 
	There exists an $a_1:=a_1(p)>0$ such that for $\mu>0$ and $2_{\alpha}<q<2_{\alpha}^{\sharp}<p\leq2_{\alpha}^{*}$, it holds that 
	\begin{align}\label{assumptiobs_basic_equality}
		(\mu a_1^{2q(1-\eta_{q})})^{p\eta_{p}-1}(a_1^{2p(1-\eta_{p})})^{1-q\eta_{q}}=\Xi(N,\alpha,p,q).
	\end{align}
	More precisely, for $\mu>0$ fixed, we have 
	\begin{theorem}\label{thm:minizing}
		Let $N\geq 3$, $\alpha \in (0,N)$, $2_{\alpha}<q<2_{\alpha}^{\sharp}<p\leq2_{\alpha}^{*}$, $\mu>0$ fixed and $a\in (0,a_1)$, where $a_1$ is defined in \eqref{assumptiobs_basic_equality}.
		Then we have $\mathcal{P}^{+}(a)\subset V(a)$ and 
		\begin{align*}
			\gamma^{+}(a)=\inf\limits_{u \in \mathcal{P}^{+}(a)}E[u]=\inf\limits_{u \in \overline{V(a)}}E[u],
		\end{align*}
		where $V(a):=\{u\in S(a):A[u]<k_1\}$ for some 
		\begin{align}
			k_1:=\left( \frac{\mu\eta_{q}\mathcal{C}_{G}(q)(p\eta_{p}-q\eta_{q})}{p\eta_{p}-1}\right)^{\frac{1}{1-q\eta_{q}}}a_1^{\frac{2q(1-\eta_{q})}{1-q\eta_{q}}}>0.
		\end{align}
		In addition, and minimizing sequence for $E$ on $V(a)$ is, up to a subsequence and translation, strongly convergent in $H^1(\mathbb{R}^N)$. 
	\end{theorem}
	\begin{remark}
		In \cite{SOAVE20206941,MR4096725}, Soave obtain the ground state to \eqref{eq:NLS_l} with $2<q<2^{\sharp}<p\leq 2^*$ by a similar description of Theorem \ref{thm:minizing}. 
		Our proof of Theorem \ref{thm:mass_subcritical}(i) is different form above methods, where we do not use the description of the ground state $u^{+}(a)$ as a local minima for $E$ restricted to $S(a)$.
		However, Theorem \ref{thm:minizing} shows we can also get similar results and it is useful for proving Theorem \ref{thm:asymptotic_mu_0}.
	\end{remark}
	The next theorem gives an almost complete answer to \textbf{(Q2)} for the focusing $L^2$-critical perturbation, except for the case $N=3$, $p=2_{\alpha}^{*}$ with $\alpha \in [1,3)$.
	\begin{theorem}\label{thm:mass_critical}
		Let $N\geq 3$, $\alpha \in (0,N)$, $q=2_{\alpha}^{\sharp}<p\leq2_{\alpha}^{*}$ and $\mu,a>0$.
		For $N=3$ and $p=2_{\alpha}^{*}$, we additionally assume that $\alpha \in (0,1)$.
		Then
		\begin{enumerate}
			\item[(i)] 
			if \eqref{assumptions:basic} holds, that is, $0<\mu a^{2q-2}<\eta_{q}^{-1} \mathcal{C}_{G}(q)^{-1},$ then there exists a pair $(u_a,\lambda_a)\in H^{1}(\mathbb{R}^N)\times\mathbb{R}^{-}$  such that $u_a$ is a normalized solution to \eqref{eq:NLS_S} at level $\gamma(a)>0$;
			\item[(ii)] 
			if \eqref{assumptions:basic} does not hold, that is, $\mu a^{2q-2}\geq \eta_{q}^{-1} \mathcal{C}_{G}(q)^{-1},$ then it is true that $\gamma(a)=0$ and $\gamma(a)$ can not be attained. Thus \eqref{eq:NLS_S} has no ground states on $S(a)$.
		\end{enumerate}
	\end{theorem}
	\begin{remark}
		To prove Theorem \ref{thm:mass_critical}(i), we use two methods.
		The first one is purely variational and similar to Theorem \ref{thm:mass_subcritical}.
		The other one is using the HLS subcritical approximation. 
		Precisely, we find a sequence $\{p_n\}\subset (2_{\alpha}^{\sharp},2_{\alpha}^*)$ with $\lim\limits_{n\rightarrow +\infty}p_n=2_{\alpha}^*$.
		Then we obtain a sequence of pairs $(u_n,\lambda_n):=(u_{p_n,a},\lambda_{p_n,a})\in H^1(\mathbb{R}^N)\times \mathbb{R}^-$, where $u_n$ are normalized solutions to \eqref{eq:NLS_S} with $p=p_n$ respectively.
		We prove that $\{u_n\}$ and $\{\lambda_n\}$ are bounded in $H^1(\mathbb{R}^N)$ and $\mathbb{R}$ respectively.
		As a consequence, up to a subsequence, $u_n \rightharpoonup u_a$ weakly in $H^1(\mathbb{R}^N)$ and $\lambda_n \rightarrow \lambda_a$ in $\mathbb{R}$. 
		The remain part is to prove that $u_a$ is a nontrivial normalized solution to \eqref{eq:NLS_S} with $\lambda_a<0$ for $p=2_{\alpha}^*$.
		The estimate of critical value in Lemma \ref{lm:energy_mass_critical} plays an important role in both two methods and our method in Lemma \ref{lm:energy_mass_critical} is slightly directer that \cite[Lemma 5.3]{MR4096725}.
		Besides, we point out that the second method is also powerful in studying Theorem \ref{thm:mass_supercritical} without a proof. 
	\end{remark}
	The next theorem gives an almost complete answer to \textbf{(Q2)} for the focusing $L^2$-supercritical perturbation, except for the case $N=3$, $p=2_{\alpha}^{*}$ with $\alpha \in [\min\{q-1,3\},3)$.
	\begin{theorem}\label{thm:mass_supercritical}
		Let $N\geq 3$, $\alpha \in (0,N)$, $2_{\alpha}^{\sharp}<q<p=2_{\alpha}^{*}$ and $\mu,a>0$.
		Suppose that \eqref{assumptions:basic} holds.
		For $N=3$ and $p=2_{\alpha}^{*}$, we additionally assume that $\alpha \in (0,\min\{q-1,3\})$.
		Then there exists a pair $(u_a,\lambda_a)\in H^{1}(\mathbb{R}^N)\times\mathbb{R}^{-}$  such that $u_a$ is a normalized solution to \eqref{eq:NLS_S} at level $\gamma(a)>0$.
	\end{theorem} 
		Indeed, using our proof of the above theorem, we can prove similar results for $2_{\alpha}^{\sharp}<q<p<2_{\alpha}^{*}$.
		However, the case $2_{\alpha}^{\sharp}<q<p<2_{\alpha}^{*}$ can be founded in \cite{MR3390522}, so we do not spend time and space describing the proof.
		Since the HLS subcritical approximation method in Theorem \ref{thm:mass_critical} also holds for Theorem \ref{thm:mass_supercritical} and the result for the case $2_{\alpha}^{\sharp}<q<p<2_{\alpha}^{*}$ in \cite{MR3390522} is known independent of $\mu$, we obtain a proof for Theorem \ref{thm:mass_supercritical} with Lemma \ref{lm:estimate_bubble}, which we do not describe it in this paper due to the similarity of Theorem \ref{thm:mass_critical}.
		But this fact gives us the confidence to deduce Theorem \ref{thm:mass_supercritical} in a direct way, and we succeed.
	\begin{remark}
		We point out that to solve \textbf{(Q2)} with the HLS critical leading term and the focusing $L^2$-supercritical perturbation about the equation \eqref{eq:NLS_l}, many people try to estimate the critical value via different testing functions and some properties of mapping $a\mapsto \gamma(a)$, for example, see \cite{2021Jeanjean,MR4096725,MR4433054}.
		However, our method is different form above choice.
		With suitable scaling skills and without the monotonicity of mapping $a\mapsto \gamma(a)$, we can solve \textbf{(Q2)} for \eqref{eq:NLS_S}.
		Our method is valid for \eqref{eq:NLS_l} after proper changes.
	\end{remark}
	The next theorem gives a partial answer to \textbf{(Q1)} and \textbf{(Q2)} for the $L^2$-supercritical leading term with a defocusing $L^2$-subcritical (resp. $L^2$-critical) perturbation.
	\begin{theorem}\label{thm:defousing}
		Let $N\geq 3$, $\alpha\in (0,N)$, $2_{\alpha}<q\leq 2_{\alpha}^{\sharp}< p \leq 2_{\alpha}^{*}$, $a>0$ and $\mu<0$.
		\begin{enumerate}
			\item[(i)] if $p\in (2_{\alpha}^{\sharp},2_{\alpha}^*)$ and \eqref{assumptions:basic} hold, then 
			there exists a pair $(u_a,\lambda_a)\in H^{1}(\mathbb{R}^N)\times\mathbb{R}^{-}$  such that $u_a$ is a normalized solution to \eqref{eq:NLS_S} at level $\gamma(a)>0$;
			\item[(ii)] if $p=2_{\alpha}^*$ and we additionally assume that $\alpha \in ((N-4)^{+},N)$ and $q\in [2,2_{\alpha}^{*})$ hold, then the equation \eqref{eq:NLS_S} has no positive solution in $H^{1}(\mathbb{R}^N)$.
		\end{enumerate}
	\end{theorem}
	In fact, the proof in Theorem \ref{thm:defousing}(i) covers the case $\mu=0$ and $p\in (2_{\alpha},2_{\alpha}^{*})$, i.e., a model case for the normalized solution problem about the homogeneous Choquard equation with a $L^2$-supercritical and HLS subcritical nonlinearity.
	In this direction, Li and Ye obtain the existence of one positive normalized solution to the following equation in \cite{MR3390522}:
		\begin{align*}
			-\Delta u - \lambda u =(I_{\alpha} * F(u))F^{\prime}(u),
		\end{align*}
	where $F$ satisfies some $L^2$-supercritical and HLS subcritical conditions.
	Ye studies the normalized solutions to \eqref{eq:NLS_S} with $\mu=0$ in \cite{MR3642765}, where the author points out the $L^2$-critical exponent of the equation \eqref{eq:NLS_S} with $\mu=0$ is $2_{\alpha}^{\sharp}=\frac{N+\alpha+2}{N}$.
	
	\begin{remark}
		It is worth pointing out that normalized solutions we obtained in Theorem \ref{thm:mass_subcritical}, Theorem \ref{thm:mass_critical} and Theorem \ref{thm:mass_supercritical} are real-valued, positive, continuous and radially decreasing in $\mathbb{R}^N$.
		Moreover, only for the case $2_{\alpha}<q<2_{\alpha}^{\sharp}<p\leq2_{\alpha}^{*}$ and $\mu>0$, we possibly obtain normalized solutions $u_a$ with negative energy $E[u_a]<0$.
	\end{remark}
	The last theorem gives a partial answer to \textbf{(Q4)}.
	We shall write $u_{a,\mu}$ to show the dependence of $u_a$ about $\mu$.
	\begin{theorem}\label{thm:asymptotic_mu_0}
		Let $N\geq 3$, $\alpha \in (0,N)$ and $2_{\alpha}<q<p=2_{\alpha}^{*}$.
		For a fixed $a>0$, suppose that $\mu>0$ and the assumptions in Theorem \ref{thm:mass_subcritical}, Theorem \ref{thm:mass_critical} and Theorem \ref{thm:mass_supercritical} hold such that there exist positive redial ground states $u_{a,\mu}$ at energy level $\gamma(a,\mu)$ respectively. 
		Then it holds that
		\begin{enumerate}
			\item[(i)] if $q\in (2_{\alpha},2_{\alpha}^{\sharp})$, then $\gamma(a,\mu)\rightarrow 0$, and $A[u_{a,\mu}]\rightarrow 0$ as $\mu \rightarrow 0^+$;
			\item[(ii)] if $q\in [2_{\alpha}^{\sharp},2_{\alpha}^{*})$, then $\gamma(a,\mu)\rightarrow \frac{\alpha+2}{2(N+\alpha)}\mathcal{S}^{\frac{N+\alpha}{\alpha+2}}_{H}$, and $A[u_{a,\mu}]\rightarrow \mathcal{S}^{\frac{N+\alpha}{\alpha+2}}_{H}$ as $\mu \rightarrow 0^+$.
		\end{enumerate}
	\end{theorem}
	
	The paper is organized as follows. In section 2 we recall some classical inequalities and present some preliminary results. 
	Section 3 is devoted to the treatment of $L^2$-supercritical leading term with a focusing perturbation.
	In subsection 3.1, we consider the $L^2$-subcritical perturbation. 
	We first obtain a bounded Palais-Smale sequence in part 3.1.1, and then recover the compactness in part 3.1.2-3.1.3.
	Theorem \ref{thm:mass_subcritical} is completed in part 3.1.2-3.1.3.
	To deal with the HLS critical leading term, an alternative Lemma \ref{lm:PS_strong_convergence} is proved in part 3.1.3.
	Moreover, Theorem \ref{thm:minizing} is investigated in part 3.1.4.
	We then study the $L^2$-critical perturbation in subsection 3.2.
	Part 3.2.1-3.2.2 are used for showing Theorem \ref{thm:mass_critical} and in part 3.2.3, we shall give a detailed description of HLS subcritical approximation method for $p=2_{\alpha}^*$ in Theorem \ref{thm:mass_critical}.
	Subsection 3.3 is devoted to proving Theorem \ref{thm:mass_supercritical}, that is, the $L^2$-supercritical perturbation case.
	In section 4, we shall focus on Theorem \ref{thm:defousing}, i.e., the $L^2$-supercritical leading term with a defocusing perturbation. 
	We obtain an existence result for $2_{\alpha}<q\leq 2_{\alpha}^{\sharp}<p<2_{\alpha}^*$ in part 3.3.1 and a nonexistence result for $2\leq q<p=2_{\alpha}^*$ in part 3.3.2.
	Finally, in section 5, we will analyze the asymptotic behavior of ground states as $\mu \rightarrow 0^+$, that is, to prove Theorem \ref{thm:asymptotic_mu_0}.
	
	We shall finish this introduction by clarifying some notations.
	Throughout this paper, $C$ are indiscriminately used to denote various absolutely positive constants. 
	$a\lesssim b$ means that $a\leq Cb$.
	For $p>1$, the $L^p$-norm of $u\in L^p(\mathbb{R}^N)$ is denoted by $\|u\|_p$.
	$B_R$ means an open ball centered on the origin with radius $R$ in $\mathbb{R}^N$.
	\section{Preliminaries}
	In this section, we shall present various preliminary results.
	The following Hardy-Littlewood-Sobolev (abbreviated as ``HLS") inequality can be founded in \cite{1999Analysisby}.
	\begin{lemma}
		Let $N \geq 1$, $\alpha \in (0,N)$ and $r,s>1$ with $1 / r+(N-\alpha) / N+1 / s=2$. 
		Set $f \in L^{r}\left(\mathbb{R}^{N}\right)$ and $g \in L^{s}\left(\mathbb{R}^{N}\right)$. 
		Then there exists a constant $\mathcal{C}_{H}(N, \alpha, r)$ such that
		\begin{equation}\label{ieq:HLS}
			\left|\int_{\mathbb{R}^{N}} \int_{\mathbb{R}^{N}} \frac{f(x) g(y)}{|x-y|^{N-\alpha}} \dif x \dif y\right| \leq \mathcal{C}_{H}(N, \alpha, r)\|f\|_{r}\|g\|_{s}.
		\end{equation}
		If $r=s=\frac{2 N}{N+\alpha}$, then
		$$
		\mathcal{C}_{H}(N,\alpha):=\mathcal{C}_{H}(N, \alpha, \frac{2 N}{N+\alpha})=\pi^{\frac{N-\alpha}{2}} \frac{\Gamma\left(\frac{\alpha}{2}\right)}{\Gamma\left(\frac{N+\alpha}{2}\right)}\left\{\frac{\Gamma\left(\frac{N}{2}\right)}{\Gamma(N)}\right\}^{-\frac{\alpha}{N}}.
		$$
		In this case, the equality in \eqref{ieq:HLS} holds if and only if $f \equiv \mathcal{C}_{H}(N,\alpha) h$ and
		$$
		h(x)=C\left(\gamma^{2}+|x-a|^{2}\right)^{-(N+\alpha) / 2}
		$$
		for some $C \in \mathbb{C}, 0 \neq \gamma \in \mathbb{R}$ and $a \in \mathbb{R}^{N}$.
	\end{lemma} 
	From the Hardy-Littlewood-Sobolev inequality \eqref{ieq:HLS}, we know, for all $u\in \mathscr{D}^{1,2}(\mathbb{R}^N)$, it holds that
	\begin{equation*}
		\int_{\mathbb{R}^{N}} \int_{\mathbb{R}^{N}} \frac{\vert u(x)\vert^{2_{\alpha}^{*}} \vert u(y)\vert^{2_{\alpha}^{*}}}{|x-y|^{N-\alpha}} \dif x \dif y\leq \mathcal{C}_{H}(N, \alpha)\|u\|_{2^*}^{22_{\alpha}^*}\leq \mathcal{C}_{H}(N, \alpha)\mathcal{S}^{-2_{\alpha}^{*}}\left(\int_{\mathbb{R}^{N}}\vert \nabla u\vert^2\dif x\right)^{2_{\alpha}^*},
	\end{equation*}
	where $\mathcal{S}$ is the best Sobolev constant.
	In the following, we use $\mathcal{S}_{H}$ to denote best constant defined by
	\begin{equation}\label{ieq:best_constant}
		\mathcal{S}_{H}:=\inf_{u\in \mathscr{D}^{1,2}(\mathbb{R}^N)\setminus\{0\}}\frac{\int_{\mathbb{R}^{N}}\vert \nabla u\vert^2\dif x}{\left(\mathcal{A}(N,\alpha)\int_{\mathbb{R}^{N}} \int_{\mathbb{R}^{N}} \frac{\vert u(x)\vert^{2_{\alpha}^{*}} \vert u(y)\vert^{2_{\alpha}^{*}}}{|x-y|^{N-\alpha}} \dif x \dif y\right)^{\frac{1}{2_{\alpha}^*}}}.
	\end{equation}
	From \cite{2019UNIQUENESS}, we know that $\mathcal{S}_{H}$ is reached and
	\begin{align*}
		\mathcal{S}_{H}=\frac{\mathcal{S}}{\left[\mathcal{A}(N,\alpha)\mathcal{C}_{H}(N, \alpha)\right]^{\frac{N-2}{N+\alpha}}}
	\end{align*}
	The following Gagliardo-Nirenberg inequality of Hartree type is an important tool, see \cite{MOROZ2013153}.
	\begin{lemma}
		Let $N \geq 1,$ $0<\alpha<N$ and $2_{\alpha}<p<2_{\alpha}^{*}$.  
		Then there exists a constant $\mathcal{C}_{G}(N,\alpha,p)$ such that
		\begin{equation}\label{ieq:GN_H}
			\int_{\mathbb{R}^{N}}\left(I_{\alpha} *|u|^{p}\right)|u|^{p} \leq \mathcal{C}_{G}(N,\alpha, p)\|\nabla u\|_{2}^{2p\eta_{p}}\|u\|_{2}^{2p(1-\eta_{p})} .
		\end{equation}
		The best constant $\mathcal{C}_{G}(N,\alpha,p)$ is defined by
		$$
		\mathcal{C}_{G}(N,\alpha, p)=\frac{p}{\|W_p\|_2^{2p-2}},
		$$
		where $W_{p}$ is a radially ground state solution of the elliptic equation
		$$
		-p\eta_{p}\Delta W+p(1-\eta_{p})W=\left(I_{\alpha} *|W|^{p}\right)|W|^{p-2} W \quad\text{in }\mathbb{R}^N.
		$$
	\end{lemma}

	\begin{remark}
		\begin{enumerate}
			\item[(i)] If $p\in (2_{\alpha},2_{\alpha}^{*})$, the Gagliardo-Nirenberg inequality of Hartree type \eqref{ieq:GN_H} implies that
			\begin{align}\label{estimate:subcritical}
				B_p[u]\leq \mathcal{C}_{G}(p) \| u \|_2^{2p(1-\eta_{p})}A[u]^{p\eta_{p}}.
			\end{align}
			\item[(ii)] If $p=2_{\alpha}^{*}$, the definition of best constant $\mathcal{S}_{H}$  \eqref{ieq:best_constant} implies that 
			\begin{align}\label{estimate:critical}
				B_p[u]\leq \mathcal{S}_{H}^{-p}A[u]^{p}=\mathcal{A}(N,\alpha)\mathcal{C}_{H}(N,\alpha)\mathcal{S}^{-p} A[u]^{p}.
			\end{align}
		\end{enumerate}
	\end{remark}
	\begin{lemma}\label{lm:Langrange_multiplier}
		Let $N\geq 3$, $\mu \in \mathbb{R}$, $p\in (2_{\alpha}^{\sharp},2_{\alpha}^{*}]$ and $q\in (2_{\alpha},2_{\alpha}^{*})$ with $q<p$.
		If $u\in H^{1}(\mathbb{R}^N)$ is a weak solution to $\eqref{eq:NLS_S}$, then $P[u]=0$.
		Moreover, if $u\ne 0$, then we have
		\begin{enumerate}
			\item[(i)] $\lambda>0$ provided that $p=2_{\alpha}^{*}$ and $\mu<0$;
			\item[(ii)] $\lambda<0$  provided that $p\in (2_{\alpha}^{\sharp},2_{\alpha}^{*}]$ and $\mu>0$.
		\end{enumerate}
	\end{lemma}
	\begin{proof}
		Testing \eqref{eq:NLS_S} by $u$, we obtain that
		\begin{align}\label{eq:testing}
			A[u]-\lambda \|u\|_2^2-B_p[u]-\mu B_q[u]=0.
		\end{align}
		Combing \eqref{eq:Pohozaev} and $\eqref{eq:testing}$, we obtain that $P[u]=0$.
		Using \eqref{eq:Pohozaev} and \eqref{eq:testing} again, we get
		\begin{align*}
			\lambda \eta_{p}\|u\|_2^2=(1-\eta_{p})A[u]+(\eta_{q}-\eta_{p})\mu B_q[u].
		\end{align*}
		Then if $p=2_{\alpha}^{*}$ with $\mu<0$, we have $\eta_{p}=1$, and hence $\lambda>0$.
		Again, using \eqref{eq:Pohozaev} and \eqref{eq:testing}, we have
		\begin{align*}
			\lambda \|u\|_2^2=(\eta_{p}-1)B_p[u]+\mu(\eta_{q}-1) B_q[u],
		\end{align*}
		which implies that if $p\in (2_{\alpha}^{\sharp},2_{\alpha}^{*}]$ and $\mu>0$, we have $\lambda<0$.
	\end{proof}
	\begin{lemma}\label{lm:regularity}
		Let $N\geq 3$, $\mu \in \mathbb{R}$, $p\in (2_{\alpha}^{\sharp},2_{\alpha}^{*}]$ and $q\in (2_{\alpha},2_{\alpha}^{*})$ with $q<p$.
		Suppose that  $u\in H^{1}(\mathbb{R}^N)$ is a weak solution to $\eqref{eq:NLS_S}$, then we have
		\begin{enumerate}
			\item[(i)] $u$ belongs to  $\left(\bigcap\limits_{r> 1}W^{2,r}_{\operatorname{loc}}(\mathbb{R}^N)\right)\cap\mathscr{C}(\mathbb{R}^N)$; 
			\item[(ii)] $u>0$ if $\mu>0$ and $u\geq 0$ with $u\not\equiv 0$.
		\end{enumerate}
	\end{lemma}
	\begin{proof}
		\begin{enumerate}
			\item[(i)]Applying \cite[Theorem 2.1]{MR4106818} to $f(u)=\vert u \vert^{p-2}u+\mu\vert u \vert^{q-2}u$ and $g(u)\equiv0$, we obtain that $u\in W^{2,r}_{\operatorname{loc}}(\mathbb{R}^N)$ for every $r>1$.
			As a consequence, we obtain that $u \in \mathscr{C}(\mathbb{R}^N)$.
			\item[(ii)] We shall argue by contradiction. 
			By Lemma \ref{lm:Langrange_multiplier}, we obtain that $\lambda<0$.
			Suppose there exits $x_0\in\mathbb{R}^N$ such that $u(x_0)=0$ and define an elliptic operator $L=-\Delta - \lambda$.
			For all $R>\vert x_0 \vert$, we have $u \in W^{1,2}(B_R)$  and $Lu\geq 0$ since $\mu>0$.
			By the strong maximum principle \cite[Lemma 8.20]{MR1814364} and $u\geq 0$, we obtain that $u$ is a constant in $B_{R}$ and hence $u\equiv 0$ in $B_R$, which contradicts the fact $u \not\equiv 0$ since $R$ can be large enough.
		\end{enumerate}
	\end{proof}
	The following lemma can be founded in \cite[Radial Lemma A.IV]{MR695535}.
	\begin{lemma}\label{lm:decay}
		Let $N\geq 3$ and $t\in [1,+\infty)$. If $u \in L^t(\mathbb{R}^N)$ is a radial nonincreasing function, then it holds that 
		\begin{align*}
			\vert u(x) \vert \leq \vert x \vert^{-\frac{N}{t}}{N}^{\frac{1}{t}}\omega^{-\frac{1}{t}}\|u\|_t,\quad x\ne 0,
		\end{align*}
		where $\omega$ is the area of unit sphere in $\mathbb{R}^N$.
	\end{lemma}
	Following \cite[Section 8]{MR695536}, we recall that, for $a>0$, $S(a)$ is a submanifold of the Hilbert manifold $H^{1}(\mathbb{R}^N)$ with codimension 1.
	The tangent space at $u\in S(c)$ is defined by
	\begin{align*}
		T_uS(a):=\{\phi\in H^{1}(\mathbb{R}^N): \langle u,\phi\rangle_2=0\}.
	\end{align*}
	It is easy to see that $\left.E\right\vert_{S(a)}:S(a)\rightarrow \mathbb{R}$ is a $\mathscr{C}^1$ functional on $S(a)$ and for any $u\in S(a)$ and $\phi\in T_uS(a)$, we get
	\begin{align*}
		\langle \left.\dif E\right\vert_{S(a)},\phi\rangle=\langle \dif E,\phi\rangle.
	\end{align*}

	We shall use the following compact embedding theorem, see \cite[Lemma 2]{MR454365} or \cite[Corollary 1.26]{MR1400007}.
	\begin{lemma}\label{lm:compact_embedding}
		Let $N\geq 3$ and $p\in (2,2^*)$. 
		Then the following embedding is compact:
		\begin{equation}
			H^1_{\operatorname{rad}}(\mathbb{R}^N)\subset L^{p}(\mathbb{R}^N),
		\end{equation}
		where $H^1_{\operatorname{rad}}(\mathbb{R}^N)$ is the subspace of $H^1(\mathbb{R}^N)$ consisting of radial functions.
	\end{lemma}
	The following lemma is a general version of \cite[Lemma 2.1]{MR2230354}.
	\begin{lemma}\label{lm:compact_operator_nonlocal}
		If $\{u_n\}$ is a sequence satisfying $u_n \rightharpoonup u$ weakly in $H^1_{\operatorname{rad}}(\mathbb{R}^N)$, then we have $B_q[u_n]\rightarrow B_q[u]$ for $q\in (2_{\alpha},2_{\alpha}^{*})$.
	\end{lemma}
	\begin{proof}
		It is easy to check that
		$$
		B_q[u_n]-B_q[u]=\int_{\mathbb{R}^{N}}\left(I_{\alpha} *\left(\left|u_{n}\right|^{q}+|u|^{q}\right)\right)\left(\left|u_{n}\right|^{q}-|u|^{q}\right)\dif x .
		$$
		It follows from $2<\frac{2Nq}{N+\alpha}<2^{*}$ and H\"{o}lder inequality that 
		$$
		\begin{aligned}
			&\vert B_q[u_n]-B_q[u]\vert  \\
			\lesssim& \left(\int_{\mathbb{R}^{N}}\left(\left|u_{n}\right|^{q}+|u|^{q}\dif x\right)^{\frac{2 N}{ N+\alpha}}\right)^{\frac{N+\alpha}{2 N}}\left(\int_{\mathbb{R}^{N}}\left\vert \vert u_{n}\vert^{q}-\vert u\vert^{q}\right\vert^{\frac{2 N}{N+\alpha}}\dif x\right)^{\frac{N+\alpha}{2 N}} \\
			\lesssim& \left(\int_{\mathbb{R}^{N}}\left(\left|u_{n}\right|^{q-1}+|u|^{q-1}\right)^{\frac{2 N}{N+\alpha}}\left|u_{n}-u\right|^{\frac{2 N}{N+\alpha}}\dif x\right)^{\frac{N+\alpha}{2 N}} 
			\left(\int_{\mathbb{R}^{N}}\left(\left|u_{n}\right|^{q}+|u|^{q}\right)^{\frac{2 N}{N+\alpha}}\dif x\right)^{\frac{N+\alpha}{2 N}} \\
			\lesssim& \int_{\mathbb{R}^{N}}\left(\left|u_{n}\right|^{\frac{2 N q}{N+\alpha}}+|u|^{\frac{2 N q}{N+\alpha}}\dif x\right)^{\frac{N+\alpha}{N}\left(1-\frac{1}{2 q}\right)}\left(\int_{\mathbb{R}^{N}}\left|u_{n}-u\right|^{\frac{2 N q}{N+\alpha}}\dif x\right)^{\frac{N+\alpha}{2 N q}}\rightarrow 0,
		\end{aligned}
		$$
		where we have used $\vert \vert a \vert^p-\vert b\vert^p\vert\leq C_p\vert a-b\vert(\vert a\vert^{p-1}+\vert b \vert^{p-1})$ for every $a,b \in \mathbb{R}$ and Lemma \ref{lm:compact_embedding}.
	\end{proof}
	We shall use the Brezis–Lieb type lemma for the nonlocal term, see \cite[Lemma 3.4]{MR2088936} for the subcritical case and \cite[Lemma 2.2]{MR3817173} for the critical case.
	\begin{lemma}\label{lm:Brezis–Lieb_nonlocal}
		Let $N \geq 3$, $0<\alpha<N$ and $2_{\alpha} \leq p \leq 2_{\alpha}^{*}$. If $\left\{u_{n}\right\}_{n\in\mathbb{N}}$ is a bounded sequence in $L^{2^{*}}(\mathbb{R}^N)$ such that $u_{n} \rightarrow u$ a.e. in $\mathbb{R}^N$ as $n \rightarrow \infty$, then we have
		$$
		B_p[u_n]=B_p[u_n-u]+B_p[u]+o_n(1)
		$$
		as $n \rightarrow \infty$.
	\end{lemma}
	
	The following lemma plays an important role in this paper:
	\begin{lemma}\label{lm:coercive}
		Let $N\geq 3$, $\alpha\in (0,N)$ and $a>0$. 
		Then we have $E$ restricted to $\mathcal{P}(a)$ is coercive on $H^{1}(\mathbb{R}^N)$ ( i.e., 
		if $\{u_n\}\subset \mathcal{P}(a)$ satisfies $\|u_n\|\rightarrow +\infty$, then $E[u_n]\rightarrow +\infty$), 
		provided that one of the following holds:
		\begin{enumerate}
			\item[(i)] $\mu\leq 0$;
			\item[(ii)] $2_{\alpha}<q<2_{\alpha}^{\sharp}<p\leq 2_{\alpha}^{*}$ and $\mu>0$;
			\item[(iii)] $q=2_{\alpha}^{\sharp}<p\leq 2_{\alpha}^{*}$ and $0<\mu a^{2q-2}<q\mathcal{C}_{G}(N,\alpha,q)^{-1}$;
			\item[(iv)] $p=2_{\alpha}^{*}$, $q\in(2_{\alpha}^{\sharp},2_{\alpha}^{*})$ and $\mu>0$.
		\end{enumerate} 
	\end{lemma}
	\begin{proof}
			To prove $(i)-(iii)$, let $u\in \mathcal{P}(a)$, then we have 
			$$\frac{1}{\eta_{p}}(A[u]-\mu\eta_{q}B_q[u])=B_p[u].$$
			Hence for $2_{\alpha}\leq q\leq 2^{\sharp}_{\alpha}< p\leq 2_{\alpha}^{*}$ and $\mu\in\mathbb{R}$, we have
			\begin{align*}
				E[u]=&
				(\frac{1}{2}-\frac{1}{2p\eta_{p}})A[u]-\frac{\mu}{2q}(1-\frac{2q\eta_q}{2p\eta_p})B_q[u]\\
				\geq &
				\begin{cases}
					(\frac{1}{2}-\frac{1}{2p\eta_{p}})A[u]-\frac{\mu}{2q}(1-\frac{2q\eta_q}{2p\eta_p})\mathcal{C}_{G}(q)a^{2q(1-\eta_{q})}A[u]^{q\eta_{q}},&\text{ if }\ \mu>0,\\
					(\frac{1}{2}-\frac{1}{2p\eta_{p}})A[u],&\text{ if }\ \mu\leq 0.
				\end{cases} 
			\end{align*}
			This concludes the proof except for $q=2_{\alpha}^{\sharp}$ with $\mu>0$. In this case, we have
			\begin{equation*}
				E[u]\geq (\frac{1}{2}-\frac{1}{2p\eta_{p}})(1-\frac{\mu}{q}\mathcal{C}_{G}(N,\alpha,q)a^{2q-2})A[u],
			\end{equation*}
			and we complete the proof since $0<\mu a^{2q-2}<q\mathcal{C}_{G}(N,\alpha,q)^{-1}$ for $q=2_{\alpha}^{\sharp}$.

			$(iv)$ The direct method in point $(i)-(iii)$ is invalid. 
			Motivated by \cite[Proposition 2.7]{MR2557725} and  \cite[Lemma 2.5]{MR4150876}, we shall argue by contradiction.
			Notice that for $u\in \mathcal{P}(a) $, we have
			\begin{align*}
				E[u]=E[u]-\frac{1}{2}P[u]=\frac{p-1}{2p}B_p[u]+\mu\frac{q\eta_{q}-1}{2q}B_q[u]\geq 0,
			\end{align*}
			which implies that $\inf\limits_{u\in\mathcal{P}(a)} E[u] \geq 0$.
			Suppose that $\{u_n\}\subset \mathcal{P}(a)$ is a sequence such that $A[u_n]\rightarrow +\infty$ and $E[u_n]$ is bounded form above.
			The key step is using different scaling for each $\{u_n\}$.
			More precisely, for $\{t_n\}\subset \mathbb{R}^{+}$, set $$t_n^2=\frac{1}{A[u_n]}>0\quad \text{and} \quad v_n(x):=t_n^{\frac{N}{2}}u_n(t_nx).$$
			Obverse that $t_n \rightarrow 0^{+}$ as $n \rightarrow +\infty$ and 
			\begin{align*}
				\|v_n\|_2=\|u_n\|_2=a,\quad A[v_n]=t_n^2A[u_n]=1,\quad B_s[v_n]=t_n^{2s\eta_s}B_s[u_n].
			\end{align*}
			Consequently, $\{v_n\}$ is a bounded sequence in $H^1(\mathbb{R}^N)$.
			Let $$\rho:=\limsup_{n\rightarrow+\infty}\left(\sup_{y\in\mathbb{R}^N}\int_{B(y,1)}\vert v_n \vert^2\dif x\right).$$
			To derive a contradiction, we consider two cases: \textbf{nonvanishing} and \textbf{vanishing}.
			\begin{itemize}
				\item \textbf{Nonvanishing}: that is $\rho >0$.
				Up to a subsequence, there exists a sequence $\{z_n\}\subset \mathbb{R}^N$ and $w\in H^{1}(\mathbb{R}^N)\setminus \{0\}$ satisfying
				\begin{align*}
					w_n:=v_n(\cdot + z_n)\rightharpoonup w\quad \text{in } H^{1}(\mathbb{R}^N)\quad \text{and}\quad w_n \rightarrow w\quad \text{a.e. in } \mathbb{R}^N.
				\end{align*}
				Then we obtain that
				\begin{align*}
					0\leq \frac{E[u_n]}{A[u_n]}=&\frac{p-1}{2p}-\frac{\mu}{2q}\left(1-\frac{q\eta_{q}}{p}\right)\frac{B_q[u_n]}{A[u_n]}\\
					=&\frac{p-1}{2p}-\frac{\mu}{2q}\left(1-\frac{q\eta_{q}}{p}\right)t_n^{2+N+\alpha}B_q[t_n^{-\frac{N}{2}}v_n]\\
					=&C_1-C_2t_n^{-2(q\eta_{q}-1)}\int_{\mathbb{R}^N}\int_{\mathbb{R}^N}\frac{\vert v_n(x) \vert^q \vert v_n(y) \vert^q}{\vert x-y \vert^{N-\alpha}}\dif x \dif y
					\rightarrow -\infty.
				\end{align*}
				This leads to a contradiction.
				\item \textbf{Vanishing}: that ia $\rho=0$. By Lions' vanishing Lemma \cite[Lemma I.1]{MR778974}, we obtain that $v_n \rightarrow 0$ in $L^s({\mathbb{R}^N})$ for $s\in (2,2^{*})$.
				Then we have
				\begin{align*}
					E[u_n]=E[t_n^{-\frac{N}{2}}v_n(t_n^{-1}\cdot)]
					\geq& E[t^{-\frac{N}{2}}v_n(t^{-1}\cdot)]
					\geq \frac{p-1}{2p}t^2-C_3t^{-2s\eta_{q}}\|v_n\|_{\frac{2q}{N+\alpha}}^{2q}=\frac{p-1}{2p}t^2+o_n(1),
				\end{align*}
				where we have used the Hardy-Littewood-Sobolev inequality \eqref{ieq:HLS} and $q\in (2_{\alpha}, 2_{\alpha}^{*})$ implying that $\frac{2q}{N+\alpha} \in (2, 2^{*})$ for $N\geq 3$.
				Therefore, we obtain a contradiction since $t$ can be chosen large enough and $E$ is bounded form above.
			\end{itemize}
	\end{proof}
	\section{$L^2$-supercritical leading term with focusing perturbation}
	In this section, we shall always assume that $\mu,a>0$, $2_{\alpha}^{\sharp}<p\leq2_{\alpha}^{*}$ and $2_{\alpha}<q<p$.
	\subsection{$L^2$-subcritical perturbation}
	\subsubsection{The existence of bounded Palais-Smale sequences}
	According to $2_{\alpha}< q< 2^{\sharp}_{\alpha}< p\leq 2_{\alpha}^{*}$ and $\mu,a>0$, it follows that $0<q\eta_{q}<1<p\eta_{p}$ and $f^{\prime\prime}_u(s)$ defined by \eqref{eq:scaling_function_2nd_derivative} possesses a unique zero $s_u^{\diamond}=\left[\frac{\eta_{q}(1-q\eta_{q})B_q[u]}{\eta_{p}(p\eta_{p}-1)B_p[u]}\right]^{\frac{1}{p\eta_{p}-q\eta_{q}}}$. 
	Moreover, it holds that
	$$f^{\prime\prime}_u(s_u^{\diamond})=0,\quad f^{\prime\prime}_u(s)>0\text{ if } 0<s<s_u^{\diamond},\quad f^{\prime\prime}_u(s)<0\text{ if } s>s_u^{\diamond}.$$
	\begin{lemma}\label{lm:derivative_positive}
		Let $N\geq 3$, $\alpha\in (0,N)$, $2_{\alpha}< q< 2^{\sharp}_{\alpha}< p\leq 2_{\alpha}^{*}$ and $\mu,a>0$.
		Suppose that \eqref{assumptions:basic} is true.
		For any  $u\in S(a)$, it holds that $f^{\prime}_u(s_u^{\diamond})>0$.
	\end{lemma}
	\begin{proof}
		We first consider the case $p\in (2_{\alpha}^{\sharp},2_{\alpha}^{*})$. 
		It follows from $f^{\prime\prime}_u(s_u^{\diamond})=0$ that
		$$\eta_{p}(p\eta_{p}-1)(s_u^{\diamond})^{p\eta_{p}-1}B_p[u]=\mu\eta_{q}(1-q\eta_{q}) (s_u^{\diamond})^{q\eta_{q}-1}B_q[u].$$ 
		Consequently, by \eqref{assumptions:basic} and \eqref{estimate:subcritical}, it follows that
		\begin{align*}
			f^{\prime}_u(s_u^{\diamond})=&A[u]-\eta_{p}(s_u^{\diamond})^{p\eta_{p}-1}B_p[u]-\mu\eta_{q} (s_u^{\diamond})^{q\eta_{q}-1}B_q[u]\\
			= & A[u]-\mu\eta_{q}\frac{p\eta_{p}-q\eta_{q}}{p\eta_{p}-1}(s_u^{\diamond})^{q\eta_{q}-1}B_q[u]\\
			= & A[u]-\mu\eta_{q}\frac{p\eta_{p}-q\eta_{q}}{p\eta_{p}-1}\left[\frac{\mu\eta_{q}(1-q\eta_{q})B_q[u]}{\eta_{p}(p\eta_{p}-1)B_p[u]}\right]^{\frac{q\eta_{q}-1}{p\eta_{p}-q\eta_{q}}}B_q[u]\\
			=& (B_p[u])^{\frac{1-q\eta_{q}}{p\eta_{p}-q\eta_{q}}}\left\{A[u](B_p[u])^{\frac{q\eta_{q}-1}{p\eta_{p}-q\eta_{q}}}-\mu^{\frac{p\eta_{p}-1}{p\eta_{p}-q\eta_{q}}}\eta_{q}\frac{p\eta_{p}-q\eta_{q}}{p\eta_{p}-1}\left[\frac{\eta_{q}(1-q\eta_{q})}{\eta_{p}(p\eta_{p}-1)}\right]^{\frac{q\eta_{q}-1}{p\eta_{p}-q\eta_{q}}}(B_q[u])^{\frac{p\eta_{p}-1}{p\eta_{p}-q\eta_{q}}}\right\}\\
			\geq &
			(B_p[u])^{\frac{1-q\eta_{q}}{p\eta_{p}-q\eta_{q}}}(A[u])^{\frac{p\eta_{p}-1}{p\eta_{p}-q\eta_{q}}}\left\{(\mathcal{C}_{G}(p)a^{2p(1-\eta_{p})})^{\frac{q\eta_{q}-1}{p\eta_{p}-q\eta_{q}}}\right.\\
			&\left.-\mu^{\frac{p\eta_{p}-1}{p\eta_{p}-q\eta_{q}}}\eta_{q}\frac{p\eta_{p}-q\eta_{q}}{p\eta_{p}-1}\left[\frac{\eta_{q}(1-q\eta_{q})}{\eta_{p}(p\eta_{p}-1)}\right]^{\frac{q\eta_{q}-1}{p\eta_{p}-q\eta_{q}}}(\mathcal{C}_{G}(q)a^{2q(1-\eta_{q})})^{\frac{p\eta_{p}-1}{p\eta_{p}-q\eta_{q}}}\right\}
			>0.
		\end{align*}
		The main difference of proof for the case $p=2_{\alpha}^{*}$ is replacing the estimate \eqref{estimate:subcritical} by \eqref{estimate:critical} and we omit the details.
	\end{proof}
	\begin{lemma}
		Let $N\geq 3$, $\alpha\in (0,N)$, $2_{\alpha}< q< 2^{\sharp}_{\alpha}< p\leq 2_{\alpha}^{*}$ and $\mu,a>0$.
		Suppose that \eqref{assumptions:basic} holds.
		Then we have  $\mathcal{P}^{0}(a)=\emptyset$.
	\end{lemma}
	\begin{proof}
		We shall argue by contradiction and suppose that there exists a nontrivial function $u$ belonging to $\mathcal{P}^{0}(a)$.
		According to the fact $f_{u}^{\prime\prime}(1)=0$ and $s_{u}^{\diamond}$ is the unique zero of $f^{\prime\prime}_u(s)$ on $(0,+\infty)$, it follows that $s_u^{\diamond}=1$, and hence  $f^{\prime}_u(s_u^{\diamond})=f^{\prime}_u(1)=0$.
		This leads to a contradiction since Lemma \ref{lm:derivative_positive} implies that $f^{\prime}_u(1)=f^{\prime}_u(s_u^{\diamond})>0$.
	\end{proof}
	\begin{lemma}\label{lm:unique_minimum_maximum}
		Let $N\geq 3$, $\alpha\in (0,N)$, $2_{\alpha}< q< 2^{\sharp}_{\alpha}< p\leq 2_{\alpha}^{*}$ and $\mu,a>0$.
		Suppose that \eqref{assumptions:basic} holds.
		For any $u\in S(a)$, there exists
		\begin{enumerate}
			\item[(i)] a unique $s_{u}^{+} \in\left(0, s_u^{\diamond}\right)$ such that $s_{u}^{+}$ is a unique local minimum point for $f_{u}$ and $s_{u}^{+}  \circ  u \in \mathcal{P}^{+}(a)$; 
			\item[(ii)] a unique $s_{u}^{-} \in\left(s_u^{\diamond}, \infty\right)$ such that $s_{u}^{-}$is a unique local maximum point for $f_{u}$ and $s_{u}^{-}  \circ  u \in \mathcal{P}^{-}(a)$.
		\end{enumerate}
		Moreover, the maps $u \in S(a) \mapsto s_{u}^{+} \in \mathbb{R}$ and $u \in S(a) \mapsto s_{u}^{-} \in \mathbb{R}$ are of class $\mathcal{C}^{1}$.
	\end{lemma}
	\begin{proof}
		Since $0<q\eta_q<1<p\eta_p$ and the formula \eqref{eq:scaling_function_1st_derivative}, it follows that
		$$ f^{\prime}_u(s)\rightarrow -\infty \text{ as } s\rightarrow 0+,\quad \text{and}\quad
		f^{\prime}_u(s)\rightarrow -\infty \text{ as } s\rightarrow +\infty.$$
		By Lemma \ref{lm:derivative_positive}, it holds that $f^{\prime}_u(s_u^{\diamond})>0$.
		Therefore, there exists at least two zeros $s_u^+,s_u^-$ of $f_u(s)$ with $0<s_u^+<s_u^{\diamond}<s_u^{-}$ on $\mathbb{R}^{+}$.
		Taking into consideration that
		$$f^{\prime\prime}_u(s_u^{\diamond})=0,\quad f^{\prime\prime}_u(s)>0\text{ if } 0<s<s_u^{\diamond},\quad f^{\prime\prime}_u(s)<0\text{ if } s>s_u^{\diamond},$$
		we have $s_{u}^{+}$ (resp. $s_{u}^{-}$) is a unique local minimum (resp. maximum) point for $f_{u}(s)$ and $s_{u}  \circ  u^+ \in \mathcal{P}^{+}(a)$ (resp. $s_{u}^{-}  \circ  u \in \mathcal{P}^{-}(a)$).
		The remain parts is similar to \cite[Lemma 5.3(4)]{SOAVE20206941}.
	\end{proof}
	
	\begin{lemma}\label{lm:A_bounded_below}
		Let $N\geq 3$, $\alpha\in (0,N)$, $2_{\alpha}< q< 2^{\sharp}_{\alpha}< p\leq 2_{\alpha}^{*}$ and $\mu,a>0$.
		Suppose that \eqref{assumptions:basic} holds.
		For any $u\in S(a)$, there exists
		\begin{enumerate}
			\item[(i)] $E[u]<0$ for all $u\in \mathcal{P}^{+}(a)$;
			\item[(ii)] there exists a $\delta>0$ such that $A[u]\geq \delta$ for all $u\in\mathcal{P}^{-}(a)$.
		\end{enumerate}
	\end{lemma}
	\begin{proof}
		\begin{enumerate}
			\item[(i)] If $u\in \mathcal{P}^{+}(a)$, we have
			\begin{align*}
				0=A[u]-\eta_{p}B_p[u]-\mu\eta_{q} B_q[u],\quad
				0<-\eta_{p}(p\eta_{p}-1)B_p[u]+\mu\eta_{q}(1-q\eta_{q}) B_q[u].
			\end{align*}
			Therefore, it follows that
			\begin{align*}
				E[u]=&
				\frac{p\eta_{p}-1}{2p}B_p[u]+\frac{\mu(q\eta_{q}-1)}{2q}B_q[u]
				< 
				-\frac{\mu(1-q\eta_{q})(p\eta_{p}-q\eta_{q})}{2pq\eta_{p}}B_q[u].
			\end{align*}
			Since $2_{\alpha}<q<2_{\alpha}^{\sharp}<p\leq 2_{\alpha}^{*}$, it holds that $E[u]<0$ for all $u\in \mathcal{P}^{+}(a)$.
			\item[(ii)] 
			If $u\in \mathcal{P}^{-}(a)$, we have
			\begin{align*}
				0=A[u]-\eta_{p}B_p[u]-\mu\eta_{q} B_q[u],\quad
				0>-\eta_{p}(p\eta_{p}-1)B_p[u]+\mu\eta_{q}(1-q\eta_{q}) B_q[u].
			\end{align*}
			Similarly, we obtain
			\begin{align*}
				A[u]
				<\frac{\eta_{p}(p\eta_{p}-q\eta_{q})}{1-q\eta_{q}}B_p[u]
				\leq  \begin{cases}
					\frac{\eta_{p}(p\eta_{p}-q\eta_{q})}{1-q\eta_{q}}\mathcal{C}_{G}(p)a^{2p(1-\eta_{p})}A[u]^{p\eta_{p}},&\text{if } p\in (2_{\alpha}^{\sharp},2_{\alpha}^{*}),\\
					\frac{p-q\eta_{q}}{1-q\eta_{q}}\mathcal{S}_{H}^{-p}A[u]^{p},&\text{if } p=2_{\alpha}^{*}.\\
				\end{cases} 
			\end{align*}
			Consequently, it follows from $2_{\alpha}^{\sharp}<p\leq 2_{\alpha}^{*}$ that there exists a $\delta>0$ satisfying $A[u]\geq \delta$ for all $u\in\mathcal{P}^{-}(a)$.
		\end{enumerate}
	\end{proof}
	We define 
	$$S_r(a)=S(a)\cap H^1_{\operatorname{rad}}(\mathbb{R}^N),\quad  \mathcal{P}_{r}(a)=\mathcal{P}(a)\cap H^1_{\operatorname{rad}}(\mathbb{R}^N),\quad \mathcal{P}_{r}^{\pm}(a)=\mathcal{P}^{\pm}(a)\cap H^1_{\operatorname{rad}}(\mathbb{R}^N)$$
	and
	$$I^{\pm}:S(a)\rightarrow\mathbb{R},\quad I^{\pm}[u]:=E[s_u^{\pm}\circ u].$$
	\begin{lemma}\label{lm:inf_Pohozaev_pm}
		Let $N\geq 3$, $\alpha\in (0,N)$, $2_{\alpha}< q< 2^{\sharp}_{\alpha}< p\leq 2_{\alpha}^{*}$ and $\mu>0$.
		Suppose that \eqref{assumptions:basic} holds.
		For any $u\in S(a)$, there exists
		$$
		\gamma^{\pm}(a):=\inf_{u\in\mathcal{P}^{\pm}(a)}E[u]=\inf_{u\in\mathcal{P}^{\pm}_{r}(a)}E[u].
		$$
		Moreover, if $\inf\limits_{u\in\mathcal{P}^{\pm}(a)}E[u]$ is reached by $u_0$, then $u_0$ must be a Schwarz symmetric function.
	\end{lemma}
	\begin{proof}
		 Since $\mathcal{P}_r^{\pm}(a) \subset \mathcal{P}^{\pm}(a)$, it follows that
		$$
		\inf _{u \in \mathcal{P}_r^{\pm}(a)} E[u] \geq \inf _{u \in \mathcal{P}^{\pm}(a)} E[u] .
		$$
		Therefore, it suffices to show that
		$$
		\inf _{u \in \mathcal{P}_r^{\pm}(a)} E[u] \leq \inf _{u \in \mathcal{P}^{\pm}(a)} E[u].
		$$
		By a similar proof in \cite[Lemma 8.1]{MR4096725}, it holds that 
		\begin{align*}
			\inf_{u \in \mathcal{P}^{+}(a)}E[u]=\inf_{u\in S(a)}\min_{s\in[0,s_u^{+}]}E[s\circ u]\quad \text{and}\quad
			 \inf_{u \in \mathcal{P}^{-}(a)}E[u]=\inf_{u\in S(a)}\max_{s\in[s_u^{+},s_u^{-}]}E[s\circ u].
		\end{align*}
		Let $u\in S(a)$ and $v\in S_r(a)$ be the Schwarz rearrangement of $\vert u \vert$, see \cite[Section 3.3]{1999Analysisby}.
		According to the Riesz's rearrangement inequality (see \cite[Section 3.7]{1999Analysisby}), we have
		\begin{align*}
			A[v]\leq A[u],\quad\text{and}\quad B_s[v]\geq B_s[u],\  \forall s\in (2_{\alpha},2_{\alpha}^{*}],
		\end{align*}
		As a consequence, for any $s>0$, it follows that
		\begin{align*}
			E[s\circ v]=\frac{1}{2}sA[v]-\frac{1}{2p}s^{p\eta_{p}}B_p[v]-\frac{\mu}{2q}s^{q\eta_{q}}B_q[v]\leq E[s\circ u].
		\end{align*}
		Besides, it holds that
		\begin{align*}
			f_v^{\prime}(s)\leq f_{u}^{\prime}(s)<0,\quad \forall s>0,
		\end{align*}
		which implies that $0<s_u^{+}\leq s_v^{+}<s_v^{-}\leq s_u^{-}$.
		Therefore, we obtain that
		\begin{align*}
			\min_{s\in[0,s_v^{+}]}E[s\circ v]\leq \min_{s\in[0,s_u^{+}]}E[s\circ u]\quad 
			\text{ and}\quad \max_{s\in[s_v^{+},s_v^{-}]}E[s\circ v]\leq \max_{s\in[s_u^{+},s_u^{-}]}E[s\circ u].
		\end{align*}
		This completes the proof of first part.
		 Now if $u_0 \in \mathcal{P}^{+}(a)$ satisfies $E[u_0]=\inf\limits_{u \in \mathcal{P}^{+}(a)} E(u)$, we find that $v$, the Schwarz rearrangement of $\left|u_0\right|$, belongs to $\mathcal{P}_r^{+}(a)$. 
		 Indeed, if one of the following holds:
		 \begin{align*}
		 	A[v]<A[u_0],\quad B_q[v]>B_q[u_0],\quad B_p[v]>B_p[u_0],
		 \end{align*} 
	 	then we have $E[s\circ v]<E[s\circ u]$. 
	 	Consequently, it follows that
		\begin{align*}
			\inf _{u \in \mathcal{P}^{+}(a)} E[u]=&\inf_{u\in S(a)}\min_{s\in[0,s_u^{+}]}E[s\circ u] \leq \min_{s\in[0,s_v^{+}]}E[s\circ v]
			<\min_{s\in[0,s_{u_0}^{+}]}E[s\circ u_0]=\inf _{u \in \mathcal{P}^{+}(a)} E[u],
		\end{align*}
		which leads to a contradiction. 
		Hence it holds that $A[v]=A[u_0]$, $B_q[v]=B_q[u_0]$ and $B_p[v]=B_p[u_0]$.
		By \cite[Section 3.9]{1999Analysisby}, we obtain that $v\in\mathcal{P}_r^{+}(a)$ and $E[v]=E[u_0]$.
		The other case is similar and this completes the proof.
	\end{proof}
	By a similar proof of \cite[Proposition 2.9]{MR2557725} or \cite[Lemma 3.8]{JEANJEAN2021277}, it follows that
	\begin{lemma}
		Let $N\geq 3$, $\alpha\in (0,N)$, $2_{\alpha}< q< 2^{\sharp}_{\alpha}< p\leq 2_{\alpha}^{*}$ and $\mu,a>0$. 
		Suppose that \eqref{assumptions:basic} is true.
		Then it holds that  $\langle \dif I^{+}[u],\psi\rangle=\langle \dif E[s_u^+\star u],s_u^+\star \psi\rangle$ and $\langle \dif I^{-}[u],\psi\rangle=\langle \dif E[s_u^-\star u],s_u^-\star \psi\rangle$ for any $u \in S(a), \psi \in T_{u} S(a)$.
	\end{lemma}
	Now we can proof the main lemma in this subsection.
	\begin{lemma}\label{lm:PS_sequence_existence_mass_subcritical}
		Let $N\geq 3$, $\alpha\in (0,N)$, $2_{\alpha}< q< 2^{\sharp}_{\alpha}< p\leq 2_{\alpha}^{*}$ and $\mu,a>0$. 
		Suppose that \eqref{assumptions:basic} holds.
		Then there exists a Palais-Smale sequence $\{u_n^+\}_{n\in\mathbb{N}}\subset \mathcal{P}^+(a)$ for $E$ restricted to $S_r(a)$ at level $\gamma^{+}(a)$ and a Palais-Smale sequence $\{u_n^-\}_{n\in\mathbb{N}}\subset \mathcal{P}^-(a)$ for $E$ restricted to $S_r(a)$ at level $\gamma^{-}(a)$.
	\end{lemma}
	\begin{proof}
		Let $\mathcal{F}$ be the class of all singletons belonging to $S_r(a)$, which is a homotopy-stable family of compact subsets of $S_r(a)$ with closed boundary $B=\emptyset$ in the sense of \cite[Definition 3.1]{1993Duality}.
		Set $e^{\pm}(a):=\inf\limits_{D \in \mathcal{F}}\max\limits_{u \in D}I^{\pm}[u]$. 
		It follows from Lemma \ref{lm:inf_Pohozaev_pm} that $e^{\pm}(a)=\gamma^{\pm}(a)$. 
		Let $\{D_n^{\pm}\}\subset \mathcal{F}$ be a sequence satisfying
		\begin{equation*}
			\max_{u\in D_n} I^{\pm}[u]<e^{\pm}(a)+\frac{1}{n},\quad \forall n\in \mathbb{N}.
		\end{equation*}
		Then we define the homotopy mapping
		\begin{equation*}
			\eta^{\pm}:[0,1]\times S(a) \rightarrow S(a),\quad \eta^{\pm}(t,u)=(1-t+ts_u^{\pm})\circ u,
		\end{equation*}
		which implies  
		$$G_{n}^{\pm}:=\eta^{\pm}(\{1\}\times D_n^{\pm})=\{s_u^{\pm}\circ u:u\in D_n^{\pm}\}\in\mathcal{F},\quad \forall n\in \mathbb{N}.$$
		It follows form Lemma \ref{lm:unique_minimum_maximum} that $G_{n}^{\pm}$ is a subset of $\mathcal{P}^{\pm}(a)$ for every $n\in\mathbb{N}$. 
		Since $I^{\pm}[s_u^{\pm}\circ u]=E[1\circ(s_u^{\pm}\circ u)]=E[s_u^{\pm}\circ u]=I^{\pm}[u]$, it holds that $\max\limits_{u\in D_{n}^{\pm}}I^{\pm}[u]=\max\limits_{u\in G_{n}^{\pm}}I^{\pm}[u]$.
		By the min-max theorem \cite[Theorem 3.2]{1993Duality}, there exists a Palais-Smale sequence $\{v_n^{\pm}\}$ at level $e^{\pm}(a)$ in $S_r(a)$ with $\operatorname{dist}(v_n^{\pm},G_{n}^{\pm})\rightarrow 0$ as $n\rightarrow \infty$. 
		If $\{v_n^{\pm}\}\subset \mathcal{P}^{\pm}(a)$, this concludes the proof.
		If not, we put $s_n^{\pm}:=s_{v_n^{\pm}}^{\pm}$ for every $n\in \mathbb{N}$ due to Lemma \ref{lm:unique_minimum_maximum} and consider the sequence $\{u_n^{\pm}:=s_n^{\pm}\circ v_n^{\pm}\}\subset\mathcal{P}^{\pm}(a)$.
		It is enough to prove that $\{u_n^{\pm}\}$ is a Palais-Smale sequence at level $e^{\pm}(a)$ in $S_r(a)$.
		
		\textbf{Claim}: There exists a constant $C>0$ such that $C^{-1}<(s_n^{\pm})^2<C$.
		
		Indeed, it holds that
		\begin{equation*}
			(s_n^{\pm})^2=\frac{A[u_n^{\pm}]}{A[v_n^{\pm}]}.
		\end{equation*}
		We first prove that there exists a constant $M>0$ satisfying
		$$M^{-1}<A[u_n^{\pm}]<M,\quad \forall n\in \mathbb{N}.$$
		Suppose that $\limsup\limits_{n\rightarrow+\infty} A[u_n^{\pm}]=+\infty$.
		According to the fact that $\{u_n^{\pm}\}\subset \mathcal{P}(a)$, it follows from Lemma \ref{lm:coercive} that 
		\begin{equation*}
			\limsup\limits_{n\rightarrow+\infty}E[u_n^{\pm}]=+\infty,
		\end{equation*}
		which contradicts the assumption $$-\infty<\lim\limits_{n\rightarrow\infty}E[u_n^+]=\gamma^{+}(a)<0<\lim\limits_{n\rightarrow\infty}E[u_n^-]=\gamma^{-}(a)<+\infty.$$
		As a result, we have proved that $\{A[u_n^{\pm}]\}$ is uniformly bounded above.
		It follows from Lemma \ref{lm:A_bounded_below} that $\{A[u_n^-]\}$ is uniformly bounded below.
		To prove that $\{A[u_n^{+}]\}$ is uniformly bounded below, we notice that $$E[u_n^+]=I^+[v_n^+]\rightarrow e^+(a)=\gamma^+(a)<0 \quad \text{as}\quad n\rightarrow \infty,$$ which implies that, for $n$ large enough, it holds that
		\begin{align*}
			0>E[u_n]=&(\frac{1}{2}-\frac{1}{2p\eta_{p}})A[u_n^+]-\frac{\mu}{2q}(1-\frac{2q\eta_q}{2p\eta_p})B_q[u_n^+]\\
			\geq &(\frac{1}{2}-\frac{1}{2p\eta_{p}})A[u_n^+]-\frac{\mu}{2q}(1-\frac{2q\eta_q}{2p\eta_p})\mathcal{C}_{G}(q)a^{2q(1-\eta_{q})}A[u_n^+]^{q\eta_{q}}.
		\end{align*}
		Since $2_{\alpha}<q<2_{\alpha}^{\sharp}<p\leq 2_{\alpha}^*$, it is clear that $\{A[u_n^{+}]\}$ is uniformly bounded below.
		
		On the other hand, by Lemma \ref{lm:coercive} and the fact $G_{n}^{\pm} \subset \mathcal{P}^{\pm}(a)$ for every $n\in\mathbb{N}$, it holds that $\{G_{n}\}$ is uniformly bounded above in $H^{1}\left(\mathbb{R}^{N}\right)$.
		It follows from the fact $\operatorname{dist}(v_n^{\pm},G_{n}^{\pm})\rightarrow 0$ as $n\rightarrow \infty$ that  $\limsup \limits_{n\rightarrow +\infty} A[v_n^{\pm}]<\infty$, that is, $\{A[v_n^{\pm}]\}$ is uniformly bounded above.
		Besides, since $G_{n}^{\pm}$ is compact for every $n \in \mathbb{N}$, there exists a function $w_{n}^{\pm} \in G_{n}$ for every $n\in \mathbb{N}$ satisfying $\left\|v_{n}^{\pm}-w_{n}^{\pm}\right\|_{H^{1}\left(\mathbb{R}^{N}\right)} \rightarrow 0$ as $n \rightarrow +\infty$.
		According to the triangle inequality, it follows that there exists a $\delta>0$ such that for $n$ large enough, 
		$$
		A[v_{n}^{\pm}] \geq A[w_n^{\pm}]-A[w_n^{\pm}-v_n^{\pm}] \geq \frac{\delta}{2},
		$$
		that is, $\{A[v_n^{\pm}]\}$ is uniformly bounded from below.
		This completes the claim. 
		
		Consequently, we have
		\begin{align*}
			\|\left.\dif E\right|_{S(a)}[u_n^{\pm}]\|_{H^{-1}(\mathbb{R}^N)}=&\sup\limits_{\|\psi\|\leq 1,\psi\in T_uS(a)}\vert \langle \dif E[u_n^{\pm}],\psi\rangle\vert\\
			=&\sup\limits_{\|\psi\|\leq 1,\psi\in T_uS(a)}\vert \langle \dif E[s_n^{\pm}\circ v_n^{\pm}],s_n^{\pm}\circ( (s_n^{\pm})^{-1}\circ\psi)\rangle\vert\\
			=&\sup\limits_{\|\psi\|\leq 1,\psi\in T_uS(a)}\vert \langle \dif I^{\pm}[ v_n^{\pm}],( s_n^{\pm})^{-1}\circ\psi\rangle\vert\\
			\lesssim&\sup\limits_{\|\psi\|\leq 1,\psi\in T_uS(a)}\|\dif I^{\pm}[ v_n^{\pm}]\|_{H^{-1}(\mathbb{R}^N)}\|\psi\|_{H^{1}(\mathbb{R}^N)}\\
			\leq&\|\dif I^{\pm}[ v_n^{\pm}]\|_{H^{-1}(\mathbb{R}^N)},
		\end{align*}
		where we have used the claim $\{s_n^{\pm}\}$ is bounded below.
		It follows that $\{u_n^{\pm}\}$ is a Palais-Smale sequence at level $\gamma^{\pm}(a)$ in $S_r(a)$ respectively.
	\end{proof}
	\begin{remark}
		It follows from Lemma \ref{lm:coercive} that the Palais-Smale sequence obtained in Lemma \ref{lm:PS_sequence_existence_mass_subcritical} is bounded. 
		In the following part, we always assume this fact holds without a proof.
	\end{remark}
	\subsubsection{The compactness of Palais-Smale sequences under the HLS subcritical leading trem}
	\begin{lemma}\label{lm:nontirvial_weak_limit}
		Let $N\geq 3$, $\alpha\in (0,N)$, $2_{\alpha}< q< 2^{\sharp}_{\alpha}< p< 2_{\alpha}^{*}$ and $\mu,a>0$. 
		Suppose that \eqref{assumptions:basic} holds.
		If either $\{u_n\}\subset \mathcal{P}^{+}(a)$ is a minimizing sequence for $\gamma^{+}(a)$ or
		$\{u_n\}\subset \mathcal{P}^{-}(a)$ is a minimizing sequence for $\gamma^{-}(a)$, it weakly converges,
		up to a subsequence, to a nontrivial limit.
	\end{lemma}
	\begin{proof}
		According to Lemma \ref{lm:coercive}, it follows that $E$ restricted to $\mathcal{P}(a)$ is coercive on $H^1(\mathbb{R}^N)$, which implies that $\{u_n\}$ is bounded.
		Therefore, up to a subsequence, it holds that $u_n \rightharpoonup u_a$ in $H^1({\mathbb{R}^N})$.
		We shall argue by contradiction. 
		Suppose that $u_a=0$, that is, $\{u_n\}$ is vanishing. 
		Using Lion's vanishing lemma (see \cite[Lemma I.1]{MR778974}), we deduce that, for every $q\in (2,2^{*})$,
		\begin{align*}
			\|u_n\|_q \rightarrow 0, \quad\text{as}\quad n \rightarrow +\infty.
		\end{align*}
		As a result, according to the Hardy-Littlewood-Sobolev inequality \eqref{ieq:HLS}, it is clear that,  for every $s\in (2_{\alpha},2_{\alpha}^{*})$,
		\begin{align*}
			0\leq B_s[u_n]\leq \mathcal{C}_{H}(s)\|u\|_{\frac{2Ns}{N+\alpha}}^{2s}\rightarrow 0,\quad \text{as}\quad n\rightarrow +\infty.
		\end{align*} 
		It follows that $\{u_n\}\subset\mathcal{P}(a)$,  which implies that
		\begin{align*}
			A[u_n]=\eta_{p}B_p[u_n]+\mu\eta_{q}B_q[u_n]\rightarrow 0.
		\end{align*}
		If $\{u_n\}\subset \mathcal{P}^{-}(a)$, then by Lemma \ref{lm:derivative_positive}, it is easy to see that  $\{A[u_n]\}$ is bounded below.
		This leads to a contradiction.
		If $\{u_n\}\subset \mathcal{P}^{+}(a)$ is a Palais-Smale sequence at level $\gamma^+(a)<0$, then we have
		\begin{align*}
			\gamma^{+}(a)+o_n(1)=E[u_n]=\frac{1}{2}A[u_n]-\frac{1}{2p}B_p[u_n]-\frac{\mu}{2q}B_q[u_n]\rightarrow 0,\quad \text{as}\quad n\rightarrow +\infty,
		\end{align*}
		which is impossible.
		This completes the proof.
	\end{proof}
	\begin{lemma}\label{lm:convergence_sub}
		Let $N\geq 3$, $\alpha\in (0,N)$, $2_{\alpha}< q< 2^{\sharp}_{\alpha}< p< 2_{\alpha}^{*}$ and $\mu,a>0$. 
		Suppose that \eqref{assumptions:basic} holds.
		If $\{u_n\}\subset \mathcal{P}_{r}(a)$ is a bounded Palais-Smale sequence for $E$ restricted to $S(a)$, then, up to a subsequence, it holds that $u_n \rightarrow u_a$ strongly in $H^1_{\operatorname{rad}}(\mathbb{R}^N)$ with $u_a\in H^{1}_{\operatorname{rad}}(\mathbb{R}^N)\setminus\{0\}$. 
		In particular, $u_a$ is a radial solution to \eqref{eq:NLS_S} for some $\lambda_a<0$ and $\| u_a \|_2=a$.
	\end{lemma}
	\begin{proof}
		Since $\{u_n\}$ is bounded in $H^1_{\operatorname{rad}}(\mathbb{R}^N)$, up to a subsequence, there exists a nontrivial function $u_a\in H^1_{\operatorname{rad}}(\mathbb{R}^N)$ satisfying $u_n \rightharpoonup u_a$ weakly in $H^1_{\operatorname{rad}}(\mathbb{R}^N)$. 
		By Lemma \ref{lm:compact_embedding}, it follows that $u_n\rightarrow u_a$ strongly in $L^{q}(\mathbb{R}^N)$ for all $q\in (2,2^{*})$ and a.e. in $\mathbb{R}^N$. 
		According to the fact that $\{u_n\}\subset S_r(a)$ and \cite[Lemma 3]{MR695536}, it is clear that 
		\begin{align*}
			&\left.\dif E\right\vert_{S(a)}[u_n]\rightarrow 0 \quad\text{in} \quad H^{-1}(\mathbb{R}^N)   
			\iff \dif E[u_n]-\frac{1}{a^2}\langle \dif E[u_n],u_n\rangle u_n \rightarrow 0 \quad\text{in} \quad H^{-1}(\mathbb{R}^N).
		\end{align*}
		Consequently, for any $w\in H^{1}(\mathbb{R}^N)$, we have
		\begin{align}\label{eq:differential_restricted}
			o_n(1)=&\left\langle \dif E[u_n]-\frac{1}{a^2}\langle \dif E[u_n],u_n\rangle u_n,w\right\rangle\notag\\
			=&\int_{\mathbb{R}^N}\nabla u_n\nabla w - \lambda_n \int_{\mathbb{R}^N} u_n w-\int_{\mathbb{R}^N}(I_{\alpha} * \vert u_n\vert^p)\vert u_n\vert^{p-2}u_n w
			-\mu\int_{\mathbb{R}^N}(I_{\alpha} * \vert u_n\vert^q)\vert u_n\vert^{q-2}u_n w,
		\end{align}
		which implies that 
		\begin{align*}
			\lambda_na^2=&
			(\eta_{p}-1)B_p[u_n]-\mu(1-\eta_{q})B_q[u_n]
			\rightarrow (\eta_{p}-1)B_p[u_a]-\mu(1-\eta_{q})B_q[u_a]=:\lambda_aa^2,
		\end{align*}
		where we have used Lemma \ref{lm:compact_operator_nonlocal} and $\{u_n\}\subset \mathcal{P}(a)$. 
		By Lemma \ref{lm:compact_operator_nonlocal}, the weak convergence $u_n\rightharpoonup u_a$ in $H^1_{\operatorname{rad}}(\mathbb{R}^N)$ and $\lambda_n \rightarrow \lambda_a$ in $\mathbb{R}$, the equation \eqref{eq:differential_restricted} leads to
		\begin{equation}\label{eq:differential_criticl_point}
			0=\int_{\mathbb{R}^N}\nabla u_a\nabla w - \lambda_a \int_{\mathbb{R}^N} u_a w-\int_{\mathbb{R}^N}(I_{\alpha} * \vert u_a\vert^p)\vert u_a\vert^{p-2}u_a w
			-\mu\int_{\mathbb{R}^N}(I_{\alpha} * \vert u_a\vert^q)\vert u_a\vert^{q-2}u_a w.
		\end{equation}
		Therefore, it holds that 
		\begin{equation*}
			-\Delta u_a=\lambda_a u_a+ (I_{\alpha} * \vert u_a \vert^p)\vert u_a \vert^{p-2}u_a+\mu (I_{\alpha} * \vert u_a \vert^q)\vert u_a \vert^{q-2}u_a \quad \text{in}\quad H^{-1}(\mathbb{R}^N).
		\end{equation*}
		According to Lemma \ref{lm:Langrange_multiplier}, it follows that $P[u_a]=0$ and $\lambda_a< 0$.
		Putting $w=u_n$ in \eqref{eq:differential_restricted} and $w=u_n$ in \eqref{eq:differential_criticl_point}, we  obtain that
		\begin{align*}
			\int_{\mathbb{R}^N}\vert \nabla u_n\vert^2\dif x -\lambda_n \int_{\mathbb{R}^N} \vert u_n\vert^2\dif x-B_p[u_n]-\mu B_q[u_n]=&o_n(1),\\
			\int_{\mathbb{R}^N}\vert \nabla u_a\vert^2\dif x -\lambda_a \int_{\mathbb{R}^N} \vert u_a\vert^2\dif x-B_p[u_a]-\mu B_q[u_a]=&0.
		\end{align*}
		By Lemma \ref{lm:compact_operator_nonlocal} and $\lambda_n\rightarrow \lambda_a<0$, it holds that 
		\begin{equation*}
			\int_{\mathbb{R}^N}\vert \nabla u_n\vert^2\dif x-\lambda_a \int_{\mathbb{R}^N} \vert u_n\vert^2\dif x\rightarrow\int_{\mathbb{R}^N}\vert \nabla u_a\vert^2\dif x -\lambda_a \int_{\mathbb{R}^N}\vert u_a\vert^2\dif x,\quad \text{as}\quad n\rightarrow \infty,
		\end{equation*}
		which implies that $u_n\rightarrow u_a$ strongly in $H^1_{\operatorname{rad}}(\mathbb{R}^N)$.
		The proof is completed.
	\end{proof}
	Now we can prove Theorem \ref{thm:mass_subcritical} for $p \in (2_{\alpha}^{\sharp},2_{\alpha}^{*})$:
	\begin{proof}[Proof of Theorem \ref{thm:mass_subcritical} for $p \in (2_{\alpha}^{\sharp},2_{\alpha}^{*})$]
		According to Lemma \ref{lm:PS_sequence_existence_mass_subcritical}, it follows that there exist  bounded Palais-Smale sequences $\{u_n^{\pm}\}$ for $\gamma^{\pm}(a)$ respectively.
		Then using Lemma \ref{lm:convergence_sub}, we obtain that $(i)$ and $(ii)$ hold.
		The proof is completed.
	\end{proof}
	\subsubsection{The compactness of Palais-Smale sequences in the HLS critical case}
	In this subsection, we always assume that $p=2_{\alpha}^{*}$.
	\begin{lemma}\label{lm:PS_strong_convergence}
		Set $N\geq 3$, $\alpha\in (0,N)$, $2_{\alpha}<q<p=2^{*}_{\alpha}$ and $\mu,a>0$. 
		Suppose that \eqref{assumptions:basic} holds.
		Let $\{u_n\}\subset \mathcal{P}^+_r(a)$ or $\{u_n\}\subset \mathcal{P}^-_r(a)$ be a Palais-Smale sequence for $E$ restricted to $S(a)$ at levet $m\in\mathbb{R}$ which is weakly convergent, up to a subsequence, to the function $u_a$. If $\{u_n\}\subset \mathcal{P}^+_r(a)$ we assume that $m\ne 0$ and if $\{u_n\}\subset \mathcal{P}^-_r(a)$ we assume that
		\begin{equation*}
			m<\frac{\alpha+2}{2(N+\alpha)}\mathcal{S}^{\frac{N+\alpha}{\alpha+2}}_{H}.
		\end{equation*}
		Then $u_a$ is nontrivial and one of the following alternatives holds:
		\begin{enumerate}
			\item[(i)] $E[u_a]\leq m-\frac{\alpha+2}{2(N+\alpha)}\mathcal{S}^{\frac{N+\alpha}{\alpha+2}}_{H};$
			\item[(ii)] $u_n\rightarrow u_a$ strongly in $H^1_{\operatorname{rad}}(\mathbb{R}^N).$
		\end{enumerate} 
	\end{lemma}
	\begin{proof}
		We first prove that $u_a$ is nontrivial. 
		If the statement was not true, then we have $u_a=0$.
		Since $u_n\rightharpoonup u_a$ weakly in $H^1_{\operatorname{rad}}(\mathbb{R}^N)$, by Lemma \ref{lm:compact_embedding}, it holds that $u_n\rightarrow u_a$ strongly in $L^{r}(\mathbb{R}^N)$ for all $r\in (2,2^*)$, i.e.,  $\|u_n\|_q\rightarrow 0$ as $n\rightarrow \infty$. 
		According to the Hardy-Littlewood-Sobolev inequality \eqref{ieq:HLS} and $\frac{2Nq}{N+\alpha}\in (2,2^*)$, it follows that 
		$$0\leq B_q[u_n]\lesssim\|u_n\|_{\frac{2Nq}{N+\alpha}}^{2q}\rightarrow 0.$$
		Therefore, we deduce form $\{u_n\}\subset \mathcal{P}(a)$  that 
		$$A[u_n]=B_p[u_n]+\mu B_q[u_n]=B_p[u_n]+o_n(1).$$
		As a result, we suppose that  
		$$\lim_{n\rightarrow\infty}A[u_n]=\lim_{n\rightarrow\infty}B_p[u_n]=:l\geq 0.$$
		By the definition of $\mathcal{S}_H$ in \eqref{ieq:best_constant}, it holds that 
		$$l=\lim_{n\rightarrow\infty}A[u_n]\leq \lim_{n\rightarrow\infty}\mathcal{S}_H^{-p}A[u_n]^{p}=\mathcal{S}_H^{-p}l^{p},$$
		which implies that either $l=0$ or $l\geq \mathcal{S}_H^{\frac{p}{p-1}}=\mathcal{S}^{\frac{N+\alpha}{\alpha+2}}_{H}>0.$
		If $\{u_n\}\subset \mathcal{P}^+_r(a)$, then we have $l=0$ since
		$$0\leq(p-1)B_p[u_n]<\mu\eta_{q}(1-q\eta_{q})B_q[u_n]\rightarrow 0,\quad \text{as}\quad  n\rightarrow +\infty.$$
		However, the above inequality also implies that $E[u_n]=o_n(1)$ and this contradicts the assumption that $E[u_n]+o_n(1)=m\ne 0$. 
		If $\{u_n\}\subset \mathcal{P}^-_r(a)$, Lemma \ref{lm:A_bounded_below} implies that $l\geq \mathcal{S}^{\frac{N+\alpha}{\alpha+2}}_{H}>0.$
		It follows from $\{u_n\}\subset \mathcal{P}(a)$ that
		\begin{align*}
			m+o_n(1)=&E[u_n]=\frac{p-1}{2p}A[u_n]-\frac{\mu}{2q}(1-\frac{2q\eta_q}{2p})B_q[u_n]
			=\frac{\alpha+2}{2(N+\alpha)}l+o_n(1)\geq \frac{\alpha+2}{2(N+\alpha)}\mathcal{S}^{\frac{N+\alpha}{\alpha+2}}_{H}+o_n(1).
		\end{align*}
		This leads to a contradiction.
		
		It follows from  $\{u_n\}\subset S_r(a)$ and \cite[Lemma 3]{MR695536}  that 
		\begin{align*}
			&\left.\dif E\right\vert_{S(a)}[u_n]\rightarrow 0 \text{ in } H^{-1}(\mathbb{R}^N)   
			\iff \dif E[u_n]-\frac{1}{a^2}\langle \dif E[u_n],u_n\rangle u_n \rightarrow 0 \text{ in } H^{-1}(\mathbb{R}^N).
		\end{align*}
		Consequently, for any $w\in H^{1}(\mathbb{R}^N)$, we have
		\begin{align}\label{eq:differential_restricted_critical}
			o_n(1)=&\left\langle \dif E[u_n]-\frac{1}{a^2}\langle \dif E[u_n],u_n\rangle u_n,w\right\rangle\notag\\
			=&\int_{\mathbb{R}^N}\nabla u_n\nabla w - \lambda_n \int_{\mathbb{R}^N} u_n w-\int_{\mathbb{R}^N}(I_{\alpha} * \vert u_n\vert^p)\vert u_n\vert^{p-2}u_n w
			-\mu\int_{\mathbb{R}^N}(I_{\alpha} * \vert u_n\vert^q)\vert u_n\vert^{q-2}u_n w,
		\end{align}
		which implies that 
		\begin{align*}
			\lambda_na^2=&-\mu(1-\eta_{q})B_q[u_n]
			\rightarrow -\mu(1-\eta_{q})B_q[u_a]=:\lambda_aa^2,
		\end{align*}
		where we have used Lemma \ref{lm:compact_operator_nonlocal} and $P[u_n]=0$ for every $n\in \mathbb{N}$. 
		By Lemma \ref{lm:compact_operator_nonlocal}, Lemma \ref{lm:Brezis–Lieb_nonlocal}, the weak convergence $u_n\rightharpoonup u_a$ in $H^1_{\operatorname{rad}}(\mathbb{R}^N)$ and $\lambda_n \rightarrow \lambda_a$ in $\mathbb{R}$, the equation \eqref{eq:differential_restricted_critical} leads to
		\begin{align}\label{eq:differential_criticl_point_critical}
			0=\int_{\mathbb{R}^N}\nabla u_a\nabla w\dif x -& \lambda_a \int_{\mathbb{R}^N} u_a w\dif x-\int_{\mathbb{R}^N}(I_{\alpha} * \vert u_a\vert^p)\vert u_a\vert^{p-2}u_a w\dif x
			-\mu\int_{\mathbb{R}^N}(I_{\alpha} * \vert u_a\vert^q)\vert u_a\vert^{q-2}u_a w\dif x.
		\end{align}
		As a consequence, we have 
		\begin{align*}
			-\Delta u_a=\lambda_a u_a+ (I_{\alpha} * \vert u_a \vert^p)\vert u_a \vert^{p-2}u_a+\mu (I_{\alpha} * \vert u_a \vert^q)\vert u_a \vert^{q-2}u_a \quad\text{in}\  H^{-1}(\mathbb{R}^N).
		\end{align*}
		Therefore, it holds that $P[u_a]=0$ and $\lambda_a< 0$ due to Lemma \ref{lm:Langrange_multiplier}.
		
		Note that $v_n:=u_n-u_a\rightharpoonup 0$ weakly in $H^1_{\operatorname{rad}}(\mathbb{R}^N)$. 
		Using Brezis-Lieb type lemma for nonlocal term (see Lemma \ref{lm:Brezis–Lieb_nonlocal}), we have 
		$$A[u_n]=A[u_a]+A[v_n]+o_n(1),\quad B_s[u_n]=B_s[u_a]+B_s[v_n]+o_n(1),\quad \forall s\in (2_{\alpha},2_{\alpha}^{*}].$$
		Therefore, according to $P[u_n]=0$ and Lemma \ref{lm:compact_operator_nonlocal}, it follows that 
		$$ A[u_a]+A[v_n]-B_p[u_a]-B_p[v_n]-\mu\eta_qB_q[u_a]=o_n(1).$$
		By $P[u_a]=0$, we can suppose that
		$$\lim_{n\rightarrow\infty} A[v_n]=\lim_{n\rightarrow\infty}B_p[v_n]=:k\geq 0.$$
		Again, by the definition of $\mathcal{S}_H$ in \eqref{ieq:best_constant}, it holds that either $k=0$ or $k\geq S^{\frac{N+\alpha}{\alpha+2}}_{H}>0.$
		
		If $k\geq S^{\frac{N+\alpha}{\alpha+2}}_{H}>0$ holds, then we have
		\begin{align*}
			m+o_n(1)=E[u_n]=&(\frac{1}{2}-\frac{1}{2p})A[u_n]-\frac{\mu}{2q}(1-\frac{2q\eta_q}{2p})B_q[u_n]\\
			=&(\frac{1}{2}-\frac{1}{2p})(A[u_a]+A[v_n])-\frac{\mu}{2q}(1-\frac{2q\eta_q}{2p})B_q[u_a]+o_n(1)\\
			=&E[u_a]+(\frac{1}{2}-\frac{1}{2p})k+o_n(1)\\
			\geq& E[u_a]+\frac{\alpha+2}{2(N+\alpha)}\mathcal{S}^{\frac{N+\alpha}{\alpha+2}}_{H}+o_n(1),
		\end{align*}
		which proves that the alternative (i) holds.
		
		If $k=0$ holds, then  putting $w=u_n$ in \eqref{eq:differential_restricted_critical} and $w=u_n$ in \eqref{eq:differential_criticl_point_critical}, we arrive at
		\begin{align*}
			\int_{\mathbb{R}^N}\vert \nabla u_n\vert^2\dif x -\lambda_n \int_{\mathbb{R}^N} \vert u_n\vert^2\dif x-B_p[u_n]-\mu B_q[u_n]=&o_n(1),\\
			\int_{\mathbb{R}^N}\vert \nabla u_a\vert^2\dif x -\lambda_a \int_{\mathbb{R}^N} \vert u_a\vert^2\dif x-B_p[u_a]-\mu B_q[u_a]=&0.
		\end{align*}
		By Lemma \ref{lm:compact_operator_nonlocal}, Lemma \ref{lm:Brezis–Lieb_nonlocal},  $\lambda_n\rightarrow \lambda_a<0$ and $k=0$, it follows that 
		\begin{align*}
			\int_{\mathbb{R}^N}\vert \nabla u_n\vert^2\dif x-\lambda_a \int_{\mathbb{R}^N} \vert u_n\vert^2\dif x\rightarrow\int_{\mathbb{R}^N}\vert \nabla u_a\vert^2\dif x -\lambda_a \int_{\mathbb{R}^N}\vert u_a\vert^2\dif x,\quad \text{as}\quad n\rightarrow \infty,
		\end{align*}
		which implies that $u_n\rightarrow u_a$ strongly in $H^1_{\operatorname{rad}}(\mathbb{R}^N)$, that is, the alternative $(ii)$ holds.
		This completes the proof.
	\end{proof}
	A direct application is to obtain the existence of one ground state in Theorem \ref{thm:mass_subcritical}(i):
	\begin{proof}[Proof of Theorem \ref{thm:mass_subcritical}(i) for $p=2_{\alpha}^{*}$]
		By Lemma \ref{lm:PS_sequence_existence_mass_subcritical}, there exists a bounded Palais-Smale sequence $\{u_n\}\subset \mathcal{P}^{+}(a)$ for $\left.E\right\vert_{S(a)}$ at level $\gamma^{+}(a)<0$.
		Up to a subsequence, there exists a nontrivial $u_a \in H^1(\mathbb{R}^N)$  satisfying $u_n\rightharpoonup u_a$ weakly in $H^1(\mathbb{R}^N)$. 
		By Lemma \ref{lm:PS_strong_convergence} and Lemma \ref{lm:nonincreasing} below, we obtain that $\gamma^{+}(a)$ is reached.
		The remain part is similar to the proof for $p\in (2_{\alpha}^{\sharp},2_{\alpha}^*)$ in part 3.1.2.
	\end{proof}
	Now we shall consider the second critical point for $\gamma^{-}(a)$. 
	Our proof is inspired by \cite[Lemma 3.1]{MR4433054} and \cite[Lemma 3.14]{JEANJEAN2021277}.
	Let $u_{\varepsilon}$ be an extremal function for the Sobolev inequality in $\mathbb{R}^{N}$ defined by
	$$
	u_{\varepsilon}(x):=\frac{\left[N(N-2) \varepsilon^{2}\right]^{\frac{N-2}{4}}}{\left[\varepsilon^{2}+|x|^{2}\right]^{\frac{N-2}{2}}}, \quad \varepsilon>0, \quad x \in \mathbb{R}^{N}
	$$
	Let $\xi \in \mathscr{C}_{0}^{\infty}\left(\mathbb{R}^{N}\right)$ be a radial non-increasing cut-off function with $\xi \equiv 1$ in $B_{1}, \xi \equiv 0$ in $\mathbb{R}^{N} \backslash B_{2}$.  Put $U_{\varepsilon}(x)=\xi(x) u_{\varepsilon}(x)$.
	We recall the following result, see \cite{2021Jeanjean}[Lemma 7.1].
	\begin{lemma}\label{lm:estimate_bubble}
		Let $N \geq 3$ and $\omega$ be the area of the unit sphere in $\mathbb{R}^{N}$. Then it holds that
		\begin{enumerate}
			\item[(i)] 
			$$
			\left\|\nabla U_{\varepsilon}\right\|_2=\mathcal{S}^{\frac{N}{2}}+O\left(\varepsilon^{N-2}\right) \quad \text {and}\quad \left\|U_{\varepsilon}\right\|_{2^{*}}^{2^{*}}=\mathcal{S}^{\frac{N}{2}}+O\left(\varepsilon^{N}\right).
			$$
			\item[(ii)] For some positive constant $K>0$,
			$$
			\left\|U_{\varepsilon}\right\|_q^q= 
			\begin{cases}
				K \varepsilon^{N-\frac{(N-2)}{2} q}+o\left(\varepsilon^{N-\frac{(N-2)}{2} q}\right), & \text { if } q\in(\frac{N}{N-2},2^*),\\ 
				\omega \varepsilon^{\frac{N}{2}}|\log \varepsilon|+O\left(\varepsilon^{\frac{N}{2}}\right), & \text { if } q=\frac{N}{N-2}, \\
				\omega\left(\int_{0}^{2} \frac{\xi^{q}(r)}{r^{(N-2) q-(N-1)}} \right) \varepsilon^{\frac{q(N-2)}{2}}+o(\varepsilon^{\frac{q(N-2)}{2}}), & \text { if }q\in[1,\frac{N}{N-2}) .
			\end{cases}
			$$
		\end{enumerate}
	\end{lemma}
	The proof of following lemma is similar to \cite{MR3582271, MR3817173} .
	\begin{lemma}\label{lm:estimate_bubble_nonlocal}
		Let $N\geq 3$ and $\alpha \in (0,N)$. Then 
		\begin{enumerate}
			\item[(i)] it holds that
			\begin{align*}
				B_{2_{\alpha}^*}[U_{\epsilon}]
				\begin{cases}
					\geq \left(\mathcal{A}(N,\alpha)\mathcal{C}_{H}(N,\alpha)\right)^{\frac{N}{2}}\mathcal{S}_{H}^{\frac{N+\alpha}{2}}-O(\epsilon^{\frac{N+\alpha}{2}}),\\
					\leq \left(\mathcal{A}(N,\alpha)\mathcal{C}_{H}(N,\alpha)\right)^{\frac{N}{2}}\mathcal{S}_{H}^{\frac{N+\alpha}{2}}+O(\epsilon^{N-2}).
				\end{cases} 
			\end{align*}
			\item[(ii)]
			For $q\in (2_{\alpha},2_{\alpha}^{*})$, it is true that
			\begin{align*}
				B_q[U_{\epsilon}]\geq O(\epsilon^{N+\alpha-(N-2)q})=O(\epsilon^{2q(1-\eta_{q})}).
			\end{align*}
		\end{enumerate}
	\end{lemma}
	\begin{lemma}\label{lm:q_bounded}
		There exist two constants $C_1,C_2>0$ such that for any $p,q\in[1,+\infty)$,  
		\begin{align*}
			C_1\|U_{\varepsilon}\|_q^q\leq \int_{\mathbb{R}^N}\vert u_a^+(x)\vert^p\vert U_{\varepsilon}(x)\vert^q\dif x\leq C_2\|U_{\varepsilon}\|_q^q.
		\end{align*}
	\end{lemma}
	\begin{proof}
		On one hand, since $u_{a}^{+}$ is bounded, it follows  that
		$$
		\int_{\mathbb{R}^{N}}\left|u_{a}^{+}(x)\right|^{p}\left|U_{\varepsilon}(x)\right|^{q} \dif x \leq\left\|u_{a}^{+}\right\|_{\infty}^{p} \|U_{\varepsilon}\|_q^q.
		$$
		On the other hand, since $u_{a}^{+}>0$ on $\mathbb{R}^{N}$ is continuous and the function $U_{\varepsilon}$ is compactly support in $B_{2}$, it holds that
		$$
		\begin{aligned}
			\int_{\mathbb{R}^{N}}\left|u_{a}^{+}(x)\right|^{p}\left|U_{\varepsilon}(x)\right|^{q} \dif x =\int_{B_{2}}\left|u_{a}^{+}(x)\right|^{p}\left|U_{\varepsilon}(x)\right|^{q} \dif x 
			\geq \min _{x \in B_{2}}\left|u_{a}^{+}(x)\right|^{p} \int_{B_{2}}\left|U_{\varepsilon}(x)\right|^{q} \dif x 
			=\min _{x \in B_{2}}\left|u_{a}^{+}(x)\right|^{p} \|U_{\varepsilon}\|_q^q.
		\end{aligned}
		$$
		The proof is complete.
	\end{proof}
	For any $\epsilon>0$ and any $t>0$, define $w_{\epsilon,t}=u_a^++tU_{\epsilon}$, we have
	\begin{align*}
		\|w_{\epsilon,t}\|_2^2=&\|u_a^++tU_{\epsilon}\|_2^2=\|u_a\|_2^2+\|tU_{\epsilon}\|_2^2+2\int_{\mathbb{R}^N}u_a^+(x)(tU_{\epsilon}(x))\dif x,\\
		A[w_{\epsilon,t}]=&A[u_a^++tU_{\epsilon}]=A[u_a^+]+A[tU_{\epsilon}]+2\int_{\mathbb{R}^N}\nabla u_a^+(x)\cdot \nabla(tU_{\epsilon}(x))\dif x,\\
		B_q[w_{\epsilon,t}]=&B_q[u_a^++tU_{\epsilon}]=\mathcal{A}(N,\alpha)\int_{\mathbb{R}^N}\int_{\mathbb{R}^N}\frac{\vert u_a^+(x)+tU_{\epsilon}(x)\vert^q \vert u_a^+(y)+tU_{\epsilon}(y)\vert^q}{\vert x-y\vert^{N-\alpha}}\dif x\dif y\\
		\geq&\begin{cases}
			B_q[u_a^+]+B_q[tU_{\epsilon}],\quad &\text{if } q>1,\\
			B_q[u_a^+]+B_q[tU_{\epsilon}]+2q\int_{\mathbb{R}^N}(I_{\alpha}*\vert u^+_a\vert^q)\vert u^+_a\vert^{q-2}u^+_a(tU_{\epsilon})\dif x, \quad &\text{if } q\geq 2,\\
			B_q[u_a^+]+B_q[tU_{\epsilon}]+2q\int_{\mathbb{R}^N}(I_{\alpha}*\vert u^+_a\vert^q)\vert u^+_a\vert^{q-2}u^+_a(tU_{\epsilon})\dif x \\
			+2q\int_{\mathbb{R}^N}(I_{\alpha}*\vert tU_{\epsilon}\vert^q)\vert tU_{\epsilon}\vert^{q-2}tU_{\epsilon}(u_a^+)\dif x, \quad &\text{if } q\geq 3.
		\end{cases}
	\end{align*}
	The last inequality follows by the following fact:
	\begin{align*}
		(a+b)^r\geq 
		\begin{cases}
			a^r+b^r,&\text{if } r\geq 1,\\
			a^r+b^r+rab^{r-1},&\text{if } r\geq 2,\\
			a^r+b^r+rab^{r-1}+ra^{r-1}b,&\text{if } r\geq 3,
		\end{cases}
	\end{align*}
	where $a,b\geq 0$.
	Since $u_a^+$ is a solution of the following equation
	\begin{equation*}
		-\Delta u=\lambda_a^+ u+ (I_{\alpha} * \vert u \vert^p)\vert u \vert^{p-2}u+\mu (I_{\alpha} * \vert u \vert^q)\vert u \vert^{q-2}u\quad\text{ in }\mathbb{R}^N,
	\end{equation*}
	by testing above equation via $tU_{\epsilon}$, it holds that
	\begin{align}\label{eq:differential_criticl_point_plus}
		\lambda_a^+ \int_{\mathbb{R}^N} u_a^+ (tU_{\epsilon})+\mu\int_{\mathbb{R}^N}(I_{\alpha} * \vert u_a^+\vert^q)\vert u_a^+\vert^{q-2}u_a^+ (tU_{\epsilon})
		=\int_{\mathbb{R}^N}\nabla u_a^+\nabla (tU_{\epsilon})  -\int_{\mathbb{R}^N}(I_{\alpha} * \vert u_a^+\vert^p)\vert u_a^+\vert^{p-2}u_a^+ (tU_{\epsilon}).
	\end{align}
	Put $\overline{w}_{\epsilon,t}(x)=\theta^{\frac{N-2}{2}}w_{\epsilon,t}(\theta x)$ and it is clear that
	\begin{align*}
		\|\overline{w}_{\epsilon,t}\|_2^2=\theta^{-2}\|w_{\epsilon,t}\|_2^2,\quad A[\overline{w}_{\epsilon,t}]=A[w_{\epsilon,t}],\quad B_q[\overline{w}_{\epsilon,t}]=\theta^{q(N-2)-N-\alpha}B_q[w_{\epsilon,t}].
	\end{align*}
	In particular, $B_p[\overline{w}_{\epsilon,t}]=B_p[w_{\epsilon,t}]$ with $p=2_{\alpha}^*$ and choose $\theta =\frac{\|w_{\epsilon,t}\|_2}{a}$, then it follows that $\|\overline{w}_{\epsilon,t}\|_2=a$, i.e., 
	$\overline{w}_{\epsilon,t}\in S(a)$.
	By Lemma \ref{lm:unique_minimum_maximum}, there exists a $s_{\epsilon,t}^->0$ such that $s_{\epsilon,t}^-\circ\overline{w}_{\epsilon,t}\in\mathcal{P}^{-}(a)$.
	It follows that
	\begin{align*}
		A[s_{\epsilon,t}^- \circ \overline{w}_{\epsilon,t}]=B_p[s_{\epsilon,t}^- \circ \overline{w}_{\epsilon,t}]+\mu\eta_{q}B_q[s_{\epsilon,t}^- \circ \overline{w}_{\epsilon,t}],
	\end{align*}
	or equivalently, 
	\begin{align*}
		(s_{\epsilon,t}^-)^2A[  \overline{w}_{\epsilon,t}]=(s_{\epsilon,t}^-)^{2p}B_p[  \overline{w}_{\epsilon,t}]+\mu\eta_{q}(s_{\epsilon,t}^-)^{2q\eta_{q}}B_q[ \overline{w}_{\epsilon,t}],
	\end{align*}
	which implies that 
	\begin{align*}
		A[\overline{w}_{\epsilon,t}]>(s_{\epsilon,t}^-)^{2p-2}B_p[  \overline{w}_{\epsilon,t}].
	\end{align*}
	However, according to the H\"{o}lder inequality, Lemma \ref{lm:estimate_bubble} and Lemma \ref{lm:estimate_bubble_nonlocal}, it follows that
	\begin{align*}
		A[\overline{w}_{\epsilon,t}]=&A[u_a^+]+A[tU_{\epsilon}]+2\int_{\mathbb{R}^N}\nabla u_a^+(x)\cdot \nabla(tU_{\epsilon}(x))\dif x\\
		\leq & A[u_a^+]+t^2A[U_{\epsilon}]+2tA[u_a^+]^{\frac{1}{2}}A[U_{\epsilon}]^{\frac{1}{2}}\\
		\rightarrow&  A[u_a^+]+2\mathcal{S}^{\frac{N}{2}}A[u_a^+]^{\frac{1}{2}}t+\mathcal{S}^Nt^2\quad\text{ as}\quad\epsilon\rightarrow 0,
	\end{align*}
	and 
	\begin{align*}
		B_p[\overline{w}_{\epsilon,t}]\geq& B_p[u_a^+]+B_p[tU_{\epsilon}]\geq t^{2p}B_p[U_{\epsilon}]
		\geq  t^{2p}\mathcal{A}(N,\alpha)^{\frac{N}{2}}\mathcal{C}_{H}({N,\alpha})^{\frac{N}{2}}\mathcal{S}_H^{\frac{N+\alpha}{2}}-O(\epsilon^{\frac{N+\alpha}{2}}).
	\end{align*}
	Since $p=2_{\alpha}^*>1$, we find that for $\epsilon>0$ small enough, it holds that 
	\begin{align*}
		A[u_a^+]+2\mathcal{S}^{\frac{N}{2}}A[u_a^+]^{\frac{1}{2}}t+\mathcal{S}^Nt^2\geq (s_{\epsilon,t}^-)^{2p-2}t^{2p}\mathcal{A}(N,\alpha)^{\frac{N}{2}}\mathcal{C}_{H}({N,\alpha})^{\frac{N}{2}}\mathcal{S}_H^{\frac{N+\alpha}{2}},
	\end{align*}
	which implies that $s_{\epsilon,t}^-\rightarrow 0$ as $t\rightarrow +\infty$.
	Note that $\overline{w}_{\epsilon,0}=u_a^+\in\mathcal{P}^+(a)$. 
	It follows from Lemma \ref{lm:unique_minimum_maximum}  that $s_{\epsilon,0}^->1$ and hence there exists $t_{\epsilon}>0$ such that $s_{\epsilon,t_{\epsilon}}^-=1$ and $\overline{w}_{\epsilon,t_{\epsilon}}\in\mathcal{P}^-(a)$.
	Consequently, by the definition of $\gamma^-(a)$, it is clear that
	\begin{align*}
		\sup_{t\geq 0}E[\overline{w}_{\epsilon,t}]\geq E[\overline{w}_{\epsilon,t_{\epsilon}}]\geq \gamma^-(a).
	\end{align*}
	Indeed, we have proved the following lemma:
	\begin{lemma}
		For $\epsilon>0$ small enough, it holds that $\gamma^-(a)\leq\sup\limits_{t\geq 0}E[\overline{w}_{\epsilon,t}]$.
	\end{lemma}
	Next, we will prove a key lemma.
	\begin{lemma}\label{lm:energy_critical_sub}
		Let $N\geq 3$, $2_{\alpha}<q<2_{\alpha}^{\sharp}<p=2_{\alpha}^{*}$ and $\mu,a>0$. 
		Assume that \eqref{assumptions:basic} is true and one of the following conditions holds:
		\begin{enumerate}
			\item[(i)] $N\geq 6$, $0<\alpha<N-2$ and $\max\{\frac{2N-2+\alpha}{2N-4},\frac{N+\alpha-2}{N-2}\}<q<2_{\alpha}^{\sharp}$;
			\item[(ii)] $N=5$, $0<\alpha<3$ and $\max\{\frac{8\alpha}{6},\frac{7+2\alpha}{6}\}<q<2_{\alpha}^{\sharp}$;
			\item[(iii)] $3\leq N\leq 5$, $\max\{N-2, 2N-6\}\leq\alpha<N$ and $2\leq q<2_{\alpha}^{\sharp}$.
		\end{enumerate}
		Then it holds that $\gamma^-(a)<\gamma^+(a)+\frac{\alpha+2}{2(N+\alpha)}\mathcal{S}^{\frac{N+\alpha}{\alpha+2}}_{H}.$
	\end{lemma}
	\begin{proof}
		A direct calculation shows that
		\begin{align*}
			E[\overline{w}_{\epsilon,t}]=&\frac{1}{2}A[w_{\epsilon,t}]-\frac{1}{2p}B_p[w_{\epsilon,t}]-\theta^{q(N-2)-N-\alpha}\frac{\mu}{2q}B_q[w_{\epsilon,t}]\\
			\leq&\frac{1}{2}\left(A[u_a^+]+A[tU_{\epsilon}]+2\int_{\mathbb{R}^N}\nabla u_a^+(x) \nabla(tU_{\epsilon}(x))\dif x\right)-\frac{\mu\theta^{-2q(1-\eta_{q})}}{2q}B_q[u_a^+]-\frac{1}{2p}\left(B_p[u_a^+]+B_p[tU_{\epsilon}]\right)\\
			\leq&E[u_a^+]+\frac{\mu(1-\theta^{-2q(1-\eta_{q})})}{2q}B_q[u_a^+]-\frac{t^{2p}}{2p}B_p[U_{\epsilon}]+tA[u_a^+]^{\frac{1}{2}}A[U_{\epsilon}]^{\frac{1}{2}}+\frac{t^2A[U_{\epsilon}]}{2}=:I(t).
		\end{align*}
		By Lemma \ref{lm:estimate_bubble} and Lemma \ref{lm:estimate_bubble_nonlocal}, it follows that, for $\epsilon>0$ small, 
		\begin{align*}
			&-\frac{t^{2p}\mathcal{A}(N,\alpha)^{\frac{N}{2}}\mathcal{C}_{H}({N,\alpha})^{\frac{N}{2}}\mathcal{S}_H^{\frac{N+\alpha}{2}}}{2p}+tA[u_a^+]^{\frac{1}{2}}\mathcal{S}^{\frac{N}{2}}+\frac{t^2\mathcal{S}^N}{2}\\
			\leq& -\frac{t^{2p}B_p[U_{\epsilon}]}{2p}+tA[u_a^+]^{\frac{1}{2}}A[U_{\epsilon}]^{\frac{1}{2}}
			+\frac{t^2A[U_{\epsilon}]}{2}\\
			\leq&-\frac{t^{2p}\mathcal{A}(N,\alpha)^{\frac{N}{2}}\mathcal{C}_{H}({N,\alpha})^{\frac{N}{2}}\mathcal{S}_H^{\frac{N+\alpha}{2}}}{2p}+tA[u_a^+]^{\frac{1}{2}}\mathcal{S}^{\frac{N}{2}}+\frac{t^2\mathcal{S}^N}{2},
		\end{align*}
		which implies that $I(t)\rightarrow -\infty$ as $t\rightarrow +\infty$ and $I(t)\rightarrow E[u_a^+]$ as $t\rightarrow 0^{+}$ due to $\theta\rightarrow 1$.
		Therefore, there exists $\epsilon_0>0$ and $0<t_0<t_1<+\infty$ such that
		\begin{align*}
			E[\overline{w}_{\epsilon,t}]<\gamma^+(a)+\frac{\alpha+2}{2(N+\alpha)}\mathcal{S}^{\frac{N+\alpha}{\alpha+2}}_{H}
		\end{align*}
		for $ t\in (0,t_0)\cup(t_1,+\infty)$ and $\epsilon\in (0,\epsilon_0]$.
		It remains to consider the case $t\in [t_0,t_1]$, that is, $t$ is bounded.
		Similarly, a direct calculation gives that
		\begin{align}\label{estimate:energy_subcritical}
			E[\overline{w}_{\epsilon,t}]=&\frac{1}{2}A[w_{\epsilon,t}]-\frac{1}{2p}B_p[w_{\epsilon,t}]-\theta^{q(N-2)-N-\alpha}\frac{\mu}{2q}B_q[w_{\epsilon,t}]\notag\\
			\leq&\frac{1}{2}\left(A[u_a^+]+A[tU_{\epsilon}]+2\int_{\mathbb{R}^N}\nabla u_a^+(x) \nabla(tU_{\epsilon}(x))\dif x\right)\notag\\
			&-\frac{\mu}{2q}\left(B_q[u_a^+]+B_q[tU_{\epsilon}]\right)+(1-\theta^{-2q(1-\eta_{q})})\frac{\mu}{2q}\left(B_q[u_a^+]+B_q[tU_{\epsilon}]\right)\notag\\
			&-\frac{1}{2p}\left(B_p[u_a^+]+B_p[tU_{\epsilon}]+2p\int_{\mathbb{R}^N}(I_{\alpha}*\vert u^+_a\vert^p)\vert u^+_a\vert^{p-2}u^+_a(tU_{\epsilon})\dif x\right)\notag\\
			=&E[u_a^+]+E[tU_{\epsilon}]+(1-\theta^{-2q(1-\eta_{q})})\frac{\mu}{2q}\left(B_q[u_a^+]+B_q[tU_{\epsilon}]\right)\notag\\
			&+\int_{\mathbb{R}^N}\nabla u_a^+(x)\cdot \nabla(tU_{\epsilon}(x))\dif x
			-\int_{\mathbb{R}^N}(I_{\alpha}*\vert u^+_a\vert^p)\vert u^+_a\vert^{p-2}u^+_a(tU_{\epsilon})\dif x\notag\\
			=&E[u_a^+]+E[tU_{\epsilon}]+(1-\theta^{-2q(1-\eta_{q})})\frac{\mu}{2q}\left(B_q[u_a^+]+B_q[tU_{\epsilon}]\right)\notag \\
			&+\lambda_a^+ \int_{\mathbb{R}^N} u_a^+ (tU_{\epsilon})\dif x+\mu\int_{\mathbb{R}^N}(I_{\alpha} * \vert u_a^+\vert^q)\vert u_a^+\vert^{q-2}u_a^+ (tU_{\epsilon})\dif x,
		\end{align}
		where the last equality comes from \eqref{eq:differential_criticl_point_plus}. 
		It follows from Lemma \ref{lm:estimate_bubble} that
		\begin{align*}
			\theta^2=&\frac{\|w_{\epsilon,t}\|_2^2}{a^2}=1+\frac{t^2}{a^2}\|U_{\epsilon}\|_2^2+\frac{2}{a^2}\int_{\mathbb{R}^N}u_a^+(x)(tU_{\epsilon}(x))\dif x\\
			=&1+\frac{2}{a^2}\int_{\mathbb{R}^N}u_a^+(x)(tU_{\epsilon}(x))\dif x+
			\begin{cases}
				O(\epsilon^2),&\text{if } N\geq 5,\\
				O(\epsilon^2\vert \log\epsilon\vert),&\text{if } N=4,\\
				O(\epsilon),&\text{if } N=3.
			\end{cases}
		\end{align*}
		According to Lemma \ref{lm:estimate_bubble} and Lemma \ref{lm:q_bounded}, it is clear that
		\begin{align*}
			\int_{\mathbb{R}^N}u_a^+(x)(tU_{\epsilon}(x))\dif x\sim \|U_{\epsilon}\|_1=O(\epsilon^{\frac{N-2}{2}}).
		\end{align*}
		Using Talor's expansion, it holds that 
		\begin{align*}
			1-\theta^{-2q(1-\eta_{q})}=&1-\left(1+\frac{2}{a^2}\int_{\mathbb{R}^N}u_a^+(x)(tU_{\epsilon}(x))\dif x+\frac{t^2}{a^2}\|U_{\epsilon}\|_2^2\right)^{-q(1-\eta_{q})}\\
			=&\frac{2q(1-\eta_{q})}{a^2}\int_{\mathbb{R}^N}u_a^+(x)(tU_{\epsilon}(x))\dif x+\begin{cases}
				O(\epsilon^2),&\text{if } N\geq 5,\\
				O(\epsilon^2\vert \log\epsilon\vert),&\text{if } N=4,\\
				O(\epsilon),&\text{if } N=3.
			\end{cases}
		\end{align*}
		By the Hardy-Littlewood-Sobolev inequality \eqref{ieq:HLS}, Lemma \ref{lm:estimate_bubble} and Lemma \ref{lm:estimate_bubble_nonlocal}, we have
		\begin{align*}
			\int_{\mathbb{R}^N}(I_{\alpha} * \vert u_a^+\vert^q)\vert u_a^+\vert^{q-2}u_a^+ (tU_{\epsilon})\dif x\lesssim&\|\vert u_a^+\vert ^q\|_{\frac{2N}{N+\alpha}}\|\vert u_a^+\vert ^{q-1}(tU_{\epsilon})\|_{\frac{2N}{N+\alpha}}\\
			=&\| u_a^+\|_{\frac{2Nq}{N+\alpha}}^q\|\vert u_a^+\vert ^{q-1}(tU_{\epsilon})\|_{\frac{2N}{N+\alpha}}
			\lesssim \|U_{\epsilon}\|_{\frac{2N}{N+\alpha}}\\
			=&\begin{cases}
				O(\epsilon^{\frac{\alpha+2}{2}}),&\text{if }0<\alpha<(N-4)_{+},\\
				O(\epsilon^{\frac{N+\alpha}{4}}\vert\log\epsilon\vert^{\frac{N+\alpha}{2N}}),&\text{if }\alpha=(N-4)_{+},\\
				O(\epsilon^{\frac{N-2}{2}}),&\text{if }(N-4)_{+}<\alpha<N.\\
			\end{cases}
		\end{align*}
		Note that $P[u_a]=0$ and hence we arrive at
		\begin{align*}
			\lambda_a^+a^2=\mu(\eta_{q}-1)B_q[u_a^+].
		\end{align*}
		According to \eqref{estimate:energy_subcritical}, it holds that
		\begin{align}\label{estimate_w}
			E[\overline{w}_{\epsilon,t}]\leq& E[u_a^+]+E[tU_{\epsilon}]+\lambda_a^+ \int_{\mathbb{R}^N} u_a^+ (tU_{\epsilon})\dif x+\frac{\mu (1-\eta_{q})}{a^2}B_q[u_a^+]\int_{\mathbb{R}^N}u_a^+(x)(tU_{\epsilon}(x))\dif x\notag\\
			&+\mu\int_{\mathbb{R}^N}(I_{\alpha} * \vert u_a^+\vert^q)\vert u_a^+\vert^{q-2}u_a^+ (tU_{\epsilon})\dif x+\begin{cases}
				O(\epsilon^2),&\text{if } N\geq 5,\\
				O(\epsilon^2\vert \log\epsilon\vert),&\text{if } N=4,\\
				O(\epsilon),&\text{if } N=3,
			\end{cases}\notag\\
			\leq & E[u_a^+]+E[tU_{\epsilon}]+\begin{cases}
				O(\epsilon^{\frac{\alpha+2}{2}}),&\text{if }0<\alpha<(N-4)_{+},\\
				O(\epsilon^{\frac{N-2}{2}}\vert\log\epsilon\vert^{\frac{N-2}{N}}
				),&\text{if }\alpha=(N-4)_{+},\\
				O(\epsilon^{\frac{N-2}{2}}),&\text{if }(N-4)_{+}<\alpha<N.\\
			\end{cases}+\begin{cases}
				O(\epsilon^2),&\text{if } N\geq 5,\\
				O(\epsilon^2\vert \log\epsilon\vert),&\text{if } N=4,\\
				O(\epsilon),&\text{if } N=3.
			\end{cases}
		\end{align}
		It is time to estimate $E[tU_{\epsilon}]$. 
		It follows from Lemma \ref{lm:estimate_bubble} and Lemma \ref{lm:estimate_bubble_nonlocal} that 
		\begin{align}\label{estimate_energy}
			\max_{t\in[t_0,t_1]}E[tU_{\epsilon}]=&\max_{t\in[t_0,t_1]}\left(\frac{1}{2}A[tU_{\epsilon}]-\frac{1}{2p}B_p[tU_{\epsilon}]-\frac{\mu}{2q}B_q[tU_{\epsilon}]\right)\notag\\
			\leq & \max_{t\in[t_0,t_1]}\left(\frac{1}{2}A[tU_{\epsilon}]-\frac{1}{2p}B_p[tU_{\epsilon}]\right)-O(\epsilon^{2q(1-\eta_{q})})\notag\\
			\leq&  \max_{t>0}\left(\frac{t}{2}A[U_{\epsilon}]-\frac{t^{p}}{2p}B_p[U_{\epsilon}]\right)-O(\epsilon^{2q(1-\eta_{q})})\notag\\
			=&\frac{p-1}{2p}\left(A[U_{\epsilon}]\right)^{\frac{p}{p-1}}\left(B_p[U_{\epsilon}]\right)^{-\frac{1}{p-1}}-O(\epsilon^{2q(1-\eta_{q})})\notag\\
			\leq&\frac{\alpha+2}{2(N+\alpha)}\mathcal{S}^{\frac{N+\alpha}{\alpha+2}}_{H}+O(\epsilon^{\min\{N-2,\frac{N+\alpha}{2}\}})-O(\epsilon^{2q(1-\eta_{q})}).
		\end{align}
		Combing \eqref{estimate_w} and \eqref{estimate_energy}, we have
		\begin{align}
			E[\overline{w}_{\epsilon,t}]\leq&\gamma^{+}(a)+\frac{\alpha+2}{2(N+\alpha)}\mathcal{S}^{\frac{N+\alpha}{\alpha+2}}_{H}-O(\epsilon^{2q(1-\eta_{q})})+O(\epsilon^{\min\{N-2,\frac{N+\alpha}{2}\}})\notag\\
			&+\begin{cases}
				O(\epsilon^{\frac{\alpha+2}{2}}),&\text{if }0<\alpha<(N-4)_{+},\\
				O(\epsilon^{\frac{N-2}{2}}\vert\log\epsilon\vert^{\frac{N-2}{N}}
				),&\text{if }\alpha=(N-4)_{+},\\
				O(\epsilon^{\frac{N-2}{2}}),&\text{if }(N-4)_{+}<\alpha<N.\\
			\end{cases}+\begin{cases}
				O(\epsilon^2),&\text{if } N\geq 5,\\
				O(\epsilon^2\vert \log\epsilon\vert),&\text{if } N=4,\\
				O(\epsilon),&\text{if } N=3.
			\end{cases}
		\end{align}
		The remain part is to compare the growth of the negative leading term and the positive leading term as $\epsilon\rightarrow 0^{+}$ which depends on $N,\alpha$ and $q$.
		\begin{enumerate}
			\item[(a)] $N\geq 6$.
			\begin{itemize}
				\item If $0<\alpha\leq 2\leq N-4$, we need $2q(1-\eta_{q})<\frac{\alpha+2}{2}\leq 2$, which implies $q>\frac{2N-2+\alpha}{2N-4}$. 
				It is easy to check that $2_{\alpha}<\frac{2N-2+\alpha}{2N-4}<2_{\alpha}^{\sharp}$.
				Consequently, it is sufficient that $0<\alpha\leq 2$ and $\frac{2N-2+\alpha}{2N-4}<q<2_{\alpha}^{\sharp}$.
				\item If $2<\alpha<N-2$, we need $2q(1-\eta_{q})< 2$, which implies $q>\frac{N+\alpha-2}{N-2}$.
				Notice that $2_{\alpha}<\frac{N+\alpha-2}{N-2}<2_{\alpha}^{\sharp}$. 
				Therefore, it is sufficient that $2<\alpha<N-2$ and $\frac{N+\alpha-2}{N-2}<q<2_{\alpha}^{\sharp}$.
				\item If $N-2\leq\alpha<N$, it always holds that $2q(1-\eta_{q})\geq 2$ and our method does not work.
			\end{itemize}
			\item[(b)] $N=5$.
			\begin{itemize}
				\item If $0<\alpha\leq 1$, we need $2q(1-\eta_{q})<\frac{\alpha+2}{2}\leq 2$, which implies $q>\frac{8+\alpha}{6}$. 
				Thus it is sufficient that $0<\alpha\leq 1$ and $\frac{8+\alpha}{6}<q<2_{\alpha}^{\sharp}$.
				\item If $1<\alpha<3$, we need $2q(1-\eta_{q})< \frac{3}{2}<2$, which implies $q>\frac{7+2\alpha}{6}$.
				This condition must be admissible with $q<2_{\alpha}^{\sharp}$, which implies that $\alpha<\frac{7}{4}$. 
				Therefore, it is sufficient that $1<\alpha<\frac{7}{4}$ and $\frac{7+2\alpha}{6}<q<2_{\alpha}^{\sharp}$. 
				\item If $3\leq \alpha<5$, we find that $2q(1-\eta_{q})\geq 2$ and the above method does not work.
				\item If $4\leq \alpha <5$, similar to the second case in $N=3$ below, the positive leading term is $O(\epsilon^2)$ and we will add a negative leading  term $-O(\epsilon^{\frac{3}{2}})$. 
				Therefore, it is sufficient that $4\leq\alpha<5$ and $q\in (2_{\alpha},2_{\alpha}^{\sharp})$.
			\end{itemize}
			\item[(c)] $N=4$.
			\begin{itemize}
				\item If $0<\alpha<4$, we need that $2q(1-\eta_{q})< 1<2$, which gives that $q>\frac{3+\alpha}{2}$. 
				However, it is not admissible with $q<2_{\alpha}^{\sharp}$.
				\item If $2\leq\alpha<4$, similar to the second case in $N=3$ below, the positive leading term is $O(\epsilon^2\vert \log\epsilon\vert)$ and we will add a negative leading  term $-O(\epsilon^{1})$. 
				Therefore, it is sufficient that $2\leq\alpha<4$ and $q\in (2_{\alpha},2_{\alpha}^{\sharp})$.
			\end{itemize}
			\item[(d)] $N=3$.
				\begin{itemize}
					\item If $0<\alpha<3$, similar to the first case in $N=4$, the above method does not work. 
					\item If $1\leq\alpha<3$, we find that $q>2_{\alpha}\geq 2$, which means that we have
					\begin{align*}
							B_q[w_{\epsilon,t}]\geq B_q[u_a^+]+B_q[tU_{\epsilon}]+2q\int_{\mathbb{R}^N}(I_{\alpha}*\vert u^+_a\vert^q)\vert u^+_a\vert^{q-2}u^+_a(tU_{\epsilon})\dif x,
					\end{align*}
					and hence  the term $\mu\int_{\mathbb{R}^N}(I_{\alpha}*\vert u^+_a\vert^q)\vert u^+_a\vert^{q-2}u^+_a(tU_{\epsilon})\dif x$  will not appear, which means the  positive leading term is $O(\epsilon)$.
					Besides, we find that $2_{\alpha}^{*}=3+\alpha>3$, which implies that 
					\begin{align*}
						B_p[w_{\epsilon,t}]\geq& B_p[u_a^+]+B_p[tU_{\epsilon}]+2p\int_{\mathbb{R}^N}(I_{\alpha}*\vert u^+_a\vert^p)\vert u^+_a\vert^{p-2}u^+_a(tU_{\epsilon})\dif x \\
						&+2p\int_{\mathbb{R}^N}(I_{\alpha}*\vert tU_{\epsilon}\vert^p)\vert tU_{\epsilon}\vert^{p-2}tU_{\epsilon}(u_a^+)\dif x.
					\end{align*}
					Consequently, the negative adding term is given by 
					\begin{align*}
						-\int_{\mathbb{R}^N}(I_{\alpha}*\vert tU_{\epsilon}\vert^p)\vert tU_{\epsilon}\vert^{p-2}tU_{\epsilon}(u_a^+)\dif x\leq -O(\epsilon^{\frac{1}{2}}),
					\end{align*}
					and hence it is sufficient that $1\leq\alpha<3$ and $q\in (2_{\alpha},2_{\alpha}^{\sharp})$.
				\end{itemize}
		\end{enumerate}
	\end{proof}
	Now we can prove Theorem \ref{thm:mass_subcritical}(ii) for $p=2_{\alpha}^*$:
	\begin{proof}[Proof of Theorem \ref{thm:mass_subcritical}(ii) for $p=2_{\alpha}^*$]
		By Lemma \ref{lm:PS_sequence_existence_mass_subcritical}, there exists a bounded Palais-Smale sequence $\{u_n\}\subset \mathcal{P}^{-}(a)$ for $\left.E\right\vert_{S(a)}$ at level $\gamma^{-}(a)$.
		Up to a subsequence, we get $u_n\rightharpoonup u_a$ weakly in $H^1(\mathbb{R}^N)$. 
		Now applying Lemma \ref{lm:PS_strong_convergence} to $m=\gamma^{-}(a)$, we suppose that the alternative $(i)$ holds. 
		Then, using Lemma \ref{lm:energy_critical_sub}, we get that $$E[u_a^-]<\gamma^{-}(a)-\frac{\alpha+2}{2(N+\alpha)}\mathcal{S}_{H}^{\frac{N+\alpha}{\alpha+2}}<\gamma^{+}(a)<0,$$
		which is impossible.
		Hence we obtain that $u_n \rightarrow u_a^-$ in $H^1(\mathbb{R}^N)$ and $\gamma^{-}(a)=E[u_a^-]$.
		The remain proof is similar to Theorem \ref{thm:mass_subcritical}(i).
	\end{proof}
	\subsubsection{The compactness of any minimizing sequences related to $\gamma^{+}(a)$ in the HLS critical case}
	We shall prove Theorem \ref{thm:minizing} in this part.
	\begin{lemma}\label{lm:minizing}
		Let $N\geq 3$, $\alpha \in (0,N)$, $2_{\alpha}<q<2_{\alpha}^{\sharp}<p\leq 2_{\alpha}^{*}$ and $\mu,a>0$. 
		Suppose that \eqref{assumptions:basic} holds.
		Then we have
		\begin{enumerate}
			\item[(i)] if $u \in \mathcal{P}^+(a)$, then $A[u]<k_0a^{\frac{2q(1-\eta_{q})}{1-q\eta_{q}}}$, where $k_0=\left( \frac{\mu\eta_{q}\mathcal{C}_{G}(q)(p\eta_{p}-q\eta_{q})}{p\eta_{p}-1}\right)^{\frac{1}{1-q\eta_{q}}}$;
			\item[(ii)] $\mathcal{P}^{+}(a)\subset V(a)$ and $\gamma^{+}(a)=\inf\limits_{u \in \mathcal{P}^{+}(a)}E[u]=\inf\limits_{u \in \overline{V(a)}}E[u]$;
			\item[(iii)] if $u_a$ is a minimum for the minimization problem 
			$$\inf\limits_{u \in \mathcal{P}^{+}(a)}E[u]=\inf\limits_{u \in \overline{V(a)}}E[u],$$
			 then $u_a\in V(a)$ and $\gamma^+(a)$ is reached.
		\end{enumerate}
	\end{lemma}
	\begin{proof}
		\begin{enumerate}
			\item[(i)] Since $u \in \mathcal{P}^{+}(a)$,
			\begin{align*}
				A[u]=\eta_{p}B_p[u]+\mu\eta_{q}B_q[u],\quad\text{and}\quad \mu \eta_{q}(1-q\eta_{q})B_q[u]>\eta_{p}(p\eta_{p}-1)B_p[u],
			\end{align*}
			which implies that 
			\begin{align*}
				A[u]<\frac{\mu\eta_{q}(p\eta_{p}-q\eta_{q})}{p\eta_{p}-1}B_q[u]\leq \frac{\mu\eta_{q}\mathcal{C}_{G}(q)(p\eta_{p}-q\eta_{q})}{p\eta_{p}-1}a^{2q(1-\eta_{q})}A[u]^{q\eta_{q}}.
			\end{align*}
			It follows that 
			\begin{align*}
				A[u]<\left( \frac{\mu\eta_{q}\mathcal{C}_{G}(q)(p\eta_{p}-q\eta_{q})}{p\eta_{p}-1}\right)^{\frac{1}{1-q\eta_{q}}}a^{\frac{2q(1-\eta_{q})}{1-q\eta_{q}}}=k_0a^{\frac{2q(1-\eta_{q})}{1-q\eta_{q}}}<k_0a_1^{\frac{2q(1-\eta_{q})}{1-q\eta_{q}}}=k_1.
			\end{align*}
			\item[(ii)]From point $(i)$, we obtain that $\mathcal{P}^+(a)\subset V(a)$, and it suffices to prove that 
			\begin{align*}
				\inf\limits_{u \in \mathcal{P}^{+}(a)}E[u]\leq\inf\limits_{u \in \overline{V(a)}}E[u].
			\end{align*}
			It is not hard to show that $\mathcal{P}^-(a)\cap \overline{V(a)}=\emptyset$.
			Indeed, if not, suppose that there exists a nontrivial $v \in \mathcal{P}^-(a)\cap \overline{V(a)}$, then we have
			\begin{align*}
				A[v]<\frac{\eta_{p}(p\eta_{p}-q\eta_{q})}{1-q\eta_{q}}B_p[v]\leq
				\begin{cases}
					\frac{\eta_{p}\mathcal{C}_{G}(p)(p\eta_{p}-q\eta_{q})}{1-q\eta_{q}}a^{2p(1-\eta_{p})}A[v]^{p\eta_{p}},&\text{if } p\in (2_{\alpha},2_{\alpha}^{*}),\\
					\frac{\eta_{p}\mathcal{S}_{H}^{-p}(p-q\eta_{q})}{1-q\eta_{q}}A[v]^{p},&\text{if } p=2_{\alpha}^{*},
				\end{cases}
			\end{align*}
			which implies that 
			\begin{align*}
				A[v]>
				\begin{cases}
					\left(\frac{\eta_{p}\mathcal{C}_{G}(p)(p\eta_{p}-q\eta_{q})}{1-q\eta_{q}}a^{2p(1-\eta_{p})}\right)^{-\frac{1}{p\eta_{p}-1}},&\text{if } p\in (2_{\alpha},2_{\alpha}^{*}),\\
					\left(\frac{\eta_{p}\mathcal{S}_{H}^{-p}(p-q\eta_{q})}{1-q\eta_{q}}\right)^{-\frac{1}{p-1}},&\text{if } p=2_{\alpha}^{*}.
				\end{cases}
			\end{align*}
			This leads to a contradiction due to \eqref{assumptions:basic}.
			Now consider $u\in S(a)$ and there exists a unique $s_u^1$ such that $A[s_u^1\circ u]=k_1$.
			Using Lemma \ref{lm:unique_minimum_maximum}, we have
			\begin{align*}
				A[s_u^+ \circ u]< A[s_u^1 \circ u]< A[s_u^- \circ u],
			\end{align*}
			which implies that $s_u^+<s_u^1<s_u^-$.
			By the monotonicity of $f_u(s)$ and the fact that $s_u^+$ is the unique local minimum point, we obtain that 
			\begin{align*}
				E[s_u^+ \circ u]=f_u(s_u^+)\leq f_u(s)=E[s \circ u],\quad \forall s \in (0,s_u^1].
			\end{align*} 
			It follows from $u \in \overline{V(a)}$ that 
			\begin{align*}
				E[u]\geq \min_{u\in\overline{V(a)}} E[u]=\min\{E[s\circ u]:s\in \mathbb{R}^+,\ A[s\circ u]\leq k_1\}=\min\{E[s\circ u]:s\in (0,s_u^1]\}=E[s_u^+\circ u].
			\end{align*}
			\item[(iii)]
			Suppose that $u_a \in \partial V(a)$, then we have $A[u_a]=k_1$, which is impossible. 
			Hence we obtain that $u_a\in V(a)$. 
			Since $V(a)$ is relatively open in $S(a)$, then we have that $u_a$ is a weak solution to \eqref{eq:NLS_S} and $P[u_a]=0$. 
			Note that $V(a)\cap \mathcal{P}^-(a)=\emptyset$ and $\mathcal{P}^0(a)=\emptyset$, we get $u_a\in\mathcal{P}^+(a)$ and thus $\gamma^+(a)$ is reached.
		\end{enumerate}
	\end{proof}
	We need the following lemma:
	\begin{lemma}\label{lm:nonincreasing}
		Let $N\geq 3$, $\alpha \in (0,N)$, $2_{\alpha}<q<2_{\alpha}^{\sharp}<p\leq 2_{\alpha}^{*}$ and $\mu,a>0$. 
		Suppose that \eqref{assumptions:basic} holds.
		It is true that
		\begin{enumerate}
			\item[(i)] $\gamma^{+}(a)<0,\quad \forall a \in (0,a_1)$;
			\item[(ii)] The mapping $a\in (0,a_1) \mapsto \gamma^{+}(a)$ is continuous:
			\item[(iii)] Let $a\in (0,a_1)$, for all $b\in (0,a)$, we have $\gamma^{+}(a)\leq \gamma^{+}(b)+\gamma^{+}(a-b)$ and if one of $\gamma^{+}(a)$ and $\gamma^{+}(a-b)$ is reached, then the inequality is strict.
		\end{enumerate}
	\end{lemma}
	\begin{proof}
		$(i)$ is a direct corollary of Lemma \ref{lm:A_bounded_below}. 
		To show $(ii)$, let $a\in (0,a_1)$ be arbitrary and choose $\{a_n\}\subset (0,a_1)$ with $a_n\rightarrow a$ as $n\rightarrow +\infty$. 
		According to the definition of $\gamma^{+}(a_n)$, for any $\epsilon>0$, there exists a $u_n \in \mathcal{P}^{+}(a_n)$ satisfying 
		\begin{align*}
			E[u_n]\leq \gamma^{+}(a_n)+\epsilon.
		\end{align*}
		Our goal is to show that $\gamma^+(a_n)\rightarrow \gamma^+(a)$ as $n\rightarrow +\infty$.
		Let $\{t_n\}$ be a sequence satisfying $v_n(x)=t_n u_n(x)$ for every $n\in \mathbb{N}$.
		It follows that
		\begin{align*}
			\|v_n\|_2^2=t_n^2\|u_n\|_2^2=a^2,\quad A[v_n]=t_n^2A[u_n],\quad B_s[v_n]=t_n^{2s}B_s[u_n],
		\end{align*}
		which implies that 
		\begin{align*}
			E[v_n]=\frac{t_n^2}{2}A[u_n]-\frac{t_n^{2p}}{2p}B_p[u_n]-\frac{t_n^{2q}}{2q}\mu B_q[u_n]\leq t_n^2E[u_n].
		\end{align*}
		According to the fact that $t_n=aa_n^{-1}\rightarrow 1$ as $n\rightarrow +\infty$, it follows that
		\begin{align*}
			A[v_n]=t_n^2A[u_n]=(aa_n^{-1})^2A[u_n]\leq k_0a^2a_n^{-2}a_n^{\frac{2q(1-\eta_{q})}{1-q\eta_{q}}}\leq k_0 a^{\frac{2q(1-\eta_{q})}{1-q\eta_{q}}}<k_1,
		\end{align*}
		and hence $v_n \in \overline{V(a)}$, which implies that 
		\begin{align*}
			\gamma^{+}(a)\leq E[v_n]=E[u_n]+o_n(1)\leq \gamma^{+}(a_n)+\epsilon+o_n(1).
		\end{align*}
		Similarly, it holds that
		\begin{align*}
			\gamma^{+}(a_n)\leq \gamma^{+}(a)+\epsilon+o_n(1)
		\end{align*}
		and consequently $\gamma^+(a_n)\rightarrow \gamma^+(a)$ as $n\rightarrow +\infty$.
		
		$(iii)$ Using a similar proof in point $(ii)$, it is easy to show that 
		\begin{align*}
			\forall \theta \in (1,\frac{a}{b}],\quad \gamma^{+}(\theta b)\leq \theta \gamma^{+}(b).
		\end{align*}
		Then we have
		\begin{align*}
			\gamma^{+}(a)=\frac{a-b}{a}\gamma^{+}(a)+\frac{b}{a}\gamma^{+}(a)=\frac{a-b}{a}\gamma^{+}(\frac{a}{a-b}(a-b))+\frac{b}{a}\gamma^{+}(\frac{a}{b}b)\leq \gamma^+(a-b)+\gamma^+(b).
		\end{align*}
	\end{proof}
	Now we can prove the main theorem in this subsection:
	\begin{proof}[Proof of Theorem \ref{thm:minizing}]
		Let $\{u_n\}\subset \overline{V(a)}$ be an minimizing sequence for $E$ and it is bounded in $H^{1}(\mathbb{R}^N)$. 
		Let $$\rho:=\limsup_{n\rightarrow+\infty}\left(\sup_{y\in\mathbb{R}^N}\int_{B(y,1)}\vert v_n \vert^2\dif x\right).$$
		We consider an alternative: \textbf{nonvanishing} or \textbf{vanishing}.
		\begin{itemize}
			\item \textbf{Vanishing}, that is, $\rho=0$.
			According to Lion's vanishing Lemma \cite[Lemma I.1]{MR778974}, it follows that $u_n \rightarrow 0$ in $L^{s}(\mathbb{R}^N)$ for all $s\in (2,2^*)$.
			therefore, it holds that $B_q[u_n]=o_n(1)$ by the Hardy-Littlewood-Sobolev inequality \eqref{ieq:HLS}.
			It is clear that
			\begin{align*}
				E[u_n]
				\geq 
				\begin{cases}
					\frac{1}{2}A[u_n]-\frac{\mathcal{C}_{G}(p)}{2p}a^{2p(1-\eta_{p})}A[u_n]^{p\eta_{p}}+o_n(1),&\text{if } p\in (2_{\alpha}^{\sharp},2_{\alpha}^{*}),\\
					\frac{1}{2}A[u_n]-\frac{\mathcal{S}_{H}^{-p}}{2p}A[u_n]^{p}+o_n(1),&\text{if } p=2_{\alpha}^{*}.
				\end{cases}
			\end{align*}
			Since $A[u_n]\leq k_1$ and \eqref{assumptions:basic} hold, it holds that $\liminf\limits_{n\rightarrow+\infty}E[u_n]\geq 0$, which contradicts to the fact $E[u_n]\rightarrow \gamma^{+}(a)<0$.
			\item \textbf{Nonvanishing}, that is, $\rho> 0$.
			Then up to a subsequence, there exists a sequence $\{z_n\}\subset \mathbb{R}^N$ and $u_a\in H^1(\mathbb{R}^N)\setminus\{0\}$ such that 
			\begin{align*}
				w_n:=u_n(\cdot + z_n)\rightharpoonup u_a\quad \text{in } H^{1}(\mathbb{R}^N)\quad \text{and}\quad w_n \rightarrow u_a\quad \text{a.e. in } \mathbb{R}^N.
			\end{align*}
			Define $v_n(x):=w_n(x)-u_a(x)$, and we claim that $v_n \rightarrow 0$ strongly in $H^1(\mathbb{R}^N)$.
			To prove this, we shall prove $\|v_n\|_2\rightarrow 0$ and $A[v_n]\rightarrow 0$ respectively.
			Since $H^{1}(\mathbb{R}^N)$ is a Hilbert space and $v_n \rightharpoonup 0$, it is easy to show that
			\begin{align*}
				\|u_n\|_2^2=\|v_n\|_2^2+\|u_a\|_2^2+o_n(1),\quad A[u_n]=A[v_n]+A[u_a]+o_n(1).
			\end{align*}
			By lemma \ref{lm:Brezis–Lieb_nonlocal} and the transnational invariance, it follows that
			\begin{align*}
				E[u_n]=E[v_n]+E[u_a]+o_n(1).
			\end{align*} 
			By the weak lower semicontinuity of the norm $\|\cdot\|_2$, we obtain that $\|u_a\|_2\leq a$.
			Put $b:=\|u_a\|_2$ and suppose that $0<b<a$.
			For $n$ large enough, $\|v_n\|_2\leq a$ and  $A[v_n]\leq A[u_n]\leq k_1$.
			Therefore, we obtain that $v_n \in \overline{V(\|v_n\|_2)}$ and $E[v_n]\geq \gamma^+(\|v_n\|_2)$.
			Consequently, it holds that
			\begin{align*}
				\gamma^{+}(a)+o_n(1)=E[u_n]+o_n(1)=E[v_n]+E[u_a]\geq \gamma^+(\|v_n\|_2)+E[u_a].
			\end{align*}
			According to the continuity of mapping $a \mapsto \gamma^+(a)$, it follows that
			\begin{align*}
				\gamma^{+}(a)\geq \gamma^+(a-b)+E[u_a].
			\end{align*}
			By Lemma \ref{lm:nonincreasing}(iii), we arrive at
			\begin{align*}
				\gamma^{+}(b)+\gamma^+(a-b)\geq\gamma^{+}(a)\geq \gamma^+(a-b)+E[u_a],
			\end{align*}
			that is, $E[u_a]\leq \gamma^{+}(b)$.
			It follows from $u_a \in \overline{V(b)}$ that $E[u_a]\geq \gamma^{+}(b)$.
			As a consequence, $E[u_a]=\gamma^{+}(b)$.
			However, the strict inequality holds in Lemma \ref{lm:nonincreasing}(iii) if $E[u_a]=\gamma^{+}(b)$, which leads to a contradiction.
			We have proved that $\|u_a\|_2=a$ and hence $\|w_n\|_2=o_n(1)$.
			
			The last step is to show that $A[v_n]\rightarrow 0$.
			It follows from $\|v_n\|_2=o_n(1)$ that $B_q[v_n]=o_n(1)$.
			By lemma \ref{lm:Brezis–Lieb_nonlocal}, it holds that
			\begin{align*}
				E[u_n]=E[v_n]+E[u_a]+o_n(1)\rightarrow \gamma^{+}(a).
			\end{align*}
			Notice that $u_a\in \overline{V(a)}$, we obtain that $\gamma^{+}(a)\leq E[u_a],$ which gives that $E[v_n]\leq o_n(1)$.
			According to the fact $A[v_n]\leq A[u_n]\leq k_1$, similar to the vanishing case, it follows that $A[v_n]=o_n(1)$.
		\end{itemize}
		Consequently, $u_n\rightarrow u_a\in \overline{V(a)}$ strongly in $H^1(\mathbb{R}^N)$. Then by Lemma \ref{lm:minizing}(iii), we have $u_a\in V(a)$ and hence $\gamma^{+}(a)$ is reached.
	\end{proof}
	\subsection{$L^2$-critical perturbation}
	In this subsection, we shall always assume that $q=2_{\alpha}^{\sharp}<p\leq 2_{\alpha}^{\sharp}$ and $\mu,a>0$. 
	Note that we have $q\eta_{q}=1$.
	The main proof in this subsection is similar to the $L^2$-subcritical focusing perturbation.
	So we shall give some lemmas without details and focus on the differences of proof.
	\subsubsection{The existence of bounded Palais-Smale sequences}
	Now we have
	\begin{align*}
		f_u(s)=&sA[u]-\frac{\mu}{q}sB_q[u]-\frac{1}{p}s^{p\eta_{p}}B_p[u],\\
		f_u^{\prime}(s)=&A[u]-\frac{\mu}{q}B_q[u]-\eta_{p}s^{p\eta_{p}-1}B_p[u],\\
		f_u^{\prime\prime}(s)=&-\eta_{p}(p\eta_{p}-1)s^{p\eta_{p}-2}B_p[u]\leq 0,
	\end{align*}
	which implies that $\mathcal{P}(a)=\mathcal{P}^{-}(a)\cup\mathcal{P}^{0}(a)$.
	If $\mathcal{P}^{0}(a)\ne \emptyset,$ then we assume that there exists a nontrivial function $ u\in\mathcal{P}^{0}(a)$.
	It follows that
	\begin{align*}
		A[u]=\frac{\mu}{q}B_q[u]+\eta_{p}B_p[u],\quad B_p[u]=0,
	\end{align*}
	which implies that
	\begin{align*}
		A[u]=\frac{\mu}{q}B_q[u]\leq \frac{\mu}{q}\mathcal{C}_{G}(q)a^{2q-2}A[u],
	\end{align*}
	that is, $\mu a^{2q-2}\geq q\mathcal{C}_{G}(q)^{-1}$. 
	Indeed, we have proved the following lemma:
	\begin{lemma}\label{lm:P_manifold_mass_critical}
		Let $N\geq 3$, $\alpha\in (0,N)$, $q=2_{\alpha}^{\sharp}<p\leq 2_{\alpha}^{\sharp}$ and $\mu,a>0$. 
		If $\mu a^{2q-2}<q\mathcal{C}_{G}(q)^{-1}$, then we have $\mathcal{P}(a)=\mathcal{P}^{-}(a)$.
	\end{lemma}

	Suppose that $0<\mu a^{2q-2}<q\mathcal{C}_{G}(q)^{-1}$. 
	Put $s_u=\left(\frac{A[u]-\mu \eta_{q}B_q[u]}{\eta_{p}B_p[u]}\right)^{\frac{1}{p-1}}$, and notice that $f_u^{\prime}(s_u)=0$.
	If $s\in [0,s_u),$ $f_u^{\prime}(s)>0$ and if $s\in (s_u,+\infty),$ $f_u^{\prime}(s)<0$. 
	Then $s_u$ is the unique maximum point for $f_u$ on $[0,+\infty)$ and $s_u\circ u\in\mathcal{P}^{-}(a)$. Following the argument in \cite[Lemma 5.3]{SOAVE20206941}, we can prove that the map $u\in S(a) \mapsto s_u^{\diamond}\in\mathbb{R}$ is of class $\mathscr{C}^1$.
	
	\begin{lemma}\label{lm:A_bounded_below_mass_critical}
	Let $N\geq 3$, $q=2_{\alpha}^{\sharp}<p\leq 2_{\alpha}^{\sharp}$ and $\mu,a>0$.
	Suppose that $\mu a^{2q-2}<q\mathcal{C}_{G}(q)^{-1}$. Then there exists a $\delta>0$ such that $A[u]\geq \delta$ for all $u\in \mathcal{P}^{-}(a)$.
	\end{lemma}
	\begin{proof}
		If $u\in \mathcal{P}^{-}(a)$, we have $B_p[u]<0$ and hence 
		\begin{align*}
			A[u]=&\mu \eta_{q}B_q[u]+\eta_{p}B_p[u]\\
			\leq& 
			\begin{cases}
				\mu \eta_{q}\mathcal{C}_{G}(q)a^{2q-2}A[u]+\eta_{p}\mathcal{C}_H(p)a^{2p(1-\eta_{p})}A[u]^{p \eta_{p}},&\text{if }p\in (2_{\alpha}^{\sharp},2_{\alpha}^{*}),\\
				\mu \eta_{q}\mathcal{C}_{G}(q)a^{2q-2}A[u]+S_{H}^{-p}A[u]^p,&\text{if }p=2_{\alpha}^{*}.
			\end{cases}
		\end{align*}
		The Lemma is complete by noticing that $p\eta_{p}>1$ and $1-\mu \eta_{q}\mathcal{C}_{G}(q)a^{2q-2}>0$.
	\end{proof}
	Similar to Lemma \ref{lm:inf_Pohozaev_pm}, we can prove that 
	\begin{lemma}
		Let $N\geq 3$, $\alpha\in (0,N)$, $q=2_{\alpha}^{\sharp}<p\leq 2_{\alpha}^{\sharp}$ and $\mu,a>0$.
		Set $\mu a^{2q-2}<q\mathcal{C}_{G}(q)^{-1}$. 
		Then we have
		$$\gamma(a)=\inf\limits_{\mathcal{P}(a)} E[u]=\inf\limits_{\mathcal{P}^{-}(a)} E[u]=\inf\limits_{\mathcal{P}_r^{-}(a)} E[u]>0.$$
	\end{lemma}
	Again, we can define $I: S(a)\rightarrow \mathbb{R}$ by $ I[u]=E[S_u\circ u]$ and it holds that $\langle \dif I[u],\psi\rangle=\langle \dif E[S_u\circ u], S_u\circ \psi\rangle$ for any $u\in S(a)$ and $\psi\in T_uS(a)$.
	Using the method in Lemma \ref{lm:PS_sequence_existence_mass_subcritical}, it follows that
	\begin{lemma}\label{lm:PS_existence_mass_critical}
		Let $N\geq 3$, $0<\alpha<N$, $q=2_{\alpha}^{\sharp}<p\leq 2_{\alpha}^{\sharp}$ and $\mu,a>0$.
		Set $\mu a^{2q-2}<q\mathcal{C}_{G}(q)^{-1}$.
		Then there exists a bounded Palais-Smale sequence $\{u_n\}\subset \mathcal{P}(a)$ for $E$ restricted to $S_r(a)$ at level $\gamma(a)$.
	\end{lemma}

	\subsubsection{The compactness of Palais-Smale sequences under the HLS critical leading term}
	In this part, we assume $p=2_{\alpha}^{*}$.
	Now we shall investigate the properties of $\gamma(a)$. 
	\begin{lemma}\label{lm:energy_mass_critical}
		Let $N\geq 4$ with $\alpha\in (0,N)$ or $N=3$ with $\alpha\in (0,1)$, $q=2_{\alpha}^{\sharp}<p=2_{\alpha}^{*}$ and $\mu,a>0$. 
		Then we have
		\begin{enumerate}
			\item[(i)] If $0<\mu a^{2q-2}<q\mathcal{C}_{G}(q)^{-1}$, then $0<\gamma(a)<\frac{\alpha+2}{2N+2\alpha}S^{\frac{N+\alpha}{\alpha+2}}_{H}$;
			\item[(ii)] Fix $a>0$, $\gamma(a, \mu)$ is strictly decreasing with respective to $\mu$ for $0<\mu a^{2q-2}<q\mathcal{C}_{G}(q)^{-1}$ and is nonincreasing with respective to $\mu$ for $\mu a^{2q-2}\geq q\mathcal{C}_{G}(q)^{-1}$;
			\item[(iii)] It holds that $\gamma(a)=0$ for $\mu a^{2q-2}\geq q\mathcal{C}_{G}(q)^{-1}$.
		\end{enumerate}  
	\end{lemma}
	\begin{proof}
		\begin{enumerate}
			\item[(i)] Assume that $0<\mu a^{2q-2}<q\mathcal{C}_{G}(q)^{-1}$. Let $\xi\in\mathscr{C}_{c}^{\infty}$ be a radial cut-off function with $\xi\equiv 1$ in $B_1$,   $\operatorname{supp}(\xi)\subset B_2$ and $0\leq \xi(x)\leq 1,$ $\forall x\in\mathbb{R}^N$. 
			We define                 
			\begin{align*}
				U_{\epsilon}(x):=\xi(x) u_{\epsilon}(x),\quad \text{and}\quad V_{\epsilon}(x):=a
				\frac{U_{\epsilon}(x)}{\vert U_{\epsilon}\vert_2}.
			\end{align*} 
			Notice that $U_{\epsilon}\in \mathscr{C}_c^{\infty}(\mathbb{R}^N)$ and $V_{\epsilon}\in S_r(a)$. Therefore, it holds that 
			\begin{align*}
				\gamma(a)=\inf\limits_{\mathcal{P}^{-}(a)}E\leq E[s_{V_{\epsilon}}^{\diamond}\circ V_{\epsilon}]=\max_{s\geq 0}E[s\circ V_{\epsilon}],\quad \forall \epsilon>0.
			\end{align*}
			We can use the method described in \cite[Lemma 5.3]{MR4096725} to estimate the upper bound of $E[s_{V_{\epsilon}}^{\diamond}\circ V_{\epsilon}]$. 
			However, we can do this in a directer way. 
			According to the fact $f_u^{\prime}(s_u)=0$ and $s_u=\left(\frac{A[u]-\mu \eta_{q}B_q[u]}{\eta_{p}B_p[u]}\right)^{\frac{1}{p-1}}$, it follows that
			\begin{align*}
				f_u(s_u)=\frac{p-1}{p}s_u\left(A[u]-\frac{\mu}{q}B_q[u]\right)^{\frac{1}{p-1}}
				=\frac{p-1}{p}B_p[u]^{-\frac{1}{p-1}}\left(A[u]-\frac{\mu}{q}B_q[u]\right)^{\frac{p}{p-1}}.
			\end{align*}
			Notice that $A[V_{\epsilon}]=\frac{a^2}{\vert U_{\epsilon}\vert_2^2}A[U_{\epsilon}]$ and $B_s[V_{\epsilon}]=\frac{a^{2s}}{\vert U_{\epsilon}\vert_2^{2s}}B_s[U_{\epsilon}]$ for $s\in [2_{\alpha},2_{\alpha}^{*}]$. 
			As a consequence, we obtain that
			\begin{align*}
				E[s_{V_{\epsilon}}\circ V_{\epsilon}]
				=&\frac{\alpha+2}{2(N+\alpha)}B_p[V_{\epsilon}]^{-\frac{1}{p-1}}\left(A[V_{\epsilon}]-\frac{\mu}{q}B_q[V_{\epsilon}]\right)^{\frac{p}{p-1}}\\
				=&\frac{\alpha+2}{2(N+\alpha)}B_p[U_{\epsilon}]^{-\frac{1}{p-1}}\left(A[U_{\epsilon}]-\frac{\mu a^{2q-2}}{q}\vert U_{\epsilon} \vert_2^{2-2q} B_q[U_{\epsilon}]\right)^{\frac{p}{p-1}}\\
				\leq& 
				\frac{\alpha+2}{2(N+\alpha)}\mathcal{S}_H^{\frac{N+\alpha}{\alpha+2}}+\begin{cases}
					O(\epsilon^{\min\{\frac{N+\alpha}{2},\ N-2\}})-O(1),&\text{if }N\geq 5,\\
					O(\epsilon^{2})-O(\vert \log \epsilon\vert^{-\frac{\alpha+2}{N}}),&\text{if }N=4,\\
					O(\epsilon)-O(\epsilon^{\frac{\alpha+2}{N}}),&\text{if }N=3,
				\end{cases}\\
			\end{align*}
			where we have used Lemma \ref{lm:estimate_bubble} and Lemma \ref{lm:estimate_bubble_nonlocal}. 
			If $N\geq 4$ with $\alpha\in (0,N)$, it is easy to see that our lemma holds. 
			If $N=3$ with $\alpha\in (0,1)$, we have that $\frac{\alpha+2}{N}<1$ and hence our lemma holds.
			\item[(ii)]
			Notice that $\gamma(a, \mu)=\inf\limits_{u\in\mathcal{P}(a)}E[u]=\inf\limits_{u\in S(a)}\max\limits_{s>0} E[s\circ u]$. Indeed, $\max\limits_{s>0} E[s\circ u]=E[s_u \circ u]$ and $s_u\circ u\in \mathcal{P}(a)$, which implies that 
			\begin{align*}
				\inf\limits_{u\in S(a)}\max\limits_{s>0} E[s\circ u]\geq \inf\limits_{u\in\mathcal{P}(a)}E[u].
			\end{align*}
			Conversely, if $u\in\mathcal{P}(a)$, it holds that $s_u=1$ and hence
			\begin{align*}
				E[u]=\max\limits_{s>0} E[s\circ u]\geq \inf\limits_{u\in S(a)}\max\limits_{s>0} E[s\circ u].
			\end{align*}
			Now first consider $0<\mu_1<\mu_2<qa^{2-2q}\mathcal{C}_{H}(\alpha,q)^{-2q}$. 
			We shall write $s_{u,\mu}$ to show the dependence of $\mu$. 
			It follows that
			\begin{align*}
				\gamma_{\mu_2}(a)=&\inf\limits_{u\in S(a)}\max\limits_{s>0} E_{\mu_2}[s\circ u]=\inf\limits_{u\in S(a)}E_{\mu_2}[s_{u,\mu_2} \circ u]
				=\inf\limits_{u\in S(a)}\left(E_{\mu_1}[s_{u,\mu_2} \circ u]-\frac{\mu_2-\mu_1}{2q}s_{u,\mu_2}B_q[u]\right)\\
				\leq & \inf\limits_{u\in S(a)}E_{\mu_1}[s_{u,\mu_2} \circ u]-\frac{\mu_2-\mu_1}{2q}s_{u,\mu_2}\inf\limits_{u\in S(a)} B_q[u]
				< \inf\limits_{u\in S(a)}\max\limits_{s>0} E_{\mu_1}[s\circ u]=\gamma_{\mu_1}(a).
			\end{align*}
			Now we will consider $qa^{2-2q}\mathcal{C}_{H}(\alpha,q)^{-2q}\leq\mu_1<\mu_2$.
			We first claim that
			\begin{align*}
				\sup\limits_{u\in S(a)}\frac{A[u]}{B_q[u]}=+\infty.
			\end{align*}
			By Hardy-Littlewood-Sobolev inequality \eqref{ieq:HLS}, we get that 
			\begin{align*}
				B_q[u]\lesssim \|u\|_{\frac{2Nq}{N+\alpha}}^{2q}.
			\end{align*}
			Since $\frac{2Nq}{N+\alpha} \in (2,2^{*})$, there exists a $\theta \in (0,1)$ satisfying
			\begin{align*}
				\frac{\theta}{2}+\frac{1-\theta}{2^{*}}=\frac{N+\alpha}{2Nq}.
			\end{align*}
			According to the H\"{o}lder inequality, it follows that 
			\begin{align*}
				\| u \|_{\frac{2Nq}{N+\alpha}}\leq \|u\|_2^{\theta}\|u\|_{2^*}^{1-\theta}.
			\end{align*}
			Consequently, we have 
			\begin{align*}
				\frac{A[u]}{B_q[u]}\gtrsim \frac{A[u]}{\|u\|_2^{2q\theta}\|u\|_{2^*}^{2q(1-\theta)}}=\frac{\|\nabla u\|_2^2}{a^{2q\theta}\|u\|_{2^*}^{2}}.
			\end{align*}
			It is easy to show that 
			\begin{align*}
				\sup\limits_{u\in H^{1}(\mathbb{R}^N)\setminus\{0\}}\frac{\|\nabla u\|_2}{\|u\|_{2^*}}=+\infty.
			\end{align*}
			Now for $u\in H^{1}(\mathbb{R}^N)\setminus\{0\}$, we can define $v=a\|u\|_2^{-1}u$ and hence we have $\|v\|_2=a$ and 
			\begin{align*}
				\frac{\| \nabla v \|_2}{\| v \|_{2^{*}}}=\frac{\| \nabla u \|_2}{\| u \|_{2^{*}}},
			\end{align*}
			the claim is completed.
			
			Now we can choose $ u\in S(a)$ with $\frac{A[u]}{B_q[u]}>\mu \eta_{q}$, and put $s_u=\left(\frac{A[u]-\mu \eta_{q}B_q[u]}{\eta_{p}B_p[u]}\right)^{\frac{1}{p-1}}$, and notice that $f_u^{\prime}(s_u)=0$. 
			If $s\in [0,s_u),$ $f_u^{\prime}(s)>0$ and if $s\in (s_u,+\infty),$ $f_u^{\prime}(s)<0$. 
			Then $s_u$ is the unique maximum point for $f_u$ on $[0,+\infty)$ and $s_u\circ u\in\mathcal{P}^{-}(a)$.
			The remain part is similar to the case $0<\mu a^{2q-2}<q\mathcal{C}_{G}(q)^{-1}$.
			\item[(iii)]
			Let $\{w_n\}$ be a minizing sequence of the Gagliardo-Nirenberg inequality of Hartree type \eqref{ieq:GN_H}.
			Without loss of generality, we can assume that $\|w_n\|_2=a$, $B_q[w_n]=\mathcal{A}(N,\alpha)$ and $A[w_n]=\mathcal{C}_{G}^{-2q}a^{2(1-q)}+o_n(1)$.
			Then we obtain that
			\begin{align*}
				f_n^{\prime}(s):=f_{w_n}^{\prime}(s)=&(A[w_n]-\mu\eta_{q}B_q[w_n])s-\eta_{p}B_p[w_n]s^{p-1}\\
				=&(\mathcal{C}_{G}^{-2q}a^{2(1-q)}+o_n(1)-\mu\eta_{q})s-\eta_{p}B_p[w_n]s^{p-1}.
			\end{align*}
			There exists a unique $s_n:=s_{w_n}=\left(\frac{A[w_n]-\mu q^{-1}B_q[w_n]}{B_p[w_n]}\right)^{\frac{1}{p-1}}>0$ such that $s_n\circ w_n\in\mathcal{P}(a)$ for $0<\mu a^{2q-2}<q\mathcal{C}_{G}(q)^{-1}$.
			It follows that
			\begin{align*}
				B_p[w_n]\lesssim A[w_n]=\mathcal{C}_{G}^{-2q}a^{2(1-q)}+o_n(1).
			\end{align*}
			As a consequence, we get
			\begin{align*}
				E[w_n]
				=(\frac{1}{2}-\frac{1}{2p})s_n^pB_p[ w_n]
				=(\frac{1}{2}-\frac{1}{2p})\left(A[w_n]-\mu q^{-1}B_q[w_n]\right)^{\frac{p}{p-1}}B_p[w_n]^{\frac{1}{p-1}}=o_n(1)
			\end{align*}
			as $\mu a^{2q-2}\rightarrow q\mathcal{C}_{G}(q)^{-1}$.
			Thus we have $\gamma(a, \mu)\leq 0$ for $\mu a^{2q-2}= q\mathcal{C}_{G}(q)^{-1}$ via the monotone property of $\gamma(a, \mu)\leq 0$ for $\mu a^{2q-2}\geq q\mathcal{C}_{G}(q)^{-1}$.
			According to the fact $E[u]=(\frac{1}{2}-\frac{1}{2p})B_p[u]\geq 0$ for $u\in\mathcal{P}(a)$, it follows that $\gamma(a, \mu)= 0$ for $\mu a^{2q-2}\geq q\mathcal{C}_{G}(q)^{-1}$.
			The proof is completed.
		\end{enumerate}
	\end{proof}
	\begin{remark}
		The above lemma is inspired by \cite[Lemma 3.3]{MR4433054} with a slightly different proof.
		Indeed, a key step in \cite[Lemma 3.3]{MR4433054} is to show that, for $q=2^{\sharp}=2+\frac{4}{N}$, it holds that
		\begin{align*}
			\sup_{u\in S(a)}\frac{\|u\|_2^2}{\|u\|_q^q}=+\infty.
		\end{align*}
		Using our method, we can give a more or less easier proof: 
		\begin{align*}
			\sup_{u\in S(a)}\frac{\|u\|_2^2}{\|u\|_q^q}\geq \sup_{u\in S(a)}\frac{\|u\|_2^2}{\|u\|_2^{q\vartheta}\|u\|_{2^*}^{q(1-\vartheta)}},
		\end{align*}
		where $\frac{\vartheta}{2}+\frac{1-\vartheta}{2^{*}}=\frac{1}{q}$ and $q(1-\vartheta)=2$. 
		The remain part is similar to our case.
	\end{remark}
	Now we can prove Theorem \ref{thm:mass_critical} for $p=2_{\alpha}^*$.
	
	\begin{proof}[Proof of Theorem \ref{thm:mass_critical} for $p=2_{\alpha}^*$]
		\begin{enumerate}
			\item[(i)] For $0<\mu a^{2q-2}<q\mathcal{C}_{G}(q)^{-1}$, by Lemma \ref{lm:PS_existence_mass_critical}, there exists a bounded Palais-Smale sequence $\{u_n\}\subset\mathcal{P}(a)$ at level $\gamma(a, \mu)$.
			By Lemma \ref{lm:energy_mass_critical}, $\gamma(a, \mu)\in (0,\frac{\alpha+2}{2N+2\alpha}\mathcal{S}_{H}^{\frac{N+\alpha}{\alpha+2}})$, and hence one of the two alternatives in Lemma \ref{lm:PS_strong_convergence} holds.
			Let us assume that alternative $(i)$ holds; then, up to a subsequence, there exists a function $u_a\in H^1(\mathbb{R}^N)\setminus\{0\}$
			such that $u_n\rightharpoonup u_a$ weakly in $ H^1(\mathbb{R}^N)$ and 
			\begin{align*}
				E[u_a]\leq \gamma(a, \mu)-\frac{\alpha+2}{2N+2\alpha}\mathcal{S}_{H}^{\frac{N+\alpha}{\alpha+2}}<0.
			\end{align*}
			This leads to a contradiction since $P[u_a]=0$ and $E[u_a]=(\frac{1}{2}-\frac{1}{2p})B_p[u_a]> 0$.
			Consequently, the alternative $(ii)$ in Lemma \ref{lm:PS_strong_convergence} holds, namely $u_n\rightarrow u_a$ strongly in $H^1(\mathbb{R}^N)$, and $u_a$ is a real valued redial normalized solution to \eqref{eq:NLS_S} for some $\lambda_a<0$, with energy $\gamma(a, \mu)$.
			\item[(ii)] For $\mu a^{2q-2}\geq q\mathcal{C}_{G}(q)^{-1}$, we shall argue by contradiction.
			Suppose that $0=\gamma(a, \mu)$ is attained by some $u_a$, then by Lemma \ref{lm:energy_mass_critical}, we obtain that $B_p[u_a]=0$.
			It is impossible since $u_a\in S(a)$.
			Therefore, $\gamma(a, \mu)$ can not be attained and hence $\left. E\right|_{S(a)}$ has no ground states.
		\end{enumerate}
	\end{proof}
	\subsubsection{The compactness of Palais-Smale sequences under the HLS subcritical leading term}
	Similar to the proof of Lemma \ref{lm:energy_mass_critical}, we can prove the following lemma:
	\begin{lemma}\label{lm:energy_mass_critical_2}
		Let $N\geq 3$, $q=2_{\alpha}^{\sharp}<p<2_{\alpha}^{*}$ and $\mu,a>0$. 
		Then we have
		\begin{enumerate}
			\item[(i)] $\gamma(a, \mu)$ is strictly decreasing with respective to $\mu$ for $0<\mu a^{2q-2}<q\mathcal{C}_{G}(q)^{-1}$ and is nonincreasing with respective to $\mu$ for $\mu a^{2q-2}\geq q\mathcal{C}_{G}(q)^{-1}$;
			\item[(ii)] It holds that $\gamma(a, \mu)=0$ for $\mu a^{2q-2}\geq q\mathcal{C}_{G}(q)^{-1}$.
		\end{enumerate}  
	\end{lemma}
	
	The proof of $p\in (2_{\alpha}^{\sharp},2_{\alpha}^*)$ in Theorem \ref{thm:mass_critical} is similar to $p=2_{\alpha}^*$, and hence we omit the details.
	Then we shall use an approximation method to give a second proof of  Theorem \ref{thm:mass_critical} with $p\in (2_{\alpha}^{\sharp},2_{\alpha}^*)$.
	In the following part of this subsection, we shall write $E_p[u]:=E[u]$, $P_p[u]=P[u]$ and $\gamma_p(a)=\gamma(a)$, where $p$ is the exponent appeared in \eqref{eq:NLS_S}.

	We first prove some properties of the critical value $\gamma_p(a)$:
	\begin{lemma}
		Let $N\geq 3$, $q=2_{\alpha}^{\sharp}<p<2_{\alpha}^{*}$ and $a>0$.
		Suppose that $0<\mu a^{2q-2}<q\mathcal{C}_{G}(q)^{-1}$ holds.
		Then we have $\limsup\limits_{p\rightarrow 2_{\alpha}^{*}} \gamma_{p}(a)\leq\gamma_{2_{\alpha}^{*}}(a)$.
	\end{lemma} 
	\begin{proof}
		Recall the definition of $\gamma_{2_{\alpha}^{*}}(a):=\inf\limits_{u\in\mathcal{P}_{2_{\alpha}^{*}}(a)} E_{2_{\alpha}^{*}}[u]$, and hence for any $\epsilon >0$, there exists a $u\in \mathcal{P}_{2_{\alpha}^{*}}(a)$ such that $E_{2_{\alpha}^{*}}[u]<\gamma_{2_{\alpha}^{*}}(a)+\epsilon$.
		Notice that 
		\begin{align*}
			E_{2_{\alpha}^{*}}[s\circ u]=\frac{s}{2}A[u]-\frac{\mu s}{2q}B_q[u]-\frac{s^{2_{\alpha}^{*}}}{22_{\alpha}^{*}}B_{2_{\alpha}^{*}}[u]\rightarrow -\infty\quad\text{as}\quad s\rightarrow+\infty,
		\end{align*}
		which implies that there exists $s_0>0$ satisfying $E_{2_{\alpha}^*}[s_0\circ u]\leq -2$.
		By the Hardy-Littlewood-Sobolev inequality \eqref{ieq:HLS} and the H\"{o}lder inequality, we obtain that, for some $\theta \in (0,1)$,
		\begin{align*}
			B_p[u] \lesssim \| u\|_{\frac{2pN}{N+\alpha}}^{2p}\leq \|u\|_2^{2p\theta}\|u\|_{2^*}^{2p(1-\theta)}\leq \|u\|_{H^{1}(\mathbb{R}^N)}^{2p}.
		\end{align*}
		On the other hand, by the Young inequality, we have
		\begin{align*}
			\vert u \vert^p \lesssim \vert u \vert^q + \vert u \vert^{2_{\alpha}^{*}}.
		\end{align*}
		Choosing a sequence $\{p_n\} \subset  (2_{\alpha}^{\sharp},2_{\alpha}^{*})$ with $\lim\limits_{n\rightarrow +\infty} p_n=2_{\alpha}^{*}$, we have
		\begin{align*}
			\vert u \vert^{p_n}\rightarrow \vert u \vert^{2_{\alpha}^{*}}\quad \text{a.e. in }\mathbb{R}^N \quad \text{and}\quad \vert u \vert^{p_n} \lesssim \vert u \vert^q+ \vert u \vert^{2_{\alpha}^{*}}  \in L^{1}(\mathbb{R}^N).
		\end{align*}
		Consequently, using the Lebesgue dominated convergence theorem, we obtain that
		\begin{align*}
			B_{p_n}[u]\rightarrow B_{2_{\alpha}^*}[u]\quad \text{as}\quad n\rightarrow +\infty.
		\end{align*}
		Now we can find that
		\begin{align*}
			\vert E_{p_n}[s\circ u]-E_{2_{\alpha}^{*}}[s\circ u] \vert=&\left\vert \frac{1}{p_n} s^{p_n\eta_{p_n}}B_{p_n}[u]-\frac{1}{2_{\alpha}^*}s^{2_{\alpha}^*}B_{2_{\alpha}^{*}}[u]\right \vert\\
			\leq & \left\vert \frac{1}{p_n} s^{p_n\eta_{p_n}}B_{p_n}[u]-\frac{1}{2_{\alpha}^*}s^{2_{\alpha}^*}B_{p_n}[u]\right \vert+\left\vert \frac{1}{2_{\alpha}^*}s^{2_{\alpha}^*}B_{p_n}[u]-\frac{1}{2_{\alpha}^*}s^{2_{\alpha}^*}B_{2_{\alpha}^{*}}[u]\right \vert\\
			\rightarrow & 0,\quad \text{as}\quad n\rightarrow +\infty.
		\end{align*}
		Without loss of generality, we can assume that for $s\in [0, s_0]$, $E_{p_n}[s\circ u]\leq -1$.
		Therefore, there exists a unique maximum point $s_n:=s_{p_n}^{\diamond}\in (0,s_0)$ of $E_{p_n}[s\circ u]$ such that $P_{p_n}[s_n\circ u]=0$.
		Choosing a any subsequence $\{s_n\}$ with $s_n\rightarrow s_{2_{\alpha}^*}$, a similar argument shows that $0=P_{p_n}[s_n\circ u]\rightarrow P_{2_{\alpha}^{*}}[s_{2_{\alpha}^*}\circ u]$, which implies that $P_{2_{\alpha}^{*}}[s_{2_{\alpha}^*}\circ u]=0$.
		The uniqueness gives that $s_{2_{\alpha}^*}=1$ since $P_{2_{\alpha}^{*}}[u]=0$.
		Therefore, we have
		\begin{align*}
			\gamma_{p_n}(a)\leq E_{p_n}[s_n\circ u]\rightarrow  E_{2_{\alpha^{*}}}[u]<\gamma_{2_{\alpha}^{*}}(a)+\epsilon,
		\end{align*}
		which implies that $\limsup\limits_{p\rightarrow 2_{\alpha}^{*}} \gamma_{p}(a)\leq\gamma_{2_{\alpha}^{*}}(a)$.
	\end{proof}
	\begin{lemma}
		Let $N\geq 3$, $q=2_{\alpha}^{\sharp}<p<2_{\alpha}^{*}$ and $a>0$.
		Suppose that $0<\mu a^{2q-2}<q\mathcal{C}_{G}(q)^{-1}$ holds.
		Then we have $\liminf\limits_{p\rightarrow 2_{\alpha}^{*}} \gamma_{p}(a)>0$.
	\end{lemma}
	\begin{proof}
		We shall argue by contradiction and suppose that $\liminf\limits_{p\rightarrow 2_{\alpha}^{*}} \gamma_{p}(a)\leq0$.
		Let $\{p_n\}\subset (2_{\alpha}^{\sharp},2_{\alpha}^*)$ be a sequence with $\lim\limits_{n\rightarrow+\infty}p_n=2_{\alpha}^{*}$ and $\lim\limits_{n\rightarrow+\infty}\gamma_{p_n}(a)=\liminf\limits_{p\rightarrow 2_{\alpha}^{*}} \gamma_{p}(a).$
		By the existence of ground states, there exists $u_n:=u_{p_n}\in \mathcal{P}_{p_n}(a)$ such that $\gamma_{n}(a):=\gamma_{p_n}(a)=E_{p_n}[u_n]$.
		As a consequence, via $P_{p_n}[u_n]=0$, we obtain that
		\begin{align*}
			o_n(1)=E_{p_n}[u_n]=&
			\left(\frac{1}{2}-\frac{1}{2p_n\eta_{p_n}}\right)A[u_n]-\left(\frac{\mu}{2q}-\frac{\mu\eta_{q}}{2p_n\eta_{p_n}}\right)B_q[u_n]\\
			\geq & \left\{\left(\frac{1}{2}-\frac{1}{2p_n\eta_{p_n}}\right)-\left(\frac{\mu}{2q}-\frac{\mu\eta_{q}}{2p_n\eta_{p_n}}\right)\mathcal{C}_{G}(q)a^{2q-2}\right\}A[u_n],
		\end{align*}
		which implies that $A[u_n]\rightarrow 0$ as $n\rightarrow +\infty$.
		Form Lemma \ref{lm:A_bounded_below_mass_critical}, we get a contradiction.
	\end{proof}
	Now we can give a new proof of Theorem \ref{thm:mass_critical} with $p=2_{\alpha}^*$.
	\begin{proof}[The second proof of Theorem \ref{thm:mass_critical}(i) with $p=2_{\alpha}^*$ ]
		Let $\{p_n\}\subset (2_{\alpha}^{\sharp},2_{\alpha}^*)$ be a sequence with $\lim\limits_{n\rightarrow+\infty}p_n=2_{\alpha}^{*}$ and $\lim\limits_{n\rightarrow+\infty}\gamma_{p_n}(a)=\limsup\limits_{p\rightarrow 2_{\alpha}^{*}} \gamma_{p}(a).$
		By the existence of ground states, there exists a $u_n:=u_{p_n}\in \mathcal{P}_{p_n}(a)$ such that $\gamma_{n}(a):=\gamma_{p_n}(a)=E_{p_n}[u_n]$.
		first, using $P_{p_n}[u_n]=0$ and $0<\mu a^{2q-2}<q\mathcal{C}_{G}(q)^{-1}$, we obtain that 
		\begin{align*}
			\gamma_{2_{\alpha}^*}(a)+1\geq E_{p_n}[u_n]=&\left(\frac{1}{2}-\frac{1}{2p_n\eta_{p_n}}\right)A[u_n]-\left(\frac{\mu}{2q}-\frac{\mu\eta_{q}}{2p_n\eta_{p_n}}\right)B_q[u_n]\\
			\geq & \left\{\left(\frac{1}{2}-\frac{1}{2p_n\eta_{p_n}}\right)-\left(\frac{\mu}{2q}-\frac{\mu\eta_{q}}{2p_n\eta_{p_n}}\right)\mathcal{C}_{G}(q)a^{2q-2}\right\}A[u_n],
		\end{align*}
		which implies that $\{u_n\}$ is bounded in $H^1(\mathbb{R}^N)$.
		Up to a subsequence, we have $u_n \rightharpoonup u_a$ weakly in $H^1(\mathbb{R}^N)$, $u_n\rightarrow u_a$ in $L^t(\mathbb{R}^N)$ for $t\in (2,2^*)$ and $u_n\rightarrow u_a$ a.e. in $\mathbb{R}^N$.
		By Lemma \ref{lm:Langrange_multiplier}, there exists $\lambda_n<0$ such that $u_n$ satisfies \eqref{eq:NLS_S} and hence we have
		\begin{align*}
			-\lambda_n a^2=-(\eta_{p_n}-1)B_{p_n}[u_n]-\mu(\eta_{q}-1) B_q[u_n]\lesssim \max\{1,\|u_n\|_{H^{1}(\mathbb{R^N})}^{22_{\alpha}^{*}}\},
		\end{align*}
		that is, $\{\lambda_n\}$ is bounded.
		Besides, it holds that 
		\begin{align*}
			\lambda_n a^2=&\frac{1-\eta_{p_n}}{\eta_{p_n}}A[u_n]+\frac{\eta_{q}-\eta_{p_n}}{\eta_{p_n}}\mu B_q[u_n]\leq \frac{1-\eta_{p_n}+(\eta_{q}-\eta_{p_n})\mu \mathcal{C}_{G}(q)a^{2q-2}}{\eta_{p_n}}A[u_n]\lesssim -C.
		\end{align*}
		So there exists $\lambda_a< 0$ such that, up to a subsequence, $\lambda_n\rightarrow \lambda_a$ as $n\rightarrow +\infty$.
		Choosing a test function $\phi \in \mathcal{C}_c^{\infty}(\mathbb{R}^N)$ and by Young inequality and Lemma \ref{lm:decay}, we obtain that
		\begin{align*}
			\left\vert\vert u_n\vert^{p_n-2}u_n\phi\right\vert\lesssim\vert u_n\vert^{2-1}\vert\phi\vert+\vert u_n\vert^{2^{*}-1}\vert\phi\vert
			\lesssim \vert x\vert^{\frac{2-N}{2}}\vert\phi\vert+\vert x\vert^{\frac{2-N}{2}(2^{*}-1)}\vert\phi\vert.
		\end{align*}
		Similarly, we have
		\begin{align*}
			\vert u_n\vert^{p_n}\lesssim\vert u_n\vert^{2}+\vert u_n\vert^{2^{*}}
			\lesssim \vert x\vert^{\frac{-N}{2}}+\vert x\vert^{\frac{-2^*N}{2}}.
		\end{align*}
		Thus, using the Lebesgue dominated convergence theorem, it holds that 
		\begin{align*}
			0=\langle \dif \mathcal{I}_{\lambda_n,\mu}[u_n], \phi\rangle \rightarrow \langle \dif \mathcal{I}_{\lambda_a,\mu}[u_a], \phi\rangle,\quad \text{as}\quad n\rightarrow +\infty.
		\end{align*}
		Consequently, $(u_a,\lambda_a)$ is a solution of \eqref{eq:NLS_S} and $P_{2_{\alpha}^*}[u_a]=0$.
		Our next goal is to prove that $u_a$ is nonvanishing.
		If not, $u_a=0$, we immediately get $B_q[u_n]=o_n(1)$. 
		Therefore, $P[u_n]=0$, \eqref{ieq:best_constant} and the Young inequality imply that
		\begin{align*}
			A[u_n]=\eta_{p_n}B_{p_n}[u_n]+o_n(1)\leq \eta_{p_n}\left(\frac{p_n-q}{2_{\alpha}^*-q}\right)^2B_{p}[u_n]+o_n(1)\leq \eta_{p_n}\left(\frac{p_n-q}{2_{\alpha}^*-q}\right)^2\left(\frac{A[u_n]}{\mathcal{S}_H}\right)^{2_{\alpha}^*}+o_n(1),
		\end{align*}
		which implies that either $A[u_n]=o_n(1)$ or 
		$	\limsup\limits_{n\rightarrow +\infty} A[u_n] \geq \mathcal{S}_H^{\frac{N+\alpha}{\alpha+2}}$.
		The first case is impossible since it disobeys Lemma \ref{lm:A_bounded_below_mass_critical}.
		As a consequence, it holds that 
		\begin{align*}
			\gamma_{2_{\alpha}^*}(a)\geq \limsup_{n\rightarrow+\infty} \gamma_n(a)=&\limsup_{n\rightarrow+\infty}\left\{\left(\frac{1}{2}-\frac{1}{2p_n\eta_{p_n}}\right)A[u_n]-\left(\frac{\mu}{2q}-\frac{\mu\eta_{q}}{2p_n\eta_{p_n}}\right)B_q[u_n]\right\}\\
			=&\limsup_{n\rightarrow+\infty}\left(\frac{1}{2}-\frac{1}{2p_n\eta_{p_n}}\right)A[u_n]\geq \frac{2+\alpha}{2N+2\alpha}\mathcal{S}_H^{\frac{N+\alpha}{\alpha+2}},
		\end{align*}
		which is impossible since Lemma \ref{lm:energy_mass_critical} holds.
		Now we have $u_a \ne 0$ and $\|u_a\|_2\leq a$.
		If $\|u_a\|_2=c<a$, we can define $\bar{u}_a:=\theta^Bu_a(\theta^A x)$.
		Then we have
		\begin{align*}
			\|\bar{u}_a \|_2^2=\theta^{2B-AN}\|u\|_2^2,\quad A[\bar{u}_a]=\theta^{2A+2B-AN}A[u_a]\quad B_s[\bar{u}_a]=\theta^{(2B-AN)s+2s\eta_s}B_s[u_a].
		\end{align*}
		Let $\|\bar{u}_a \|_2=a$ and $B_p[\bar{u}_a]=B_p[u_{a}]$, and hence 
		\begin{align*}
			\theta =\left(\frac{c}{a}\right)^{\frac{1}{p}}<1,\quad A[\bar{u}_a]=\left(\frac{c}{a}\right)^{-\frac{2(p-1)}{p}}A[u_a]\leq A[u_a].
		\end{align*}
		There exists $1\geq\bar{s}>0$ such that $P[\bar{s}\circ \bar{u}_a]=0$.
		Therefore, it holds that
		\begin{align*}
			\gamma_{p}(a)\leq& E[\bar{s}\circ \bar{u}_a]=\frac{q\eta_q-1}{2q\eta_q}\bar{s}A[ \bar{u}_a]+\frac{p-q\eta_q}{2pq}\bar{s}^{p}B_p[u]\leq \frac{q\eta_q-1}{2q\eta_q}A[ \bar{u}_a]+\frac{p-q\eta_q}{2pq}B_p[u]\\
			\leq& \liminf_{n\rightarrow +\infty}\left\{\frac{q\eta_q-1}{2q\eta_q}A[ \bar{u}_a]+\frac{p_n-q\eta_q}{2p_nq}B_{p_n}[u]\right\} \liminf_{n\rightarrow +\infty}\gamma_{n}(a)\leq \gamma_{p}(a),
		\end{align*}
		which implies that $E[\bar{s}\circ \bar{u}_a]=\gamma_{p}(a)$. 
		The remain part is clear.
	\end{proof}
	\subsection{$L^2$-supercritical perturbation}
	In this subsection, we shall always assume that $2_{\alpha}^{\sharp}<q<p= 2_{\alpha}^{\sharp}$ and $\mu,a>0$. 
	Indeed, the case of $2_{\alpha}^{\sharp}<q<p< 2_{\alpha}^{*}$ and $\mu,a>0$ is contained in \cite[Theorem 1.1]{MR3390522}.
	The main proof in this subsection is similar to the $L^2$-critical focusing perturbation.
	So we shall give some lemmas without details and focus on the differences of proof.
	\subsubsection{The existence of bounded Palais-Smale sequences}
	It is clear that
	$\mathcal{P}(a)=\mathcal{P}^{-}(a)\cup\mathcal{P}^{0}(a)$.
	If $\mathcal{P}^{0}(a)\ne \emptyset,$ then we assume that $0\neq u\in\mathcal{P}^{0}(a)$, which means that
	\begin{align*}
		A[u]=\mu \eta_{q} B_q[u]+B_p[u],\quad (p-1)B_p[u]=-\mu\eta_{q}(q\eta_{q}-1)B_q[u].
	\end{align*}
	Then we have $u \equiv 0$, which is impossible.
	Indeed, we have proved the following lemma:
	\begin{lemma}\label{lm:P_manifold_mass_supercritical}
		 Let $N\geq 3$, $\alpha \in (0,N)$, $2_{\alpha}^{\sharp}<q<p=2_{\alpha}^{*}$ and $\mu,a>0$.
		 Then we have $\mathcal{P}(a)=\mathcal{P}^{-}(a)$.
	\end{lemma}
	Notice that $f_u^{\prime}(s_u)=0$ for some $s_u>0$. 
	If $s\in [0,s_u),$ $f_u^{\prime}(s)>0$ and if $s\in (s_u,+\infty),$ $f_u^{\prime}(s)<0$. Then $s_u$ is the unique maximum point for $f_u$ on $[0,+\infty)$ and $s_u\circ u\in\mathcal{P}^{-}(a)$. Following the argument in \cite[Lemma 5.3]{SOAVE20206941}, we can prove that the map $u\in S(a) \rightarrow s_u\in\mathbb{R}$ is of class $\mathscr{C}^1$.
	
	\begin{lemma}
		Let $N\geq 3$, $\alpha \in (0,N)$, $2_{\alpha}^{\sharp}<q<p=2_{\alpha}^{*}$ and $\mu,a>0$. Then there exists $\delta>0$ such that $A[u]\geq \delta$ for all $u\in \mathcal{P}^{-}(a)$.
	\end{lemma}
	\begin{proof}
		If $u\in \mathcal{P}^{-}(a)$, we have  
		\begin{align*}
			A[u]=&\mu\eta_{q} B_q[u]+B_p[u]
			\leq  \mu \eta_{q} \mathcal{C}_{G}(q)^{2q}a^{2q-2q\eta_{q}}A[u]^{q\eta_{q}}+\mathcal{A}(N,\alpha)S_{H}^{-p}A[u]^p,
		\end{align*}
		which implies that there exists $\delta>0$ such that $A[u]\geq \delta$ for all $u\in \mathcal{P}^{-}(a)$.
	\end{proof}
	Similar to Lemma \ref{lm:inf_Pohozaev_pm}, we can prove that 
	\begin{lemma}
		Let $N\geq 3$, $\alpha \in (0,N)$, $2_{\alpha}^{\sharp}<q<p=2_{\alpha}^{*}$ and $\mu,a>0$.
		 Then we have
		$$\gamma(a)=\inf\limits_{\mathcal{P}(a)} E[u]=\inf\limits_{\mathcal{P}^{-}(a)} E[u]=\inf\limits_{\mathcal{P}_r^{-}(a)} E[u]>0.$$
	\end{lemma}
	Again, we can define $I: S(a)\rightarrow \mathbb{R}$ by $ I[u]=E[S_u\circ u]$ and hence we have $\langle \dif I[u],\psi\rangle=\langle \dif E[S_u\circ u], S_u\circ \psi\rangle$ for any $u\in S(a)$ and $\psi\in T_uS(a)$. Using the method given in Lemma \ref{lm:PS_sequence_existence_mass_subcritical}, we have
	\begin{lemma}\label{lm:PS_existence_mass_supercritical}
		Let $N\geq 3$, $\alpha \in (0,N)$, $2_{\alpha}^{\sharp}<q<p=2_{\alpha}^{*}$ and $\mu,a>0$.
		Then there exists a Palais-Smale sequence $\{u_n\}\subset \mathcal{P}^{-}(a)$ for $E$ restricted to $S_r(a)$ at level $\gamma(a)$.
	\end{lemma}
	\subsubsection{The compactness of Palais-Smale sequences under the HLS critical leading term}
	Now we shall investigate the properties of $\gamma(a)$. 
	We shall use suitable scaling and cut-off techniques to get an estimate about the energy. 
	\begin{lemma}\label{lm:energy_mass_supercritical}
		Let $N\geq 4$ or $N=3$ with $0<\alpha<\min\{q-1,3\}$. Set $p=2_{\alpha}^{*}$, $q\in (2_{\alpha}^{\sharp}, 2_{\alpha}^*)$ and $\mu,a >0$, then $0<\gamma(a)<\frac{\alpha+2}{2N+2\alpha}S^{\frac{N+\alpha}{\alpha+2}}_{H}$.  
	\end{lemma}
	\begin{proof}
		As before, let $\xi\in\mathscr{C}_{c}^{\infty}$ be a radial cut-off function with $\xi\equiv 1$ in $B_1$,   $\operatorname{supp}(\xi)\subset B_2$ and $0\leq \xi(x)\leq 1,$ $\forall x\in\mathbb{R}^N$.
		
		Define $V_{\epsilon}(x)=(a^{-1}\| U_{\epsilon}\|_2)^{\frac{N-2}{2}}U_{\epsilon}(a^{-1}\| U_{\epsilon}\|_2x)$. Then we have
		\begin{align*}
			\|V_{\epsilon}\|_2=a,\quad A[V_{\epsilon}]=A[U_{\epsilon}],\quad B_s[V_{\epsilon}]=(a^{-1}\| U_{\epsilon}\|_2)^{2s\eta_s-2s}B_s[U_{\epsilon}].
		\end{align*}
		Therefore, we obtain that 
		\begin{align*}
			\gamma(a)=\inf\limits_{\mathcal{P}^{-}(a)}E\leq E[s_{V_{\epsilon}}^{\diamond}\circ V_{\epsilon}]=\sup_{s\geq 0}E[s\circ V_{\epsilon}],\quad \forall \epsilon>0.
		\end{align*}
		Notice that $B_p[U_{\epsilon}]=B_p[V_{\epsilon}]$. As a consequence, by Lemma \ref{lm:estimate_bubble} and Lemma \ref{lm:estimate_bubble_nonlocal}, we have
		\begin{align*}
			E[s\circ V_{\epsilon}]
			=&\frac{1}{2}sA[V_{\epsilon}]-\frac{\mu}{2q}s^{q\eta_{q}}B_q[V_{\epsilon}]-\frac{1}{2p}s^pB_p[V_{\epsilon}]\\
			=&\frac{1}{2}sA[U_{\epsilon}]-\frac{1}{2p}s^pB_p[U_{\epsilon}]-\frac{\mu}{2q}s^{q\eta_{q}}(a^{-1}\| U_{\epsilon}\|_2)^{2s\eta_s-2s}B_q[U_{\epsilon}]\\
			\leq&\frac{1}{2}s\left(\mathcal{S}^{\frac{N}{2}}+O(\epsilon^{N-2})\right)-\frac{1}{2p}s^{p}\left(\mathcal{A}(N,\alpha)^{\frac{N}{2}}C(N,\alpha)^{\frac{N}{2}}\mathcal{S}_{H}^{\frac{N+\alpha}{2}}-O(\epsilon^{\frac{N+\alpha}{2}})\right)\\
			&-\frac{\mu}{2q}s^{q\eta_{q}}a^{2s-2s\eta_s}\times \begin{cases}
				O(1),&\text{if }N\geq 5,\\
				O(\vert\log\epsilon\vert^{2q(\eta_{q}-1)}),&\text{if }N=4,\\
				O(\epsilon^{q(1-\eta_{q})}),&\text{if }N=3.
			\end{cases}
		\end{align*}
		We claim that $s_{V_{\epsilon}}^{\diamond}$ is bounded independent of $\epsilon$.
		If not, $s_{V_{\epsilon}}^{\diamond}\rightarrow 0$ or $s_{V_{\epsilon}}^{\diamond}\rightarrow 0$ as $\epsilon \rightarrow 0^{+}$, then we find $\gamma(a)\leq 0,$ which is impossible.
		Consequently, we obtain that
		\begin{align*}
			\sup_{s\geq 0}E[s\circ V_{\epsilon}]<&\sup_{s\geq 0}\left(\frac{1}{2}s\mathcal{S}^{\frac{N}{2}}-\frac{1}{2p}s^p\mathcal{A}(N,\alpha)^{\frac{N}{2}}C(N,\alpha)^{\frac{N}{2}}\mathcal{S}_{H}^{\frac{N+\alpha}{2}}\right)
			=\frac{\alpha+2}{2N+2\alpha}\mathcal{S}^{\frac{N+\alpha}{\alpha+2}}_{H}.
		\end{align*}
	\end{proof}
	Now we can prove Theorem \ref{thm:mass_supercritical} similar to Theorem \ref{thm:mass_subcritical} and we omit the proof.
	\section{$L^2$-supercritical leading term with defocusing perturbation}
	\subsection{HLS subcritical leading term}
	In this subsection, we shall always assume that $p \in (2_{\alpha}^{\sharp},2_{\alpha}^{*})$, $q\in (2_{\alpha}, 2_{\alpha}^{\sharp}]$ and $\mu\leq0$.
	Under these assumptions, it is easy to see that $f^{\prime\prime}_u(s)\leq 0$ and hence $\mathcal{P}(a)=\mathcal{P}^{0}(a)\cup \mathcal{P}^{-}(a)$. 
	If $\mathcal{P}^{0}(a)\ne \emptyset$, we suppose that there exists a nontrivial $u \in \mathcal{P}(a)$ such that $f^{\prime\prime}_u(1)= 0$, that is,
	\begin{align*}
		0\leq\eta_{p}(p\eta_{p}-1)B_p[u]=\mu B_q(1-q\eta_{q})B_q[u]\leq 0,
	\end{align*}
	which implies $u \equiv 0$, a contradiction. Hence we prove
	\begin{lemma}
		Let $N\geq 3$, $\alpha\in (0,N)$, $p \in (2_{\alpha}^{\sharp},2_{\alpha}^{*})$, $q\in (2_{\alpha}, 2_{\alpha}^{\sharp}]$ and $\mu\leq0$. 
		Then we have $\mathcal{P}(a)= \mathcal{P}^{-}(a)$.
	\end{lemma}
	Now we find that there exists a unique local maximum point $s_u\in (0,+\infty)$ such  that $f^{\prime}_u(s_u)=0$.
	Moreover, $f_u^{\prime}(s)>0$ if $s\in (0,s_u^{\diamond})$ and $f_u^{\prime}(s)<0$ if $s\in (s_u,+\infty)$.
	It is not hard to prove that the map $u \in S(a)\mapsto s_u\in \mathbb{R}$ is of class $\mathscr{C}^1$.
	\begin{lemma}\label{lm:A_bounded_defousing}
		Let $N\geq 3$, $\alpha\in (0,N)$, $p \in (2_{\alpha}^{\sharp},2_{\alpha}^{*})$, $q\in (2_{\alpha}, 2_{\alpha}^{\sharp}]$, $a>0$ and $\mu\leq0$. 
		Then we have
		\begin{enumerate}
			\item[(i)] $E[u]>0$ for all $u\in \mathcal{P}(a)$;
			\item[(ii)] there exists $\delta>0$ such that $A[u]\geq \delta$ for all $u \in \mathcal{P}(a)$.
		\end{enumerate}
	\end{lemma}
	\begin{proof}
		\begin{enumerate}
			\item[(i)] If $u \in \mathcal{P}(a)$, we have 
			\begin{align*}
				E[u]=\frac{1}{2}A[u]-\frac{1}{2p}B_p[u]-\frac{\mu}{2q}B_q[u]=\frac{p\eta_{p}-1}{2p}B_p[u]+\frac{q\eta_{q}-1}{2q}\mu B_q[u]\geq 0,
			\end{align*}
			and $E[u]=0$ if and only if $A[u]=0$, which contradicts with $u\in S(a)$.
			\item[(ii)] It follows from $P[u]=0$ that 
			\begin{align*}
				A[u]=\eta_{p}B_p[u]+\mu \eta_{q}B_q[u]\leq \eta_{p}B_p[u]\leq \eta_{p} \mathcal{C}_{G}(p) a^{2p(1-\eta_{p})}A[u]^{p\eta_{p}},
			\end{align*}
			which implies that 
			\begin{align*}
				A[u]\geq \left(\eta_{p} \mathcal{C}_{G}(p) a^{2p(1-\eta_{p})}\right)^{-\frac{1}{p\eta_{p}-1}}.
			\end{align*}
		\end{enumerate}
	\end{proof}
	Now we can prove the main theorem in this subsection:
	\begin{proof}[Proof of Theorem \ref{thm:defousing}]
		Similar to Lemma \ref{lm:PS_sequence_existence_mass_subcritical}, using Lemma \ref{lm:coercive} and Lemma \ref{lm:A_bounded_defousing}, we can obtain a bounded Palais-Smale sequence $\{u_n\} \subset \mathcal{P}(a)$.
		The remain part is to prove that, up to a subsequence, $u_n$ converges to a nontrivial limit $u_a$.
		The proof is similar to Lemma \ref{lm:nontirvial_weak_limit} and Lemma \ref{lm:convergence_sub}. 
		The only point we need to notice is that 
		\begin{align*}
			\lambda_a a^2=&A[u_a]-B_p[u_a]-\mu B_q[u_a]=\frac{\eta_{p}-1}{\eta_{p}}A[u_a]+\frac{\eta_{q}-\eta_{p}}{\eta_{p}}\mu B_q[u_a]\\
			\leq& \frac{\eta_{p}-1}{\eta_{p}}A[u_a]+\frac{\eta_{q}-\eta_{p}}{\eta_{p}}\mu \eta_{q} \mathcal{C}_{G}(q) a^{2q(1-\eta_{q})}A[u_a]^{q\eta_{q}}\\
			\leq &\frac{A[u_a]^{q\eta_{q}}}{\eta_{p}}\left(\frac{\eta_{p}-1}{\left(\eta_{p} \mathcal{C}_{G}(p) a^{2p(1-\eta_{p})}\right)^{\frac{1-q\eta_{q}}{p\eta_{p}-1}}}+(\eta_{q}-\eta_{p})\mu \eta_{q} \mathcal{C}_{G}(q) a^{2q(1-\eta_{q})}\right)<0,
		\end{align*}
		where we have used \eqref{assumptions:basic}  and Lemma \ref{lm:A_bounded_defousing}.
	\end{proof}
	\subsection{HLS critical leading term}
	In this subsection, we shall always assume that $p =2_{\alpha}^{*}$, $q\in [2, 2_{\alpha}^{*})$, $a>0$ and $\mu <0$.
	We first recall a Liouville-type theorem, see \cite[Theorem 2.1]{MR2846170}:
	\begin{theorem}\label{thm:L_type}
		Assume that $N \geq 3$ and the nonlinearity $f:(0, \infty) \mapsto(0, \infty)$ is continuous and satisfies
		$$
		\liminf _{s \rightarrow 0} s^{-\frac{N}{N-2}} f(s)>0 .
		$$
		Then the differential inequality $-\Delta u \geq f(u)$ has no positive solution in any exterior domain of $\mathbb{R}^N$.
	\end{theorem}
	Then we can prove the main theorem in this subsection:
	\begin{proof}[Proof of Theorem \ref{thm:defousing}(ii)]
		Let $u \in H^1(\mathbb{R}^N)$ be a nontrivial positive solution to \eqref{eq:NLS_S}. 
		By Lemma \ref{lm:Langrange_multiplier}, we obtain that $\lambda>0$ and $P[u]=0$.
		To prove $(i)$, we get that 
		\begin{align*}
			E[u]\geq \frac{1}{2}A[u]-\frac{1}{2p}B_p[u]\geq \frac{1}{2}A[u]-\frac{1}{2p}\mathcal{S}_{H}^{-p}A[u]^{p}\geq \frac{\alpha+2}{2(N+\alpha)}\mathcal{S}^{\frac{N+\alpha}{\alpha+2}}_{H}.
		\end{align*}
		Now we consider $(ii)$. Note that for $q\in (2_{\alpha},2_{\alpha}^{*})$, via the weak Young inequality, we obtain that 
		\begin{align*}
			\left(\int_{\mathbb{R}^N}\left\vert I_{\alpha} * \vert u \vert^q \right\vert^{\frac{2N}{N-\alpha}}\right)^{\frac{1}{2}}\lesssim \| u \|_{\frac{2Nq}{N+\alpha}}^q,
		\end{align*}
		which implies that $\lim\limits_{x\rightarrow+\infty}(I_{\alpha} * \vert u \vert^q)(x)=0$ since $\frac{2Nq}{N+\alpha}\in (2,2^*)$.
		Besides, $u\in S(a)$ gives that $\lim\limits_{x\rightarrow +\infty}u(x)=0$.
		Consequently, for $q\geq 2$, there exists a number $R_0>0$ such that 
		\begin{align*}
			(I_{\alpha} * \vert u \vert^q)(x)\vert u(x)\vert^{q-2}\leq -\frac{\lambda}{2\mu}\quad \text{for}\quad\vert x \vert >R_0.
		\end{align*}
		Therefore, we obtain that for $q\geq 2$ and $\vert x \vert >R_0$, 
		\begin{align*}
			-\Delta u(x)=\lambda u(x)+ (I_{\alpha} * \vert u \vert^p)(x)\vert u(x) \vert^{p-2}u(x)+\mu (I_{\alpha} * \vert u \vert^q)(x)\vert u(x) \vert^{q-2}u(x)\geq \frac{\lambda}{2} u(x).
		\end{align*}
		Applying Theorem \ref{thm:L_type} with $f(s)=\frac{\lambda}{2}s$, we obtain a contradiction.
	\end{proof}
	\section{The asymptotic behavior of ground states as $\mu \rightarrow 0^{+}$}
	In this section, we shall prove Theorem \ref{thm:asymptotic_mu_0}.
	\begin{proof}[Proof of Theorem \ref{thm:asymptotic_mu_0}]
		\begin{enumerate}
			\item[(i)] If $q\in (2_{\alpha},2_{\alpha}^{\sharp})$, the ground state $u_{a,\mu}$ is described by Theorem \ref{thm:minizing}. 
			Note that $k_1\rightarrow 0$ as $\mu\rightarrow 0^+$, which gives that $A[u_{a,\mu}]\rightarrow 0$ as $\mu \rightarrow 0^+$.
			Following the proof in Lemma \ref{lm:coercive}, we find that 
			\begin{align*}
				0>\gamma(a,\mu)=E[u_{a,\mu}]\geq \frac{p-1}{2p}A[u_{a,\mu}]-\frac{\mu}{2q}(1-\frac{q\eta_q}{p})\mathcal{C}_{G}(q)a^{2q(1-\eta_{q})}A[u_{a,\mu}]^{q\eta_{q}}\rightarrow 0
			\end{align*}
			as $\mu \rightarrow 0^+$.
			\item[(ii)] Then we consider $q\in [2_{\alpha}^{\sharp},2_{\alpha}^{*})$. 
			For a fixed $a>0$, let $\bar{\mu}>0$ satisfy \eqref{assumptions:basic} for $q=2_{\alpha}^{\sharp}$ and $\bar{\mu}>0$ is arbitrary for $q\in (2_{\alpha}^{\sharp},2_{\alpha}^{*})$.
			We first claim that $\{u_{a,\mu}:\mu\in (0,\bar{\mu})\}$ is bounded in $H^1(\mathbb{R}^N)$.
			If $q=2_{\alpha}^{\sharp}$, by Lemma \ref{lm:energy_mass_critical} and $P[u_{a,\mu}]=0$, it follows that 
			\begin{align*}
				\gamma(a,0)\geq& \gamma(a,\mu)=E[u_{a,\mu}]=\frac{p-1}{2p}A[u_{a,\mu}]-\frac{p-q\eta_{q}}{2pq}\mu B_q[u_{a,\mu}]\\
				\geq& \frac{p-1}{2p}\left(1-\eta_q\mu \mathcal{C}_{G}(q)a^{2q(1-\eta_{q})}\right)A[u_{a,\mu}].
			\end{align*}
			Similarly, if $q=2_{\alpha}^{\sharp}$, by the Hardy-Littewood-Sobolev inequality \eqref{ieq:HLS}, it holds that 
			\begin{align*}
				\gamma(a,0)\geq& \gamma(a,\mu)=E[u_{a,\mu}]=\frac{p-1}{2p}B_p[u_{a,\mu}]+\frac{q\eta_{q}-1}{2q}\mu B_q[u_{a,\mu}],
			\end{align*}
			which implies that $\{B_p[u_{a,\mu}]\}$ and $\{\mu B_q[u_{a,\mu}]\}$ are bounded.
			Notice that $P[u_{a,\mu}]=0$, that is, $A[u_{a,\mu}]=B_p[u_{a,\mu}]+\mu B_q[u_{a,\mu}]$, which implies that $\{A[u_{a,\mu}]\}$ is bounded.
			The formula of Lagrange multiplies gives that
			\begin{align*}
				\lambda_{a,\mu}=-\mu(1-\eta_{q})a^{-2}B_q[u_{a,\mu}]\rightarrow 0\quad \text{as}\quad \mu\rightarrow 0^{+}.
			\end{align*}
			Up to a subsequence, we assume that $A[u_{a,\mu}]\rightarrow \ell\geq 0$ and $\lambda_{a,\mu}\rightarrow 0$ as $\mu \rightarrow 0^+$.
			Indeed, $\ell \ne 0$. 
			If not, $\ell =0$, we have that $B_p[u_{a,\mu}]=o_{\mu}(1)$ by \eqref{ieq:best_constant}.
			Consequently, we have $\gamma(a,\mu)=E[u_{a,\mu}]=o_{\mu}(1)$.
			However, the monotonicity implies that $\gamma(a,\mu)\geq \gamma(a,\bar{\mu})>0$ for $0<\mu<\bar{\mu}$, a contradiction.
			Again, using $P[u_{a,\mu}]=0$, we deduce that
			\begin{align*}
				B_p[u_{a,\mu}]=A[u_{a,\mu}]-\mu q\eta_{q}B_q[u_{a,\mu}]= l+o_{\mu}(1),
			\end{align*}
			and by \eqref{ieq:best_constant}, we obtain that 
			\begin{align*}
				l\leq \mathcal{S}_{H}^{-p}l^p
			\end{align*}
			On the other hand, we also have 
			\begin{align*}
				\frac{N+\alpha}{2(\alpha+2)}l+o_{\mu}(1)=E[u_{a,\mu}]=\frac{p-1}{2p}A[u_{a,\mu}]-\frac{p-q\eta_{q}}{2pq}\mu B_q[u_{a,\mu}]\leq \gamma(a,0)=\frac{N+\alpha}{2(\alpha+2)}\mathcal{S}^{\frac{N+\alpha}{\alpha+2}}_{H},
			\end{align*}
			which implies that $l=\mathcal{S}^{\frac{N+\alpha}{\alpha+2}}_{H}$, and the thesis follows.
		\end{enumerate}
	\end{proof}
	
	\textbf{Acknowledgment}
	This study was supported partially by the National Natural Science Foundation of China (Grant NO. 11871242, 12031015, 12271508). We would like to express our gratitude to Prof. Louis Jeanjean about some discussions related to the min-max description in Lemma 3.7.

\end{document}